\documentclass[11pt,reqno]{amsart}

\usepackage[T1]{fontenc}
\usepackage{lmodern}
\usepackage{microtype}
\usepackage[a4paper,margin=31mm]{geometry}
\usepackage{amsmath,amssymb,mathtools,mathrsfs}
\usepackage{booktabs,tabularx,array}
\usepackage[shortlabels]{enumitem}
\usepackage{xcolor}
\usepackage{hyperref}

\newtheorem{theorem}{Theorem}[section]
\newtheorem{proposition}[theorem]{Proposition}
\newtheorem{lemma}[theorem]{Lemma}
\newtheorem{corollary}[theorem]{Corollary}
\theoremstyle{definition}
\newtheorem{example}[theorem]{Example}

\theoremstyle{remark}
\newtheorem{remark}[theorem]{Remark}

\newcommand{\cO}{\mathcal O}
\newcommand{\cF}{\mathcal F}

\newcommand{\End}{\operatorname{End}}
\newcommand{\cEnd}{\operatorname{{\mathcal{E}}nd}}
\newcommand{\Spec}{\operatorname{Spec}}

\newcommand{\rk}{\operatorname{rk}}

\newcommand{\Ext}{\operatorname{Ext}}
\newcommand{\Sym}{\operatorname{Sym}}

\newcommand{\bP}{\mathbb P}

\newcommand{\bQ}{\mathbb Q}
\newcommand{\bZ}{\mathbb Z}

\newcommand{\bC}{\mathbb C}

\newcommand{\cA}{\mathscr A}
\newcommand{\cB}{\mathscr B}
\newcommand{\cE}{\mathscr E}
\newcommand{\cH}{\mathscr H}
\newcommand{\cX}{\mathscr X}

\DeclareMathOperator{\SL}{SL}
\DeclareMathOperator{\Sp}{Sp}
\DeclareMathOperator{\GL}{GL}
\DeclareMathOperator{\GSp}{GSp}
\DeclareMathOperator{\Gr}{Gr}
\DeclareMathOperator{\SO}{SO}
\DeclareMathOperator{\PU}{PU}
\DeclareMathOperator{\PGL}{PGL}

\DeclareMathOperator{\Hol}{Hol}

\begin{document}
	
	\markboth {\rm }{}

	\title{Stability of vector bundles and reduction mod $p$}
	\author{Adrian Langer}

	\email{alan@mimuw.edu.pl}
	\address{University of Warsaw, Institute of Mathematics,
		ul.\ Banacha 2, 02-097 Warszawa, Poland}
	
	\date{\today}

\begin{abstract}
Let $E$ be a vector bundle on a normal complex projective variety. 
We prove that slope stability of $\Sym^2E$ and Zariski density of the
reductions of $E$ modulo $p$ that are not strongly geometrically stable imply that the bundle $E$ is projectively unitary flat.
This generalizes Shepherd--Barron's result from rank $2$ to arbitrary rank. We also show that 
if $E$ has rank $\le 3$ we can replace stability of  $\Sym^2E$ by stability of $E$
but we cannot do so if $E$ has rank $4$. We also generalize these results to reflexive sheaves on
some mildly singular normal complex projective varieties.
\end{abstract}

\maketitle

\section*{Introduction}

Let $X$ be a smooth complex projective variety of dimension $n\ge 2$ and let $H$ be an ample line bundle on $X$.
Let $E$ be a vector bundle on $X$. Throughout the paper, stability, polystability and semistability are understood in the slope
sense with respect to the indicated polarization.

By the Kobayashi--Hitchin correspondence, $H$-polystability of $E$ 
is equivalent to the existence of a Hermite--Einstein metric. This result, known as the Donaldson--Uhlenbeck--Yau theorem (see \cite{ssb-Donaldson1985,ssb-UhlenbeckYau1986}), connects algebraic geometry with differential geometry. In this paper we are interested in arithmetic properties of stable vector bundles, i.e., in behaviour of stability under reductions modulo prime numbers. Here our knowledge is very limited and essentially the only known result concerns the rank two case. Namely, Shepherd--Barron  in \cite[Theorem 5]{ShepherdBarron1998} proved a striking result relating behaviour of reductions of a rank two stable vector bundle to its flatness.
More precisely, he proved that if $E$ is semistable and a dense set of reductions of $E$ is not strongly semistable then $E$ is not only stable but also irreducible projectively unitary flat, i.e., the projective bundle $\bP(E)$ arises from an irreducible projective unitary representation $\pi _1^{\mathrm {top}}(X^{\mathrm{an}} )\to \PU (2)$.

In \cite{LangerPositivitySemistability} the author proposed a natural conjectural generalization of this result to higher rank. Below we show that the original form of \cite[Conjecture 6.1]{LangerPositivitySemistability} fails but a corrected version of the conjecture holds in all ranks. The basic idea behind this failure is that arithmetic properties of a stable vector bundle are controlled not only by flatness but also by the holonomy group. 

To describe the  precise results, let us relax the condition on $E$ and assume that it is a reflexive sheaf.
Let us fix an integral
finitely generated $\mathbb Z$-subalgebra $R\subset\mathbb C$, put
$S=\Spec R$, and choose a model
\[
(\mathscr X,\mathscr H,\mathscr E)\longrightarrow S
\]
of $(X,H,E)$, with $\mathscr X/S$ smooth projective, $\mathscr H$ relatively ample, and  $\mathscr E$ an $S$-flat family of reflexive sheaves.
If $s\in S$ is a closed point then the residue field $\kappa (s)$ is finite so we can consider the absolute Frobenius morphism $F_s$ on the fiber $\cX_s$ over $s$.

Assume that $E$ is $H$-semistable and consider the set $\Sigma ^{\mathrm{nss}}(\mathscr E/S)$ of those closed points  $s\in S$ for which $\mathscr E_s$ is not strongly $\cH_s$-semistable (i.e., some Frobenius pullback of $\cE_s$ is not $\cH_s$-semistable).
We say that \emph{ $\Sigma ^{\mathrm{nss}}(E)$ is dense} if $\Sigma ^{\mathrm{nss}}(\cE/S)$ is dense in $S$ for some model
$(\mathscr X,\mathscr H,\mathscr E)/S$ (or equivalently for all models $(\mathscr X,\mathscr H,\mathscr E)/S$).

Similarly, consider the set $\Sigma ^{\mathrm{ns}}(\cE/S)$ of those closed points  $s\in S$ for which $\mathscr E_s$ is not strongly geometrically $\cH_s$-stable (i.e., some Frobenius pullback of $\cE_s$ is not geometrically $\cH_s$-stable). As above we say that \emph{ $\Sigma ^{\mathrm{ns}}(E)$ is dense} if $\Sigma ^{\mathrm{ns}}(\cE/S)$ is dense in $S$ for some model
$(\mathscr X,\mathscr H,\mathscr E)/S$ (or equivalently for all models $(\mathscr X,\mathscr H,\mathscr E)/S$).

The main aim of the paper is to prove the following theorem:

\begin{theorem}	\label{main1}
	Let \(E\) be an \(H\)-stable reflexive sheaf of rank \(r\geq2\) on a smooth polarized complex projective variety of dimension at least two. Suppose that \(\Sigma^{\mathrm{ns}}(E)\) is dense and that either \(r\leq3\) or \((\operatorname{Sym}^{2}E)^{**}\) is \(H\)-stable. Then \(E\) is an irreducible projectively unitary flat vector bundle.	
\end{theorem}

We  extend Theorem~\ref{main1} to singular varieties.
For vector bundles, the corresponding conclusions hold on
arbitrary normal complex projective varieties
(see Theorems~\ref{ssb:cor-higher-dimensional-normal}
and~\ref{cor:rank-at-most-three}).
For reflexive sheaves on normal projective varieties with quotient
singularities in codimension two, we obtain local freeness on the
regular locus and vanishing of the discriminant. If the variety
is klt, the restriction to the regular locus is also irreducible
projectively unitary flat
(see Theorems~\ref{ssb:cor-higher-dimensional-reflexive}
and~\ref{cor:rank-at-most-three-normal}).

We also prove a higher-rank analogue of the differential-form statement in \cite[Theorem~5, (2)]{ShepherdBarron1998}, relating the arithmetic hypotheses above to the existence of line
subsheaves of exterior powers of the cotangent bundle with
nonnegative or positive degree. We state this complementary
result separately in Theorem~\ref{thm:higher-rank-SB}.
\medskip

Let us recall that a vector bundle $E$ is called \emph{pseudoeffective} if $\cO_{\bP (E)}(1)$ is pseudoeffective.
A related consequence of dense failure of strong geometric stability is the following positivity statement.

\begin{theorem}	\label{main-pseudoeffective}
	Let \(E\) be an \(H\)-stable vector bundle of rank \(r\geq2\) on a smooth complex projective variety \(X\) of dimension \(d\geq2\). Let \(\pi_i:\mathbb P_X(\wedge^iE)\to X \) denote the canonical projections.
	If	\(\Sigma^{\mathrm{ns}} (E)\) is dense then \( \mathcal O_{\mathbb P(\wedge^iE)}(r) \otimes\pi_i^*(\det E)^{-i}\) is pseudoeffective for some \(1\leq i<r.\)
	In particular, if \(\det E\) is pseudoeffective, then \(\bigwedge^iE\) is pseudoeffective for some \(1\le i<r\).
\end{theorem} 

Assume for simplicity that $E$ is locally free and $\det E$ is
trivial. If all symmetric powers of $E$ are $H$-stable and
$E$ is pseudoeffective, then
\cite[Theorem~1.1]{HoringPeternell2019} implies that $E$ is
numerically flat. Since $E$ is $H$-stable, it is therefore
unitary flat. However, under the assumptions of
Theorem~\ref{main-pseudoeffective}, even if $\det E$ is trivial,
$E$ need not be pseudoeffective
(see Proposition~\ref{prop:rank-six-counterexample}).
Although this direct argument does not work in general, we use
a similar but more elaborate strategy to prove Theorem~\ref{main1}.

More precisely, first we consider the surface case and we prove a criterion relating projective unitary flatness of $E$ to pseudoeffectivity or nefness of some standard line bundles on various fibrations associated to homogeneous spaces of the holonomy group of $E$ (see Theorem \ref{ssb:thm-holonomy-orbit-positivity}). This is the most technical part and one of the main novelties  of the paper. It also allows us to strengthen \cite[Theorem~1.1]{HoringPeternell2019} by requiring stability
only of the second reflexive symmetric power (see Corollary~\ref{ssb:cor-horing-peternell}).
Stability of the second symmetric power gives a dichotomy: either the projective holonomy is finite, in which case $E$ is already projectively unitary flat, or the connected derived holonomy is $\SL$ or $\Sp$ (see Lemmas~\ref{ssb:lem-holonomy-square-dichotomy} and~\ref{ssb:lem-finite-projective-holonomy}). In the latter case, density of $\Sigma^{\mathrm{ns}}(E)$ implies that at least one of the above line bundles is pseudoeffective (see Proposition \ref{ssb:prop-grassmann-pseudoeffectivity}). This result is analogous to Theorem \ref{main-pseudoeffective}, but it is not implied by this theorem.
In ranks two and three, a Frobenius pullback which is not
geometrically stable yields a line subsheaf of sufficiently
large slope, after possibly dualizing. This allows the same
positivity strategy to apply without a symmetric-power stability
assumption, still assuming that $\Sigma^{\mathrm{ns}}(E)$ is dense.

The relation of pseudoeffectivity and nefness in Theorem  \ref{ssb:thm-holonomy-orbit-positivity} follows from Lemma~\ref{ssb:lem-pseff-nef} and uses ideas coming from earlier papers. In particular, a similar strategy was used in the proofs of \cite[Theorem 1.1]{HoringPeternell2019} and \cite[Theorem 6.1]{Druel2018}.
One of the main points in this strategy is the use of ideas underlying Mumford's example  \cite[Chapter I, Example 10.6]{Hartshorne1970} of a nef line bundle which is ample on every proper subvariety. This example was  revisited in various cases in \cite[Theorem 6.1]{Subramanian1989}, \cite[Theorem 3.1]{Mehta-Subramanian1995}, \cite[Theorem 6.1]{Biswas-Parameswaran-Subramanian2006} and \cite[Proposition 1.3]{HoringPeternell2019}. Then as in \cite{Druel2018} and \cite{HoringPeternell2019} we use the restricted base locus of some standard line bundles on some  fibrations associated to $E$ (in the above papers this was used only for $\cO _{\bP(E)}(1)$, or a certain line bundle on the full flag variety in \cite{Biswas-Parameswaran-Subramanian2006}).

A geometric application is a flatness criterion for orthogonal
bundles. For an $H$-stable vector bundle $E$ of rank at least three
with trivial determinant and full special orthogonal holonomy
on a smooth complex projective surface,
Corollary~\ref{cor:orthogonal-flatness-criterion} identifies unitary
flatness with pseudoeffectivity, or equivalently nefness, of the
relative anticanonical bundle of the associated quadric bundle
of isotropic quotients. This detects flatness even though
pseudoeffectivity of $E$ itself is automatic.

\medskip

The following examples show the necessity of the assumptions in
Theorem~\ref{main1} and the limitations of the conclusion of
Theorem~\ref{main-pseudoeffective}.

\begin{theorem}\label{thm:rank-four-counterexample}
	For each of the following three cases, there exist a smooth polarized complex projective surface $(X,H)$
	and a vector bundle $E$ of rank $r$, defined over
	$\mathbb Q$, such that \( \int_X\Delta(E)>0\) and
	\begin{enumerate}
		\item $r=3$ and $E$ is strictly $H$-semistable, or
		\item $r=4$, $E$ is $H$-stable and pseudoeffective, with
		$\det E\simeq\mathcal O_X$, but
		$\operatorname{Sym}^mE$ is not $H$-stable for any $m\geq2$, or	
		\item $r=6$, $E$ is $H$-stable, with
		$\det E\simeq\mathcal O_X$, and it is not pseudoeffective,
		whereas $\bigwedge^3E$ is pseudoeffective.
	\end{enumerate}
	After spreading out, for every sufficiently large prime
	$p\equiv2\pmod5$, the reductions $E_{p}$ are $H_p$-semistable,
	and geometrically $H_p$-stable if $r=4,6$, whereas
	$F_{X_p}^*E_{p}$ is not $H_p$-semistable.
	In particular, each $\Sigma^{\mathrm{nss}}(E)$ contains a set
	of primes of Dirichlet density $1/4$, and none of the bundles
	$E$ is projectively flat.
\end{theorem}

Item~(1) shows that the low-rank statement in
Theorem~\ref{main1} cannot be extended from stable to semistable
bundles. This contrasts with rank two, where, after a suitable
choice of model and shrinking its base, all reductions of a
strictly semistable vector bundle are strongly semistable.
The construction is given in
Proposition~\ref{prop:semistable-rank-three-syzygy}.

Item~(2) shows that the rank bound in Theorem~\ref{main1} is
sharp: even dense failure of strong semistability does not force
a stable rank four bundle to be projectively flat.
It also shows that even for vector bundles stability of
$\operatorname{Sym}^{2}E$ in Theorem~\ref{main1} cannot be
replaced by stability of $E$.
The example is constructed in
Proposition~\ref{prop:rank-four-counterexample}.
Its pseudoeffectivity follows from Lemma~\ref{lem:symmetric-powers-rank-four-example}.

Item~(3) shows that the conclusion of
Theorem~\ref{main-pseudoeffective} cannot be strengthened to
pseudoeffectivity of $E$ itself, even for stable bundles with
trivial determinant (see Proposition~\ref{prop:rank-six-counterexample}).

All three examples are constructed on the product of the Fermat
plane quintic $C$ and $\mathbb P^1$, using a rank two
vector bundle on $C$ constructed by Brenner in \cite{Brenner-Miyaoka}.
The rank four and rank six examples are tensor products of its
pullback with pullbacks of stable bundles of ranks two and three,
respectively, from $\mathbb P^2$.

\medskip

Let $E$ be a  reflexive sheaf on $(X,H)$. If $\Sigma^{\mathrm{ns}}(E)$
is not dense, then, for every model
$(\mathscr X,\mathscr H,\mathscr E)\to S$, we may replace $S$
by a nonempty open subset so that $\mathscr E_s$ is strongly
geometrically $\cH_s$-stable for every closed point $s\in S$.
We then say that \emph{almost all reductions of $E$ are strongly
	geometrically stable}.

The local-freeness and flatness conclusions have the following surprising
consequences for reductions.

\begin{corollary}
	\begin{enumerate}
		\item Let $(X,H)$ be a smooth polarized complex projective variety
		of dimension at least two, and let $E$ be a reflexive sheaf of
		rank $r\geq2$ such that $(\operatorname{Sym}^{2}E)^{**}$ is $H$-stable.
		If $E$ is not locally free, then almost all reductions of $E$ are strongly
		geometrically stable.
		
		\item Let $(X,H)$ be a normal polarized complex projective variety
		of dimension at least two with trivial étale fundamental group.
		If $E$ is a vector bundle of rank $r\geq2$ such that
		$\operatorname{Sym}^2E$ is $H$-stable, then
		almost all reductions of $E$ are strongly
		geometrically stable.
	\end{enumerate}
\end{corollary}

For item~(2), note that the image of a projective unitary
representation of $\pi_1^{\mathrm{top}}(X^{\mathrm{an}})$ is a finitely
generated linear group and hence is residually finite by Malcev's
theorem. Triviality of the \'etale fundamental group therefore
forces this image to be trivial, contradicting irreducibility
in rank at least two.

In ranks two and three, both assertions remain valid with
$H$-stability of $E$ in place of stability of its second symmetric
power (see Theorem~\ref{main1} and Theorem~\ref{cor:rank-at-most-three}).

\medskip

Finally, Section~\ref{sec:double-covers} gives counterexamples to
\cite[Conjecture~7.7]{LangerPositivitySemistability}.
We construct non-torsion line bundles on elliptic curves and
geometrically stable rank two bundles with trivial determinant
on genus-two curves whose reductions are \'etale trivializable
at every sufficiently large prime, although the complex bundles
are not \'etale trivializable. Thus even \'etale triviality of
all sufficiently large reductions does not force finite
unitary monodromy in characteristic zero.
\medskip

\emph{A brief history of the paper.}
The project began many years ago, when the author combined his
earlier proof of Bogomolov's inequality for Higgs sheaves
\cite[Theorem~4.2]{LangerHiggsOld} with the analysis of rank~$3$
systems of Hodge bundles in \cite{LangerSimpsonRankThree}
to prove the rank~$3$ case of Theorem~\ref{main1}.
The project then remained dormant owing to a lack of progress
in higher rank.

The author's first attempt to obtain a proof of the (unpublished) rank~$3$
result from ChatGPT Astra was unsuccessful. After he supplied
the main idea, however, a few further prompts enabled him to
reconstruct his original proof. Subsequently, when asked to prove
the author's conjecture
\cite[Conjecture~6.1]{LangerPositivitySemistability},
ChatGPT produced a counterexample already in rank~$4$
(see Theorem~\ref{thm:rank-four-counterexample}(2)).
This led to the research presented in this paper.
Eventually, the original proofs in ranks~$2$ and~$3$ were
rewritten using the methods developed for the higher-rank case.

\medskip

The paper is organized as follows. Section~1 collects preliminary
results. Sections~2 and~3 discuss unitary flatness and algebraic
holonomy. In Sections~4 and~5 we develop the positivity results
for normalized exterior powers and homogeneous bundles associated
with holonomy groups. Section~6 treats the low-rank cases.
The arithmetic surface theorem is proved in Section~7 and extended
to vector bundles and reflexive sheaves on normal projective
varieties in Sections~8 and~9, respectively. Section~10 contains
the differential-form analogue of Shepherd--Barron's theorem.
Section~11 provides counterexamples illustrating the necessity
of the assumptions in our main results, whereas Section~12
constructs bundles with infinite monodromy whose reductions
are \'etale trivializable at all sufficiently large primes.
Finally, Appendix A establishes restriction theorems with
prescribed basepoints and Appendix B contains some push-forward formulas for flag varieties. 

\subsection*{Notation}

All varieties are geometrically integral. A torsion free (reflexive) sheaf on $X$ stands for a coherent
torsion free (respectively, reflexive) $\cO_X$-module.

If $X$ is a normal variety and $E$ is a coherent $\cO_X$-module of rank $r>0$ then we set
\[
\Sym^{[m]}E:=(\Sym^mE)^{**} ,\quad{\bigwedge}^{[i]}E:=({\bigwedge}^iE)^{**},\quad \hbox{and} \quad \det E:=({\bigwedge} ^rE)^{**}.
\]
If $f: Y\to X$ is a morphism from a normal variety $Y$ then we set 
$f^{[*]}E:=(f^*E)^{**}.$
If $X$ is a normal variety of characteristic $p$ then we write 
\(
F_{X}^{[e]}E:=((F_{X}^e)^*E)^{**},
\)
where $F_X$ is the absolute Frobenius morphism.
For a torsion-free sheaf $E$
of rank $r$ on a smooth projective variety we set
\begin{equation*}
	\Delta(E):=2r c_2(E)-(r-1)c_1(E)^2.
\end{equation*} 
We also use the same formula for reflexive sheaves on normal projective varieties with quotient singularities in codimension $2$. In this case Chern classes denote orbifold Chern classes on the large open subset on which $X$ has quotient singularities.

Let \(X\) be an \(S\)-scheme and let \(T\to S\) be a morphism. We write
$X_T:=X\times_S T$
and, if $p_X\colon X_T\to X$ is the projection and \( E\) is
an \(\mathcal O_X\)-module, we set
$ E_T:=p_X^* E.$

\section{Preliminaries}

\subsection{Slope semistability and stability}

Let $X$ be a normal projective variety defined over a field $k$. Let us set $d=\dim X$ and 
let $H$ be an ample (or, more generally, nef) line bundle on $X$.
For a torsion-free sheaf $E$ of rank $r>0$, we define the \emph{$H$-slope} of $E$ by
\[
\mu_H(E):=\frac{c_1(E)\cdot H^{d-1}}{r}.
\]

We say that $E$ is \emph{(slope) $H$-semistable} if for every nonzero $F\subset E$ we have $\mu_H(F) \le \mu_H(E)$.
We say that $E$ is \emph{(slope) $H$-stable} if for every $F\subset E$ of rank $0<\rk F<\rk E$
we have $\mu_H(F) < \mu_H(E)$. $E$ is called \emph{geometrically $H$-stable} if $E_{\bar k}$ is $H_{\bar k}$-stable on $X_{\bar k}$. We say that $E$ is \emph{(slope) $H$-polystable} if it is a direct sum of slope $H$-stable sheaves of the same $H$-slope.

\subsection{Grothendieck's boundedness theorem}

Let $X$ be a scheme of finite  type over some noetherian scheme $S$.
Let $\mathscr B$ be a set of coherent sheaves on the fibers of $X\to S$. We say that $\mathscr B$ is \emph{bounded} 
if there exists a scheme $T$ of finite type over $S$ and a coherent sheaf $\cF$ on $X_T/T$ such that every sheaf $\cE\in \mathscr B$ is isomorphic to the restriction of $\cF$ to one of the fibers of $X_T\to T$. 

For a coherent sheaf \(Q\) of dimension \(d\) on a noetherian scheme,
let \(T_{<d}(Q)\subset Q\)
denote the maximal coherent subsheaf whose support has dimension strictly
smaller than \(d\), and set
\[
Q_{(d)}:=Q/T_{<d}(Q).
\]
Then \(Q_{(d)}\) is pure of dimension \(d\).

The following is the form of Grothendieck's boundedness lemma that we
will use; see \cite[Lemme~2.6]{Grothendieck-FGA-IV}.

\begin{theorem}
	\label{thm:Grothendieck-boundedness}
	Let	\(	f\colon\mathcal X\longrightarrow S	\)
	be a projective morphism of noetherian schemes, let
	\(\mathcal O_{\mathcal X}(1)\) be relatively very ample, and let
	\(\mathcal F\) be a coherent sheaf on \(\mathcal X\). Fix an integer
	\(d\geq1\), and let \(\mathscr B\) be a collection of coherent quotients \(	\mathcal F_{\bar s}\twoheadrightarrow Q	\)
	on geometric fibers of \(f\), such that
	\(
	\dim\operatorname{Supp}Q=d
	\)
	for every \(Q\in\mathscr B\). Write the Hilbert polynomial of \(Q\) as
	\[
	P_Q(n):=\chi\bigl(Q(n)\bigr)
	=a_d(Q)n^d+a_{d-1}(Q)n^{d-1}+\cdots+a_0(Q).
	\]
	If the numbers \(a_d(Q)\) are bounded above as \(Q\) varies in
	\(\mathscr B\), then the numbers \(a_{d-1}(Q)\) are bounded below.
	If, in addition, the numbers \(a_{d-1}(Q)\) are bounded above, then
	the sheaves
	\(
	\bigl\{Q_{(d)}\mid Q\in\mathscr B\bigr\}
	\)
	form a bounded set over \(S\).
\end{theorem}

\subsection{Cartier's descent}

In this subsection, let $X$ be a smooth projective variety over a perfect
field of characteristic $p>0$, and fix an ample divisor $H$ on $X$.
All stability conditions are taken with respect to $H$.

Let us recall that if  $E$ is a coherent sheaf on $X$ then $F_X^*E$ admits
a canonical integrable connection $\nabla _{\sf can}$.
This connection controls descent of subsheaves of $F_X^*E$ to  subsheaves of $E$ (see \cite[Theorem 5.1]{Katz1970}) and it
can be used to show the following result (see \cite[Corollary 2.4]{Langer2004}
for (1); (2) can be proven in the same way).

\begin{proposition}\label{prop:Cartier-descent}
	\begin{enumerate}
		\item Let $E$ be a semistable torsion free sheaf  such that $F_X^*E$ is not semistable.
		Let $0=E_0\subset E_1\subset \dots  \subset E_m=F_X^*E$ be the Harder--Narasimhan filtration. 
		Then for $0<i<m$ the $\cO_X$-homomorphisms $E_i\to  (F_X^*E/E_i) \otimes \Omega _X$ 
		induced by $\nabla _{\sf can}$ are nontrivial. 
		\item Let $E$ be a stable torsion free sheaf  such that $F_X^*E$ is semistable but not stable.
		Let $0=E_0\subset E_1\subset \dots  \subset E_m=F_X^*E$ be a Jordan--H\"older filtration. 
		Then for $0<i<m$ the $\cO_X$-homomorphisms $E_i\to  (F_X^*E/E_i) \otimes \Omega _X$ 
		induced by $\nabla _{\sf can}$ are nontrivial. 
	\end{enumerate}
\end{proposition}

\subsection{Openness of stability on surfaces}

The following lemma is a special case of the well-known openness of stability in the ample cone (see \cite[Proposition 4.17]{GKPT2019} for a more general, very related statement).  For the convenience 
of the reader we give a simple proof in the case that we need it. 

\begin{lemma}\label{ssb:lem-nearby-polarizations}
	Let $H$ be a nef Cartier divisor on a smooth projective surface $X$,
	and let $E$ be an $H$-stable torsion-free sheaf of rank $r$.
	Let $B$ be a Cartier divisor such that $bH-B$ is ample for some
	positive integer $b$. Then $P=nH-B$ is ample and $E$ is
	$P$-stable for all sufficiently large integers $n$.
\end{lemma}

\begin{proof}
	Set $A=bH-B$. Since $A$ is ample, choose an embedding
	$E\hookrightarrow\mathcal O_X(mA)^{\oplus a}$ with $m>0$.
	Since $H$ is nef, the divisor
	$nH-B=(n-b)H+A$ is ample for every $n\geq b$.
	
	For a saturated subsheaf $F\subset E$ of rank $0<j<r$, set
	\(
	\Theta_F=r\,c_1(F)-j\,c_1(E).
	\)
	A nonzero component of the determinant map shows that
	$jmA-c_1(F)$ is effective. Consequently,
	\[
	\Theta_F\cdot A
	\leq j(rmA-c_1(E))\cdot A.
	\]
	The right-hand side is bounded above by a constant $K\geq0$
	independent of $F$. Since $E$ is $H$-stable and $H$ is Cartier,
	we have $\Theta_F\cdot H\leq-1$. Thus, for $n>b+K$,
	\[
	\Theta_F\cdot(nH-B)
	=(n-b)\Theta_F\cdot H+\Theta_F\cdot A
	\leq -(n-b)+K<0.
	\]
	This proves $P$-stability of $E$ for such $n$.
\end{proof}

\begin{lemma}\label{ssb:lem-birational-transfer}
	Let $f:Y\to X$ be a projective birational morphism from a smooth
	projective surface to a normal projective surface over an
	algebraically closed field. Let $H$ be an ample Cartier divisor
	on $X$, and let $D$ be an effective $f$-exceptional divisor.
	Fix a positive integer $N$ such that
	\(	P=Nf^*H-D	\)
	is ample.
	
	Let $W$ be a vector bundle on $X$, and let $F\subset W$ be a
	saturated subsheaf with $0<\rk F<\rk W$. Denote by
	$\widetilde F$ the saturation of the image of $f^*F$ in $f^*W$.
	Then
	\[
	\mu_P(\widetilde F)-\mu_P(f^*W)
	\geq
	N\bigl(\mu_H(F)-\mu_H(W)\bigr).
	\]
	Consequently, if $W$ is not $H$-stable, respectively not
	$H$-semistable, then $f^*W$ is not $P$-stable, respectively
	not $P$-semistable. In positive characteristic, the same
	implications hold for strong stability and strong semistability.
\end{lemma}

\begin{proof}
	Put $h=f^*H$ and $j=\rk F$.
	Since $f$ is an isomorphism over a big open subset of $X$,
	the projection formula gives
	\[
	c_1(\widetilde F)\cdot h=c_1(F)\cdot H,
	\qquad
	c_1(f^*W)\cdot h=c_1(W)\cdot H.
	\]
	
	Let $C\subset Y$ be an integral curve contracted by $f$.
	The restriction $(f^*W)|_C$ is trivial. Since $\widetilde F$
	is saturated, its determinant map
	\(
	\det\widetilde F\to\bigwedge\nolimits^j f^*W
	\)
	is nonzero at the generic point of $C$ after restriction to $C$.
	Pulling back to the normalization of $C$ therefore gives a
	nonzero morphism from a line bundle to a trivial bundle. Hence
	\[
	c_1(\widetilde F)\cdot C\leq0,
	\qquad
	c_1(f^*W)\cdot C=0.
	\]
	As $D$ is effective and exceptional, it follows that
	\[
	c_1(\widetilde F)\cdot D\leq0,
	\qquad
	c_1(f^*W)\cdot D=0.
	\]
	Consequently,
	\[
		\mu_P(\widetilde F)-\mu_P(f^*W)
		=
		N\bigl(\mu_H(F)-\mu_H(W)\bigr)
		-\frac{c_1(\widetilde F)\cdot D}{j}
		\geq
		N\bigl(\mu_H(F)-\mu_H(W)\bigr).
	\]
	This proves the assertions about stability and semistability.
	
	In characteristic $p>0$, apply the same argument to
	$(F_X^e)^*W$ for every $e\geq0$, using
	\[
	f^*(F_X^e)^*W\simeq(F_Y^e)^*f^*W.
	\]
\end{proof}

\subsection{Highest weights, Cartan powers, and Schur modules}\label{subsection:Cartan}

Let $G$ be a connected complex reductive group. Fix a maximal
torus $T$ and a Borel subgroup $B$ corresponding to the negative
roots. For a dominant weight $\lambda\in X^*(T)$, denote by
$L(\lambda)$ the irreducible $G$-module of highest weight $\lambda$.
Its $m$th \emph{Cartan power} is $L(m\lambda)$, the unique
irreducible summand of $\Sym^mL(\lambda)$ containing the $m$th
power of a highest-weight vector. For a principal $G$-bundle
$\mathcal P$ and the associated vector bundle $E=\mathcal P\times^G L(\lambda)$, the associated
Cartan-power bundle is $\mathcal P\times^G L(m\lambda)$.

For a partition $\mu$ of length at most $r:=\rk E$, we write
$\mathbb S_\mu E$ for the bundle associated to the irreducible
polynomial $\GL(r)$-module of highest weight $\mu$ 
(see \cite[Chapter~2]{Weyman-Syzygies}). Thus
\[
\mathbb S_{(m)}E=\Sym^mE,\qquad
\mathbb S_{(1^i)}E={\bigwedge}^iE.
\]
The rectangular Schur module $\mathbb S_{(m^i)}V$ is the $m$th
Cartan power of $\bigwedge^iV$ for $\GL(V)$, and for $\rk E=2$,
\[
\mathbb S_{(a,b)}E
=\Sym^{a-b}E\otimes(\det E)^{\otimes b}.
\]
For $G=\Sp(V)$, the fundamental module $L(\omega_i)$ is instead
the primitive part of $\bigwedge^iV$; its $m$th Cartan power is
$L(m\omega_i)$.

Put $\mathcal L(\lambda)=G\times^B\mathbb C_\lambda$, with
$(gb,z)\sim(g,\lambda(b)z)$. The Borel--Weil--Bott theorem gives 
\[
H^0(G/B,\mathcal L(\lambda))\simeq L(\lambda),
\qquad
H^q(G/B,\mathcal L(\lambda))=0\quad \text{ for } q>0,
\]
for every dominant weight $\lambda$.
The same statements hold for
$\mathcal L_P(\lambda)=G\times^P\mathbb C_\lambda$ on $G/P$
whenever $P\supseteq B$ and $\lambda$ extends to a character
of $P$ (see, e.g., \cite[Part~II, Sections~2, 4 and~5]{ssb-Jantzen2003}).

\section{Unitary and projectively unitary flat bundles}

Throughout this section, varieties and morphisms are algebraic over
$\mathbb C$, and vector bundles and coherent sheaves are algebraic,
unless explicitly called holomorphic or analytic. We write
$X^{\mathrm{an}}$, $E^{\mathrm{an}}$ and $f^{\mathrm{an}}$ for
analytifications. Fundamental groups and local systems are taken in the
classical topology of the indicated analytic spaces.

A \emph{unitary flat structure} on a vector bundle $E$ on $X$ is a
complex local system $\mathbb V$ on $X^{\mathrm{an}}$ with unitary
monodromy, together with a holomorphic isomorphism
\[
\mathbb V\otimes_{\mathbb C}\mathcal O_{X^{\mathrm{an}}}
\simeq E^{\mathrm{an}}.
\]
Equivalently, $E^{\mathrm{an}}$ admits holomorphic local frames with
locally constant unitary transition matrices. A rank $r$ vector bundle $E$ is
\emph{projectively unitary flat} if its analytic projectivization is
induced by a representation
\[
\rho:\pi_1^{\mathrm{top}}(X^{\mathrm{an}})
\longrightarrow\operatorname{PU}(r).
\]
Thus $\mathbb P(E)^{\mathrm{an}}$ is holomorphically isomorphic to
$(\widetilde{X^{\mathrm{an}}}\times\mathbb P^{r-1})/
\pi_1^{\mathrm{top}}(X^{\mathrm{an}})$, where
$\widetilde{X^{\mathrm{an}}}$ is the universal cover and the action on
the second factor is given by $\rho$. The structure is
\emph{irreducible} if the image of $\rho$ preserves no nonempty proper
projective linear subspace. These notions concern analytic flat
structures; algebraicity of a connection is established separately
whenever it is needed below.

\subsection{Holomorphic sections, descent and extension}

We first record a consequence of degeneration of the Hodge--de Rham
spectral sequence with unitary coefficients.

\begin{lemma}\label{ssb:cor-unitary-horizontal}
	Let $X$ be a smooth complex projective variety. Every holomorphic
	section of a unitary flat holomorphic vector bundle on
	$X^{\mathrm{an}}$ is horizontal. Moreover, every holomorphic
	morphism between such bundles is horizontal.
\end{lemma}

\begin{proof}
	Let $(\mathcal V,\nabla)$ be such a holomorphic bundle, and put
	$\mathbb V=\ker\nabla$. By
	\cite[Theorem~7.1(a)]{ssb-Timmerscheidt1987}, the analytic spectral
	sequence
	\[
	E_1^{p,q}=H^q(X^{\mathrm{an}},
	\Omega_{X^{\mathrm{an}}}^p\otimes\mathcal V)
	\Longrightarrow H^{p+q}(X^{\mathrm{an}},\mathbb V)
	\]
	degenerates at $E_1$. In particular,
	\[
	d_1^{0,0}=H^0(\nabla):H^0(X^{\mathrm{an}},\mathcal V)
	\longrightarrow H^0(X^{\mathrm{an}},
	\Omega_{X^{\mathrm{an}}}^1\otimes\mathcal V)
	\]
	vanishes. The assertion about morphisms follows by applying this
	to the induced unitary flat holomorphic Hom bundle.
\end{proof}

\begin{lemma}\label{ssb:lem-projective-flat-algebra}
	A vector bundle $E$ on a normal complex variety $X$ is projectively
	unitary flat if and only if $\cEnd(E)^{\mathrm{an}}$
	admits a unitary flat structure for which multiplication and unit
	are horizontal.
\end{lemma}

\begin{proof}
	The forward implication follows from the adjoint representation
	of $\operatorname{PU}(r)$ on $\End (\mathbb C^r)$.
	Conversely, a flat structure as in the statement has monodromy in
	\[
	\operatorname{Aut}_{\mathbb C\text{-alg}}
	\bigl(\End (\mathbb C^r)\bigr)
	\simeq\operatorname{PGL}_r(\mathbb C).
	\]
	Its closure is compact, since the structure is unitary and the
	algebra automorphism group is closed in
	$\operatorname{GL}(\End (\mathbb C^r))$.
	Every compact subgroup of $\operatorname{PGL}_r(\mathbb C)$ is
	conjugate to a subgroup of $\operatorname{PU}(r)$. The resulting
	holomorphic projective bundle is $\mathbb P(E)^{\mathrm{an}}$,
	because the algebra-frame bundle of $\cEnd(E)^{\mathrm{an}}$
	is the projective-frame bundle of $E^{\mathrm{an}}$.
\end{proof}

\begin{lemma}\label{ssb:lem-unitary-descent}
	Let $f:Y\to X$ be a projective resolution of a normal complex
	algebraic variety, and let $E$ be an algebraic vector bundle on $X$.
	Every unitary flat structure on $(f^*E)^{\mathrm{an}}$ descends to
	a unitary flat structure on $E^{\mathrm{an}}$.
\end{lemma}

\begin{proof}
	Put $g=f^{\mathrm{an}}$. Fix $x\in X^{\mathrm{an}}$ and an analytic
	open neighbourhood $W$ on which $E^{\mathrm{an}}$ is holomorphically
	trivial. The analytic fiber $K=g^{-1}(x)$ is compact and connected.
	A holomorphic frame of $E^{\mathrm{an}}|_W$ pulls back to a frame on
	$K$. For every irreducible component $Z$ of the reduced algebraic
	fiber $f^{-1}(x)$, choose a projective resolution
	$\widetilde Z\to Z$. The pullback of this frame to the
	analytification of the smooth projective variety $\widetilde Z$
	is horizontal by Lemma~\ref{ssb:cor-unitary-horizontal}.
	Since these resolutions are proper and surjective, local constancy
	of the frame descends to each component of $K$. The resulting
	frames agree at intersections of components, being restrictions
	of the same frame on $K$. Thus they trivialize the local system on
	the underlying topological space of $K$.
	
	By triangulation of analytic sets, $K$ has a sufficiently small
	analytic open neighbourhood $N\subset g^{-1}(W)$ which deformation
	retracts onto $K$. The local system is trivial on $N$. Properness
	of $g$ allows us to shrink $W$ so that $g^{-1}(W)\subset N$.
	Express a unitary parallel frame on $g^{-1}(W)$ in the pullback of
	the chosen holomorphic frame of $E^{\mathrm{an}}|_W$. Its matrix
	and inverse are holomorphic and descend by
	\(
	g_*\mathcal O_{Y^{\mathrm{an}}}
	=\mathcal O_{X^{\mathrm{an}}}.
	\)
	They give a holomorphic frame on $W$. The transition matrices of
	the frames obtained in this way are locally constant and unitary,
	as can be checked after pullback by $g$. They define the required
	unitary flat structure on $E^{\mathrm{an}}$.
\end{proof}

\begin{lemma}\label{ssb:lem-unitary-extension}
	Let $X$ be a normal complex algebraic variety and $E$ an algebraic
	vector bundle on $X$. Every unitary flat structure on
	$(E|_{X_{\mathrm{reg}}})^{\mathrm{an}}$ extends uniquely to a unitary
	flat structure on $E^{\mathrm{an}}$.
\end{lemma}

\begin{proof}
	Let $h$ be a parallel Hermitian metric on 
	$(E|_{X_{\mathrm{reg}}})^{\mathrm{an}}$. We first prove local
	boundedness of $h$ and $h^{-1}$ in a holomorphic frame of
	$E^{\mathrm{an}}$. For a local holomorphic section $s$ of
	$E^{\mathrm{an}}$ or its dual, its logarithmic squared norm is
	plurisubharmonic on $(X_{\mathrm{reg}})^{\mathrm{an}}$.
	Choose a finite holomorphic local projection $\pi:W\to B$, where
	$W\subset X^{\mathrm{an}}$ is an analytic open neighbourhood and
	$B$ is a polydisc. Outside the analytic set
	$Z=\pi(W\cap(X_{\mathrm{sing}})^{\mathrm{an}})$, the maximum of
	this function over the finite fibers of $\pi$ is plurisubharmonic.
	This is immediate away from the branch locus and extends across
	it by local boundedness there. Since $Z$ has codimension at least
	two, the Hartogs extension theorem for plurisubharmonic functions
	gives local boundedness above also near $Z$. Applying this to a
	holomorphic frame and its dual gives, on a smaller neighbourhood,
	\begin{equation}\label{ssb:eq-bounded-unitary-metric}
		C^{-1}I\leq h\leq CI.
	\end{equation}
	
	Take an algebraic projective log resolution $f:Y\to X$ which is
	an isomorphism over $X_{\mathrm{reg}}$, and put $g=f^{\mathrm{an}}$.
	The pullback metric satisfies the same local bounds in the
	pullback holomorphic frames of
	$(f^*E)^{\mathrm{an}}=g^*E^{\mathrm{an}}$.
	Near a smooth point of an exceptional component in $Y^{\mathrm{an}}$,
	let $t$ be a transverse holomorphic coordinate and let $v$ be a
	multivalued parallel eigenvector for the unitary local monodromy,
	with eigenvalue $\exp(2\pi i\alpha)$, $0\leq\alpha<1$. Normalize
	its flat norm to be one. On a transverse punctured disc, the
	section $t^{-\alpha}v$ is single-valued holomorphic and has growth
	$O(|t|^{-\alpha})$ in a holomorphic frame of $(f^*E)^{\mathrm{an}}$
	by \eqref{ssb:eq-bounded-unitary-metric}. Its Laurent expansion
	therefore has no negative terms, so it extends across the puncture.
	If $\alpha>0$, then $v$ tends to zero in that frame, contradicting
	its unit flat norm and the upper bound for $h$. Thus every
	exceptional meridian has trivial monodromy.
	
	The local system consequently extends across the analytification
	of the simple normal crossings exceptional divisor. Its parallel
	frames and their inverses are bounded in holomorphic frames of
	$(f^*E)^{\mathrm{an}}$, again by
	\eqref{ssb:eq-bounded-unitary-metric}. The removable singularity
	theorem for holomorphic functions therefore identifies its
	holomorphic bundle with $(f^*E)^{\mathrm{an}}$. This gives a
	unitary flat structure on all of $(f^*E)^{\mathrm{an}}$, which
	descends to $E^{\mathrm{an}}$ by
	Lemma~\ref{ssb:lem-unitary-descent}. The extension is unique because
	the holomorphic identifications between its local parallel frames
	are determined on the dense regular locus.
\end{proof}

In Lemmas~\ref{ssb:lem-unitary-descent} and
\ref{ssb:lem-unitary-extension}, parallel holomorphic tensor morphisms
descend and extend as well. Together with
Lemma~\ref{ssb:lem-projective-flat-algebra}, this applies in particular
to projective unitary flat structures via the multiplication and unit
on $\cEnd(E)^{\mathrm{an}}$.

\subsection{Characterizations on smooth projective varieties}

A vector bundle $E$ on a projective variety $X$ is
\emph{universally semistable} if, for every morphism $f:Y\to X$ from
a smooth projective variety, $f^*E$ is slope semistable with respect
to every ample polarization on $Y$.

The following characterization is a consequence of the
Donaldson--Uhlenbeck--Yau theorem; see
\cite[Chapter~IV, Theorem~4.1]{Nakayama2004}. For its positive
characteristic analogue, see \cite[Theorem~C]{ssb-FulgerLanger2022}.

\begin{theorem}\label{thm:FUY}
	Let $X$ be a smooth complex projective variety of dimension
	$d\geq2$, and let $E$ be a reflexive sheaf of rank $r$ on $X$.
	The following conditions are equivalent:
	\begin{enumerate}
		\item $E$ is locally free and
		$\mathcal O_{\mathbb P(E)}(r)\otimes\pi^*(\det E)^{-1}$ is nef,
		where $\pi:\mathbb P(E)\to X$ parametrizes one-dimensional
		quotients.
		\item $E$ is locally free and $\cEnd (E)$ is nef.
		\item $E$ is $H$-semistable and
		$\Delta(E) H^{d-2}=0$ for some ample line bundle $H$.
		\item $E$ is $H$-semistable and
		$\Delta(E) H^{d-2}=0$ for every ample line bundle $H$.
		\item $E$ is locally free and universally semistable.
		\item $E$ has a filtration by subbundles
		$0=E_0\subset E_1\subset\cdots\subset E_m=E$ whose quotients
		$E^i=E_i/E_{i-1}$ are projectively unitary flat and have the
		same slope with respect to every ample polarization on $X$.
	\end{enumerate}
\end{theorem}

\begin{corollary}\label{cor:DUY}
	Let $X$ be a smooth complex projective variety of dimension
	$d\geq2$, and let $E$ be a reflexive sheaf on $X$.
	The following conditions are equivalent, both in the polystable
	case and in the stable case indicated in parentheses:
	\begin{enumerate}
		\item $E$ is $H$-polystable (respectively, $H$-stable) and
		$\Delta(E) H^{d-2}=0$ for some ample line bundle $H$.
		\item $E$ is $H$-polystable (respectively, $H$-stable) and
		$\Delta(E) H^{d-2}=0$ for every ample line bundle $H$.
		\item $E$ is locally free and projectively unitary flat
		(respectively, irreducible projectively unitary flat).
		\item $E^{\mathrm{an}}$ is locally free and admits a
		Hermite--Einstein metric, with respect to a K\"ahler form
		representing an ample class, whose Chern connection has vanishing
		trace-free curvature (respectively, is also irreducible).
	\end{enumerate}
\end{corollary}

\subsection{Reflexive sheaves on singular projective varieties}

For reflexive sheaves on normal varieties with quotient singularities
in codimension two, Chern classes are understood as rational orbifold
Chern classes. The notation $\int_X\Delta(E)=0$ means vanishing of the
entire numerical discriminant form: if $d=\dim X$, then
\[
\Delta(E) L_1\cdots L_{d-2}=0
\]
for all Cartier divisors $L_1,\ldots,L_{d-2}$. For vector bundles,
these Chern classes agree with the usual ones.

\begin{theorem}\label{general-local-freeness}
	Let $X$ be a normal complex projective variety of dimension
	$d\geq2$ with quotient singularities in codimension two.
	Let $E$ be a reflexive sheaf on $X$ that is $H$-semistable and
	satisfies $\Delta(E) H^{d-2}=0$ for some ample line bundle
	$H$. Then $E|_{X_{\mathrm{reg}}}$ is locally free,
	$\cEnd (E)|_{X_{\mathrm{reg}}}$ is nef, and
	$\int_X\Delta(E)=0$.
\end{theorem}

\begin{proof}
	The reflexive sheaf $B=\cEnd (E)$ is $H$-semistable and
	satisfies $c_1(B)=0$ and $\operatorname{ch}_2(B)=-\Delta(E)$.
	Applying \cite[Corollary~7.14]{Langer-Simpson-singular} to $B$
	shows that $B|_{X_{\mathrm{reg}}}$ is a numerically flat vector 
	bundle and that $\int_X\Delta(E)=0$. By
	\cite[Theorem~4.4]{Auslander-Goldman1960}, local freeness of this
	endomorphism sheaf implies local freeness of $E|_{X_{\mathrm{reg}}}$.
\end{proof}

A normal quasi-projective variety $X$ is called a \emph{klt space} if
there is an effective $\mathbb Q$-Weil divisor $B$ such that $(X,B)$ is
klt; see \cite[Definition~2.1]{Greb-Kebekus-Peternell2022}.
For polystable reflexive sheaves on projective klt spaces, the
nonabelian Hodge correspondence gives the following stronger result.
It extends \cite[Theorem~1.20]{Greb-Kebekus-Peternell-Duke2016} and
strengthens the projective-flatness conclusion of
\cite[Proposition~1.6]{Greb-Kebekus-Peternell2022} under the
polystability assumption.

\begin{theorem}\label{klt-local-freeness}
	Let $X$ be a normal complex projective klt space of dimension
	$d\geq2$, and let $E$ be a reflexive sheaf of rank $r$ on $X$.
	Suppose that $E$ is $H$-polystable and
	$\Delta(E) H^{d-2}=0$ for some ample line bundle $H$.
	Then $E|_U$, where $U=X_{\mathrm{reg}}$, is locally free and
	projectively unitary flat.
	Moreover, there is a finite surjective Galois morphism
	$\gamma:\bar X\to X$ from a normal projective variety, \'etale in
	codimension one, such that, putting $W=\gamma^{-1}(U)$, the
	algebraic projective bundle
	$\mathbb P((\gamma|_W)^*(E|_U))$ extends to an algebraic
	Brauer--Severi bundle $P\to\bar X$ whose analytification is induced
	by a representation
	$\pi_1^{\mathrm{top}}(\bar X^{\mathrm{an}})
	\to\operatorname{PU}(r)$.
\end{theorem}

\begin{proof}
	The restriction $E|_U$ is locally free by
	Theorem~\ref{general-local-freeness}. The reflexive sheaf
	$B=\cEnd(E)$ is $H$-polystable, as follows by
	restriction to general smooth projective curves; see also
	\cite[Lemma~4.7]{GKPT2020}. Moreover,
	\[
	c_1(B)=0,\qquad
	\operatorname{ch}_2(B) H^{d-2}
	=-\Delta(E)H^{d-2}=0.
	\]
	The nonabelian Hodge correspondence of
	\cite[Theorem~6.3]{GKPT2020}, applied to the algebraic Higgs bundle
	$(A,0)$ with $A=B|_U$, gives a unitary flat structure on
	$A^{\mathrm{an}}$. Tensor compatibility and full faithfulness
	show that the analytifications of the algebraic Higgs morphisms
	\[
	A\otimes A\longrightarrow A,\qquad\mathcal O_U\longrightarrow A
	\]
	given by multiplication and unit are horizontal. By
	Lemma~\ref{ssb:lem-projective-flat-algebra}, this gives a
	representation
	\(
	\rho:\pi_1^{\mathrm{top}}(U^{\mathrm{an}})
	\to\operatorname{PU}(r)
	\)
	inducing $\mathbb P(E|_U)^{\mathrm{an}}$.
	
	By \cite[Theorem~1.5]{Greb-Kebekus-Peternell-Duke2016}, there is
	a finite surjective Galois morphism $\gamma:\bar X\to X$,
	\'etale in codimension one, such that
	$\bar X_{\mathrm{reg}}\hookrightarrow\bar X$ induces an
	isomorphism of \'etale fundamental groups. Set $W=\gamma^{-1}(U)$.
	Purity implies that $\gamma|_W$ is \'etale. Thus $W$ is a smooth
	big open subset of $\bar X_{\mathrm{reg}}$, and
	\(
	\pi_1^{\mathrm{top}}(W^{\mathrm{an}})
	\simeq\pi_1^{\mathrm{top}}(\bar X_{\mathrm{reg}}^{\mathrm{an}}).
	\)
	The pullback of $\rho$ therefore defines a representation
	$\sigma:\pi_1^{\mathrm{top}}(\bar X_{\mathrm{reg}}^{\mathrm{an}})
	\to\operatorname{PU}(r)$.
	
	As in \cite[Proposition~3.10]{Greb-Kebekus-Peternell2022}, this
	representation factors through the natural surjection
	\[
	q:\pi_1^{\mathrm{top}}(\bar X_{\mathrm{reg}}^{\mathrm{an}})
	\twoheadrightarrow\pi_1^{\mathrm{top}}(\bar X^{\mathrm{an}}).
	\]
	Indeed, the profinite completion of $q$ is an isomorphism, whereas
	$\operatorname{im}\sigma$ is a finitely generated linear group
	and hence residually finite by Malcev's theorem. A nontrivial
	image of an element of $\ker q$ would therefore be detected by
	a finite quotient, a contradiction. Write $\sigma=\bar\rho\circ q$.
	
	The representation $\bar\rho$ defines a holomorphic matrix-algebra
	bundle $\mathcal B$ on $\bar X^{\mathrm{an}}$. On $W^{\mathrm{an}}$
	it is holomorphically isomorphic, as an algebra bundle, to
	$(\gamma^{[*]}B)^{\mathrm{an}}|_{W^{\mathrm{an}}}$. Both sheaves
	are reflexive analytic sheaves on $\bar X^{\mathrm{an}}$, so this
	isomorphism and its inverse extend uniquely across the complement
	of $W^{\mathrm{an}}$, which has codimension at least two.
	GAGA on the projective variety $\bar X$ algebraizes $\mathcal B$,
	its multiplication and unit, and this global isomorphism. The
	result is a locally free algebraic matrix-algebra bundle
	isomorphic to $\gamma^{[*]}B$. Its associated Brauer--Severi
	scheme is the required algebraic extension $P\to\bar X$.
\end{proof}

\begin{remark}
	Here a Brauer--Severi bundle is an algebraic morphism that is
	\'etale-locally isomorphic to $\mathbb P^{r-1}\times\bar X\to\bar X$.
	In general, $(\gamma|_W)^*(E|_U)$ itself need not extend to a vector
	bundle on $\bar X$ (see
	\cite[Example~3.9]{Greb-Kebekus-Peternell2022}).
\end{remark}

\subsection{A Lefschetz argument for projective unitary flatness}

\begin{lemma}\label{ssb:lem-lefschetz-projective-flatness}
	Let $(X,H)$ be a normal polarized complex projective
	variety of dimension $n\geq3$, and let $E$ be a vector
	bundle on $X$. Suppose that, for arbitrarily large integers $m$,
	the restriction $E|_D$ is projectively unitary flat for a general
	$D\in|mH|$.
	Then $E$ is projectively unitary flat on $X$.
\end{lemma}

\begin{proof}
	Put $U=X_{\mathrm{reg}}$ and $A=\cEnd(E)$.
	Choose a sufficiently general normal divisor $D_0$ as in the
	hypothesis. Applied after a suitable Veronese embedding, the quasi-projective Lefschetz theorem
	\cite[Theorem~1.1.3]{ssb-HammLe1985} gives
	\[
	\pi_1^{\mathrm{top}}\bigl((D_0\cap U)^{\mathrm{an}}\bigr)
	\xrightarrow{\sim}\pi_1^{\mathrm{top}}(U^{\mathrm{an}}).
	\]
	Extend the representation defining the unitary flat holomorphic
	algebra bundle $(A|_{D_0\cap U})^{\mathrm{an}}$ by this isomorphism.
	The regular Riemann--Hilbert correspondence
	\cite[Theorem~II.5.9]{ssb-Deligne1970} produces an algebraic vector
	bundle $V$ on $U$ with an algebraic integrable regular singular
	connection and horizontal algebraic matrix-algebra structure.
	The monodromy of its analytification lies in the adjoint image of
	$\operatorname{PU}(r)$. For the algebraic inclusion
	$j:U\hookrightarrow X$, put $F=j_*V$, where the direct image is
	taken in the algebraic category. This is a coherent reflexive
	algebraic sheaf: extend $V$ to a coherent algebraic sheaf on $X$
	and take its reflexive hull. Its multiplication and unit extend
	uniquely as algebraic morphisms.
	
	We claim that the same local system of algebras on $U^{\mathrm{an}}$
	is obtained from every other sufficiently general divisor $D$
	satisfying the hypothesis, including divisors of arbitrarily large
	degree. Intersect $D_0$ and $D$ with $n-3$ further sufficiently
	general ample divisors. The result is a smooth connected
	projective curve
	\(
	C\subset D_0\cap D\cap U,
	\)
	and successive applications of Lefschetz give a surjection
	$\pi_1^{\mathrm{top}}(C^{\mathrm{an}})
	\twoheadrightarrow\pi_1^{\mathrm{top}}(U^{\mathrm{an}})$.
	The two induced unitary flat holomorphic connections on
	$(A|_C)^{\mathrm{an}}$ coincide. Indeed, by
	Lemma~\ref{ssb:cor-unitary-horizontal}, the identity
	holomorphic morphism between these unitary flat bundles on
	$C^{\mathrm{an}}$ is horizontal. It also preserves the algebra
	structure. Surjectivity on fundamental groups proves the claim.
	
	The holomorphic flat connection on $(A|_{D\cap U})^{\mathrm{an}}$
	supplied by the hypothesis is the analytification of an algebraic
	regular singular connection. To see this, choose an algebraic
	projective resolution $q:\widetilde D\to D$ which is an
	isomorphism over $D\cap U$. The given unitary flat structure on
	$(A|_D)^{\mathrm{an}}$ pulls back to a holomorphic flat connection
	on $(q^*(A|_D))^{\mathrm{an}}$. Since $\widetilde D$ is smooth and
	projective, GAGA algebraizes its splitting of the first-jet
	sequence, hence the connection itself. Its restriction over
	$D\cap U$ is the required algebraic connection, and extension over
	$\widetilde D$ shows that it is regular singular.
	Full faithfulness of the regular Riemann--Hilbert correspondence
	therefore gives an algebraic isomorphism
	$A|_{D\cap U}\simeq F|_{D\cap U}$. For general $D$, the algebraic
	sheaf $F|_D$ is reflexive by Bertini for $S_2$ sheaves
	\cite[Corollary~1.4 and Theorem~1.5]{Langer-Simpson-singular}.
	Since $D\setminus U$ has codimension at least two in $D$, the
	isomorphism extends uniquely to an algebraic isomorphism of algebras
	\(
	\alpha_D:A|_D\xrightarrow{\sim}F|_D.
	\)
	
	The algebraic sheaf $F$ has now been fixed independently of the
	final degree. Choose $D\in|mH|$ as above with $m$ sufficiently
	large so that
	\(
	H^1(X,A^*\otimes F(-D))=0
	\)
	by the Enriques--Severi--Zariski lemma
(see	\cite[Tag~0FD8]{Stacks}). We can  also assume that
	\[
	H^0(X,F(-D))=
	H^0\bigl(X,(A\otimes A)^*\otimes F(-D)\bigr)=0.
	\]
	The first vanishing lifts $\alpha_D$ to an algebraic morphism
	$\alpha:A\to F$. Nakayama's lemma shows that $\alpha$ is an
	isomorphism near $D$. Its cokernel is supported away from the
	ample divisor $D$, so it has finite support. The map is generically
	an isomorphism, and reflexivity therefore makes it an isomorphism
	on all of $X$. The discrepancies measuring preservation of the
	unit and multiplication vanish on $D$ and factor through $F(-D)$.
	The two $H^0$ vanishings show that they vanish everywhere.
	Thus $\alpha$ is an algebraic isomorphism of algebras.
	
	Using $\alpha^{\mathrm{an}}|_{U^{\mathrm{an}}}$, transport the
	unitary flat holomorphic algebra structure on $V^{\mathrm{an}}$
	to $(A|_U)^{\mathrm{an}}$. By
	Lemma~\ref{ssb:lem-projective-flat-algebra}, it gives a projective
	unitary flat structure on $E|_U$. Apply
	Lemma~\ref{ssb:lem-unitary-extension} to the algebraic vector
	bundle $\cEnd(E)$, retaining its parallel
	multiplication and unit, to extend this structure over $X^{\mathrm{an}}$.
\end{proof}

\begin{proposition}\label{ssb:prop-stable-zero-discriminant-normal}
	Let $(X,H)$ be a normal polarized complex projective
	variety of dimension $n\geq2$, and let $E$ be an $H$-stable
	vector bundle of rank $r\geq2$. If
	\(\Delta(E)H^{n-2}=0,\)
	then $E$ is irreducible projectively unitary flat on $X$.
\end{proposition}

\begin{proof}
	We first treat the case $n=2$. Take a projective resolution
	$f:Z\to X$ which is an isomorphism over $X_{\mathrm{reg}}$, and put
	$h=f^*H$ and $\widetilde E=f^*E$. The bundle
	$\widetilde E$ is $h$-stable.
	Choose an effective exceptional divisor $D$ with $-D$ relatively
	ample, and $b$ such that $P=bh-D$ is ample. Then by Lemma \ref{ssb:lem-nearby-polarizations}
$\widetilde E$ is stable for $P_N=Nh-D$ when $N\gg0$.
	
	The projection formula gives
	$\int_Z\Delta(\widetilde E)=\int_X\Delta(E)=0$.
	Corollary~\ref{cor:DUY} gives a projective unitary flat structure
	on $\widetilde E^{\mathrm{an}}$.
	Thus $(f^*\cEnd(E))^{\mathrm{an}}$ is unitary flat
	with parallel matrix-algebra structure.
	Lemma~\ref{ssb:lem-unitary-descent} descends this structure to
	$\cEnd(E)^{\mathrm{an}}$, proving projective unitary
	flatness of $E$.
	
	For $n\geq3$, argue by induction. By the stable
	restriction theorem \cite{ssb-MehtaRamanathan1984}, for every
	sufficiently large $m$ a general divisor $D\in|mH|$ is
	normal and $E|_D$ is $H|_D$-stable. Moreover,
	\[
	\Delta(E|_D)\cdot(H|_D)^{n-3}
	=m\,\Delta(E)H^{n-2}=0.
	\]
	The induction hypothesis and
	Lemma~\ref{ssb:lem-lefschetz-projective-flatness} prove projective
	unitary flatness on $X$.
	
	Finally, an invariant proper projective linear subspace defines
	a proper holomorphic subbundle of $E^{\mathrm{an}}$. Since $X$ is
	projective, GAGA algebraizes it to a subbundle $B\subset E$.
	On the analytification of a general smooth complete-intersection
	curve, choose a Hermitian metric on $E$ inducing its unitary flat
	projective metric. The induced Chern curvatures of $B$ and $E$ are
	given by the same scalar two-form times their respective identity
	endomorphisms. Their degrees divided by their ranks therefore
	agree. Thus $\mu_H(B)=\mu_H(E)$, contrary to algebraic slope
	stability. This proves irreducibility.
\end{proof}

\section{Holonomy of polystable reflexive sheaves}

Let $X$ be a normal complex projective variety, let $H$ be an ample
line bundle on $X$, and let $E$ be an $H$-polystable reflexive sheaf
of rank $r\geq1$. Let $U\subset X_{\mathrm{reg}}$ be the open subset
on which $E$ is locally free, and fix $x\in U(\bC)$.

Applying Theorem~\ref{cor:restriction-curves-prescribed-points} to
the stable direct summands of $E$, we see that $E|_C$ is polystable
for a general smooth complete intersection curve $x\in C\subset U$
of sufficiently large multidegree. If $X$ is a curve, we take $C=X$.
By the Narasimhan--Seshadri correspondence \cite{nar-ses}, $E|_C$
determines a unitary representation, unique up to unitary equivalence,
\[
\rho:\pi_1^{\mathrm{top}}(C^{\mathrm{an}},x)
\longrightarrow U(E_x)\subset\GL(E_x)
\]
when $\deg(E|_C)=0$. In arbitrary degree, after choosing
$c\in C(\bC)\setminus\{x\}$, we use the representation
\[
\rho_c:\pi_1^{\mathrm{top}}(C^{\mathrm{an}}\setminus\{c\},x)
\longrightarrow U(E_x)\subset\GL(E_x),
\]
whose local monodromy around $c$ is
$\exp(2\pi i\deg(E|_C)/r)\operatorname{id}_{E_x}$.
The representation is irreducible precisely when $E|_C$ is stable.
Following \cite{ssb-BalajiKollar2008}, we call these the
\emph{Narasimhan--Seshadri representations} of $E|_C$.
Since the local monodromy at the auxiliary puncture is scalar,
the projectivized representations extend across that point.

The construction of Balaji and Koll\'ar extends to polystable
reflexive sheaves as follows.

\begin{theorem}[{Balaji--Koll\'ar; cf.\ \cite[Definition~19 and
Theorem~20]{ssb-BalajiKollar2008}}]
\label{Balaji-Kollar}
There exists a unique reductive subgroup
$G=\Hol_x(E)\subset\GL(E_x)$, called the \emph{holonomy group} of
$E$, characterized by the following property: $G$ is the smallest
algebraic subgroup containing the images of the Narasimhan--Seshadri
representations of $g^*(E|_U)$ for every morphism
\[
g:(C,y)\longrightarrow(U,x)
\]
from a smooth pointed projective curve for which $g^*(E|_U)$ is
polystable, and for every choice of the auxiliary point
$c\in C(\bC)\setminus\{y\}$.

Furthermore:
\begin{enumerate}
\item If $\dim X\geq2$, a very general complete intersection curve
$x\in C\subset U$ of sufficiently large multidegree detects $G$:
the image of its Narasimhan--Seshadri representation is Zariski
dense in $G$, with the auxiliary point chosen very generally in
the nonzero-degree case. If $X$ is a curve of positive genus, the
same assertion holds with $C=X$.

\item For every $a,b\geq0$, the fiber map $F\mapsto F_x$ gives a
one-to-one correspondence between direct summand subsheaves of
\(
(E^{\otimes a}\otimes(E^*)^{\otimes b})^{**}
\)
and $G$-invariant subspaces of
$E_x^{\otimes a}\otimes(E_x^*)^{\otimes b}$. Under this
correspondence, stable direct summands correspond to irreducible
representations. In particular, $E$ is $H$-stable if and only if
$G$ acts irreducibly on $E_x$.

\item The conjugacy class of $G$ is the smallest reductive conjugacy
class to which the structure group of $E|_U$ can be reduced.
\end{enumerate}
Here reductive groups need not be connected.
\end{theorem}

\begin{proof}
The argument of \cite[Theorem~20]{ssb-BalajiKollar2008} applies to
polystable sheaves. Indeed, the curve correspondence and the
generic-monodromy result in \cite[Section~1 and Lemma~15]{ssb-BalajiKollar2008}
are already formulated for polystable bundles. The restriction
theorem applies to the stable direct summands of $E$ and of its
reflexive tensor constructions, which are polystable in
characteristic zero. Thus the generic monodromy stabilizes, and
the extension of direct summands from sufficiently positive curves
gives the stated correspondence and the minimal reductive
reduction. Stability is equivalent to irreducibility because these
tensor constructions are polystable.

For completeness, if $X\simeq\mathbb P^1$, then
$E\simeq\mathcal O_{\mathbb P^1}(d)^{\oplus r}$ for some integer
$d$. Its holonomy is trivial for $d=0$ and is the scalar
$\mathbb G_m$ for $d\ne0$. In the latter case this scalar group
is detected after pullback to a curve of positive genus and a
very general choice of the auxiliary point. This proves the
remaining assertions also in this case; the restriction in~(1)
is necessary because $\mathbb P^1$ minus one point has trivial
fundamental group.
\end{proof}

Deleting a further closed subset of codimension at least two from
$U$ does not change the holonomy group: sufficiently general
curves detecting it avoid that subset.

\begin{remark}\label{stability-of-Sym-gives-usual}
For any reflexive sheaf $E$, if $\Sym^{[m]}E$ is $H$-stable for some $m\ge 2$, then $E$ is $H$-stable. Indeed, a saturated proper subsheaf $F\subset E$	with $\mu_H(F)\geq\mu_H(E)$ would induce a subsheaf $\Sym^{[m]}F\subset \Sym^{[m]}E$
	whose slope is at least $m\mu_H(E)$. 
\end{remark}

The above remark shows that in the results below, the holonomy group of $E$ is well defined.

\begin{lemma}\label{ssb:lem-holonomy-square-dichotomy}
	Assume that $r\geq2$ and $\Sym^{[2]}E$ is $H$-stable.
	Put $V=E_x$ and $G=\Hol_x(E)\subseteq\GL(V)$.
	Then exactly one of the following holds:
	\begin{enumerate}
		\item The image of $G$ in $\PGL(V)$ is finite.
		\item 
		There are a connected semisimple subgroup $D\subseteq G$
		and a subgroup $Z\subseteq Z(\GL(V))=\mathbb G_m$ such that
		\begin{equation}\label{ssb:eq-holonomy}
			G=Z\cdot D,\qquad
			D=[G^\circ,G^\circ]
			=\SL(V)\quad\text{or}\quad\Sp(V,b),
		\end{equation}
		where $b$ is a nondegenerate alternating form in the
		symplectic case. Every $\Sym^{[m]}E$, $m\geq1$, is
		$H$-stable.
		
		In the symplectic case $r=2n$ and there are a rank one
		reflexive sheaf $N$ and an algebraic alternating form,
		nondegenerate on the smooth locally free locus $U$ of $E$,
		\begin{equation}\label{ssb:eq-form}
			\beta:\bigwedge\nolimits^{[2]}E\longrightarrow N,
			\qquad \det E\simeq N^{[n]}:=(N^{\otimes n})^{**}.
		\end{equation}
	\end{enumerate}
\end{lemma}

\begin{proof}
	By Theorem~\ref{Balaji-Kollar}, $\Sym^2V$ is an irreducible
	$G$-module. Propositions~2.8 and~2.14 and Theorem~3.1(i) of
	\cite{ssb-GuralnickTiep2008} show that either the projective
	image of $G$ is finite, or
	\[
	D:=[G^\circ,G^\circ]=\SL(V)\quad\text{or}\quad\Sp(V,b).
	\]
	Indeed, the normalized group in those results has the same
	projective image as $G$; it is either finite or has connected
	component $\SL(V)$ or $\Sp(V,b)$.
	
	Suppose that the projective image is infinite. Since $D$ is
	normal in $G$, we have
	\[
	D\subseteq G\subseteq N_{\GL(V)}(D).
	\]
	The two normalizers are $\GL(V)$ and $\GSp(V,b)$, respectively.
	Every element of either normalizer is a scalar times an element
	of $D$. Since $D\subseteq G$, this gives $G=Z\cdot D$ with
	$Z=G\cap Z(\GL(V))$, proving \eqref{ssb:eq-holonomy}.
	Every symmetric power of the defining representation of
	$\SL(V)$ or $\Sp(V,b)$ is irreducible. The holonomy
	correspondence therefore gives stability of every $\Sym^{[m]}E$.
	
	In the symplectic case the line
	$\mathbb Cb\subseteq\bigwedge^2V^*$ is $G$-invariant.
	Associating it to the algebraic holonomy reduction gives a
	line subbundle
	$N^{-1}|_U\subseteq\bigwedge^2(E|_U)^*$, where $N$ denotes
	the rank one reflexive extension of its dual to $X$.
	This gives the form in \eqref{ssb:eq-form} by reflexive
	extension. On $U$, it is the fixed nondegenerate form $b$
	in a local trivialization of the holonomy bundle.
	Its $n$-fold exterior product gives
	$\det E\simeq N^{[n]}$ after reflexive extension.
\end{proof}

\begin{lemma}\label{ssb:lem-finite-projective-holonomy}
	Let $X$ be a normal complex projective variety of dimension
	$n\geq2$, and let $E$ be an $H$-stable reflexive sheaf.
	Suppose that the image of $\Hol_x(E)$ in $\PGL(E_x)$ is finite.
	Then $E|_{X_{\mathrm{reg}}}$ is locally free and irreducible
	projectively unitary flat, with finite projective monodromy.
	If $E$ is locally free on $X$, this structure extends over $X$.
	In particular, if $X$ is smooth, then $E$ is a locally free
	irreducible projectively unitary flat sheaf on $X$.
	If $X$ has quotient singularities in codimension two, then
	$\int_X\Delta(E)=0$.
\end{lemma}

\begin{proof}
	Let $U\subset X_{\mathrm{reg}}$ be the locally free locus of $E$.
	The algebraic holonomy reduction of $E|_U$ induces a reduction
	of its projective-frame bundle to a finite subgroup of
	$\PGL(E_x)$. Consequently, there is a connected finite \'etale
	Galois cover $T\to U$ trivializing $\mathbb P(E|_U)$.
	Normalize $X$ in the function field of $T$ to obtain a finite
	surjective Galois morphism $\gamma:Y\to X$, \'etale in
	codimension one. On $T$ the pullback of $E$ is a direct sum
	of copies of a line bundle. Taking its rank one reflexive
	extension $L$ to $Y$ gives
	\(	\gamma^{[*]}E\simeq L^{\oplus r}\), where $r=\operatorname{rk}E$.
	Purity shows that $\gamma$ is \'etale over $X_{\mathrm{reg}}$.
	Thus $Y$ is smooth over $X_{\mathrm{reg}}$, and $L$ is invertible
	there. Faithfully flat descent now proves that
	$E|_{X_{\mathrm{reg}}}$ is locally free. The finite projective
	structure extends over this regular locus, since its complement
	of $U$ has codimension at least two. A finite subgroup of
	$\PGL_r(\mathbb C)$ is conjugate to a subgroup of
	$\operatorname{PU}(r)$, so the resulting structure is projectively
	unitary flat. Its restriction to a sufficiently general stable
	complete intersection curve is irreducible; hence the structure
	on $X_{\mathrm{reg}}$ is irreducible as well.
	
	If $E$ is locally free on $X$, apply
	Lemma~\ref{ssb:lem-unitary-extension} to $\cEnd(E)$, retaining
	its parallel multiplication and unit, and then apply
	Lemma~\ref{ssb:lem-projective-flat-algebra}. This extends the
	projective unitary flat structure over $X$.
	
	Finally, suppose that $X$ has quotient singularities in
	codimension two. The quasi-\'etale cover $Y$ has the same property.
	The rational Chern classes of $L^{\oplus r}$ have zero
	discriminant. Functoriality of these classes under finite
	quasi-\'etale morphisms (see \cite{ssb-LangerIntersection2025}
	and \cite[Section~7]{ssb-Langer2024}) therefore gives
	\[
	\deg(\gamma)\,\Delta(E)M_1\cdots M_{n-2}
	=\Delta(\gamma^{[*]}E)\gamma^*M_1\cdots\gamma^*M_{n-2}=0
	\]
	for all Cartier divisors $M_1,\ldots,M_{n-2}$ on $X$.
\end{proof}

\begin{lemma}\label{ssb:prop-holonomy}
	Assume that $r\geq2$ and $\Sym^{[6]}E$ is $H$-stable.
	Then the second alternative of
	Lemma~\ref{ssb:lem-holonomy-square-dichotomy} holds.
	In particular, \eqref{ssb:eq-holonomy} and, in the symplectic
	case, \eqref{ssb:eq-form} hold, and every $\Sym^{[m]}E$,
	$m\geq1$, is $H$-stable.
\end{lemma}

\begin{proof}
	Put $V=E_x$ and $G=\Hol_x(E)$.
	By Theorem~\ref{Balaji-Kollar}, $\Sym^6V$ is an irreducible
	$G$-module. By \cite[Corollary~1.6]{ssb-GuralnickTiep2008},
	$\Sym^2V$ is also irreducible, so $\Sym^{[2]}E$ is $H$-stable.
	Moreover, \cite[Theorems~1.1 and~1.2]{ssb-GuralnickTiep2008}
	exclude a finite projective image when $\Sym^6V$ is irreducible.
	The assertion therefore follows from
	Lemma~\ref{ssb:lem-holonomy-square-dichotomy}.
\end{proof}

\begin{remark}
	Even when $c_1(E)\equiv0$, the scalar subgroup $Z$ need not be finite,
	and the form in \eqref{ssb:eq-form} need not be
	$\mathcal O_X$-valued.
\end{remark}

\begin{corollary}\label{stability-pullback-of-symmetric-powers}
\label{ssb:cor-derived-holonomy-pullback}
Let $X$ be a normal complex projective variety, let $H$ be ample,
and let $E$ be an $H$-polystable reflexive sheaf of rank $r\geq1$.
Fix a general point $x$
in its smooth locally free locus, and set
\(V=E_x,\) \(G=\Hol_x(E)\) and \( D=[G^\circ,G^\circ].\)
\begin{enumerate}
\item If $f:Y\to X$ is a finite surjective morphism from a normal
projective variety, then $f^{[*]}E$ is
$f^*H$-polystable. For compatible general base points $y\in Y$
and $x=f(y)$, identify $(f^{[*]}E)_y$ with $V$ and write
$G_f=\Hol_y(f^{[*]}E)$. Then
$G_f\subseteq G$ and $[G_f^\circ,G_f^\circ]=D.$
\item For every $m\geq1$, the following conditions are equivalent:
\begin{enumerate}[(a)]
\item $D$ acts irreducibly on $\Sym^mV$;
\item $\Sym^{[m]}(f^{[*]}E)$ is $f^*H$-stable for every finite
surjective morphism $f:Y\to X$ from a normal projective variety;
\item $\Sym^{[m]}(f^{[*]}E)$ is $f^*H$-stable for every generically
finite surjective morphism $f:Y\to X$ from a normal projective
variety.
\end{enumerate}
In~(c), slope stability is understood with respect to the nef and big
divisor $f^*H$.
In particular, $E$ remains stable under every finite surjective
reflexive pullback if and only if $D$ acts irreducibly on $V$.
\end{enumerate}
If $\Sym^{[6]}E$ is $H$-stable, then the conditions in~(2) hold
for every $m\geq1$.
\end{corollary}

\begin{proof}
Finite reflexive pullback preserves polystability in
characteristic zero (see \cite[Lemma~3.2.3]{Huybrechts-Lehn2010}).
 This gives the first assertion of~(1). By the curve
characterization of the holonomy group in Theorem~\ref{Balaji-Kollar}, composition with
$f$ gives $G_f\subseteq G$.

If $X\simeq\mathbb P^1$, both holonomy groups $G$ and $G_f$ are scalar and
their derived subgroups are trivial, so the conclusion of (1) holds.
So we can assume that $\dim X\geq2$ or that $X$ is a curve of positive
genus. Choose a sufficiently positive very general curve
$x\in C\subset X$ detecting $G$, and let $g:(C',y)\to(C,x)$ be
the normalization of a component of $f^{-1}(C)$ through $y$.
We choose $C$ to avoid the images of the excluded codimension-two
subsets. If $X$ is a curve, take $C=X$ and $C'=Y$.
The projectivized Narasimhan--Seshadri representation is
\(
\bar\rho:\pi_1^{\mathrm{top}}(C^{\mathrm{an}},x)
\to\PGL(V).
\)
The projectively unitary flat connection pulls back under $g$. Therefore
the projective Narasimhan--Seshadri monodromy of
$g^*(E|_C)$ is $\bar\rho\circ g_*$. This also holds
when $g$ is ramified as adjusting the scalar part of the pulled-back
connection to have constant central curvature does not change
its projective monodromy.

The image $g_*\pi_1^{\mathrm{top}}((C')^{\mathrm{an}},y)$ has finite
index in $\pi_1^{\mathrm{top}}(C^{\mathrm{an}},x)$. Indeed, remove
the branch points, use the finite-index subgroup associated with
the resulting covering, and then fill in the punctures.
Consequently, if $p:\GL(V)\to\PGL(V)$ is the quotient map, the
Zariski closure of $\bar\rho\circ g_*$ has identity component
$p(G)^\circ$. Its image is contained in $p(G_f)$, and the reverse
inclusion of identity components follows from $G_f\subseteq G$.
Thus
\(
p(G_f^\circ)=p(G^\circ).
\)
Writing $\mathfrak g$ and $\mathfrak g_f$ for their Lie algebras,
we obtain
\[
\mathfrak g=\mathfrak g_f+
(\mathfrak g\cap\mathbb C\operatorname{id}_V) ,
\qquad
[\mathfrak g,\mathfrak g]=[\mathfrak g_f,\mathfrak g_f].
\]
The connected derived subgroups therefore coincide, proving (1).

For~(2), condition~(a) and~(1) imply that $G_f$ acts irreducibly
on $\Sym^mV$ for every finite $f$. The polystability of reflexive
symmetric powers and Theorem~\ref{Balaji-Kollar} give~(b).
Conversely, let $P\to U$ be the holonomy reduction of $E|_U$.
The cover $P/G^\circ\to U$ is finite \'etale. As in
\cite[Lemma~40]{ssb-BalajiKollar2008}, choose a connected
component and extend it by normalization to a finite surjective
morphism $h:Z\to X$. On the inverse image of $U$, the pulled-back
holonomy reduction reduces further to $G^\circ$. If~(b) holds,
the stability of $\Sym^{[m]}(h^{[*]}E)$ implies that
$G^\circ$ acts irreducibly on $\Sym^mV$. The center of $G^\circ$
then acts by scalars. Since
$G^\circ=Z(G^\circ)^\circ D$, its restriction to $D$ is also
irreducible. This proves~(a).

It remains to prove that~(a) implies~(c), since the converse
follows from~(b). Let $f:Y\to X$ be generically finite and
surjective. Its Stein factorization is
\(
Y\xrightarrow{b}Z\xrightarrow{h}X,
\)
where $Z$ is normal, $b$ is birational and $h$ is finite. Let us set
\[
B=h^*H,\qquad M=\Sym^{[m]}(h^{[*]}E),\qquad
E_m=\Sym^{[m]}(f^{[*]}E).
\]
By the finite case, $M$ is $B$-stable. For a saturated subsheaf
$F\subset E_m$ of intermediate rank, let $F_Z\subset M$ be the
saturated subsheaf determined by its generic fiber. Since $b$
is an isomorphism over a big open subset of $Z$, the projection
formula gives
\[
\mu_{b^*B}(F)=\mu_B(F_Z)
<\mu_B(M)=\mu_{b^*B}(E_m).
\]
As $b^*B=f^*H$, this proves~(c).

Finally, if $\Sym^{[6]}E$ is stable and $r\geq2$,
Lemma~\ref{ssb:prop-holonomy} gives
$D=\SL(V)$ or $D=\Sp(V,b)$. All symmetric powers of the defining
representation of either group are irreducible, so~(2) applies
for every $m$. The rank one case is immediate.
\end{proof}

\section{Pseudoeffectivity of a normalized exterior power}

Let \((X,H)\) be a smooth polarized complex projective variety of
dimension \(d\geq2\), and let \(E\) be a vector bundle of rank \(r\geq2\)
on \(X\).  Choose a finitely generated integral subring
\(R\subset\mathbb C\), put \(S:=\operatorname{Spec}R\), and choose a model
\[
(\mathscr X,\mathscr H,\mathscr E)\longrightarrow S
\]
of \((X,H,E)\), with \(\mathscr X/S\) smooth and projective,
\(\mathscr H\) relatively ample, and \(\mathscr E\) locally free. 

For \(1\leq i<r\), we write
\[
\begin{gathered}
	G_i:=\bigwedge\nolimits^iE,
	\quad \pi_i:Y_i:=\mathbb P(G_i)\longrightarrow X,
	\quad \xi_i:=c_1\bigl(\mathcal O_{Y_i}(1)\bigr),\quad
	\theta_i:=\xi_i-\frac{i}{r}\pi_i^*c_1(E).
\end{gathered}
\]
In particular, \(r\theta_i\) is the first Chern class of the line bundle
\(
\widehat{\mathcal L}_i
:=\mathcal O_{Y_i}(r)\otimes\pi_i^*(\det E)^{-i}.
\)

\begin{theorem}
	\label{thm:pseudoeffective-exterior-power}
	Let \(E\) be an \(H\)-stable vector bundle of rank \(r\geq2\) on a smooth
	complex projective variety \(X\) of dimension \(d\geq2\).  If
	\(\Sigma^{\mathrm{ns}}(\mathscr E/S)\) is Zariski dense in \(S\), then
	the line bundle  \(\widehat{\mathcal L}_i\) 	is pseudoeffective for some \(1\leq i<r\). If, in addition, \(\det E\) is pseudoeffective, then
	\(\bigwedge^iE\) is pseudoeffective for the same \(i\).
\end{theorem}

We first establish the numerical consequence of semistability needed
in the proof. We write \(\operatorname{Mov}(X)\subset N_1(X)_{\mathbb R}\)
for the closed cone of movable curve classes and set
\(\alpha_0:=H^{d-1}\). This is an interior point of
\(\operatorname{Mov}(X)\). Indeed, if \(B_1,\ldots,B_\rho\) form a basis
of \(N^1(X)_{\mathbb R}\), the Hodge index theorem implies that
\(B_1H^{d-2},\ldots,B_\rho H^{d-2}\) form a basis of
\(N_1(X)_{\mathbb R}\). For sufficiently small \(t>0\), the classes
\[
\alpha_0\pm tB_jH^{d-2}=(H\pm tB_j)H^{d-2}
\]
are movable complete-intersection classes, and their convex hull
contains a neighborhood of \(\alpha_0\).

\begin{lemma}
	\label{lem:movable-class-over-polarization}
	Let \(\pi:Y\to X\) be a smooth projective morphism with
	geometrically integral fibers of positive dimension, and let \(L\)
	be a relatively ample line bundle such that
	\[
	\operatorname{Pic}(Y)\otimes\mathbb Q
	=\mathbb Q[L]\oplus\pi^*\bigl(\operatorname{Pic}(X)\otimes\mathbb Q\bigr).
	\]
	Let \(U\subset\operatorname{Mov}(X)\) be an open neighborhood of
	\(\alpha_0\), where openness is understood in
	\(N_1(X)_{\mathbb R}\). Suppose that, for every integer \(a\geq0\),
	the sheaf \(V_a:=\pi_*L^{\otimes a}\) is locally free of positive
	rank and slope-semistable with respect to every \(\alpha\in U\).
	Assume that there is a class \(\eta\in N^1(X)_{\mathbb Q}\) such that
	\[
	\frac{c_1(V_a)}{\operatorname{rk}V_a}=a\eta
	\qquad(a\geq0),
	\]
	and put \(\theta:=c_1(L)-\pi^*\eta\).
	If \(\theta\) is not pseudoeffective, there is a class
	\(\beta\in N_1(Y)_{\mathbb Q}\) such that
	\[
	\pi_*\beta\equiv\alpha_0,
	\qquad \theta\cdot\beta<0,
	\]
	and \(\beta\) is a finite positive rational linear combination of
	strongly movable curve classes.
\end{lemma}

\begin{proof}
	For an effective integral divisor \(T\) on $Y$, some
	positive multiple \(mT\) has associated line bundle
	\(L^{\otimes a}\otimes\pi^*M\), with \(a\geq0\) an integer.
	Its defining section gives a nonzero morphism
	\(
	M^{-1}\to V_a.
	\)
	By the assumed semistability, for every \(\alpha\in U\),
	\[
	-c_1(M)\cdot\alpha\leq\mu_\alpha(V_a)=a\eta\cdot\alpha.
	\]
	Writing the numerical class of \(mT\) in $N^1(Y)_{\mathbb Q}$ as a sum of $a\theta$ and the pullback of
	\(D:=c_1(M)+a\eta\), we get $D\cdot\alpha\geq0$.
	So for  every pseudoeffective class \(a\theta+\pi^*D\)  and $\alpha\in U$ we also have
	$	a\geq0$ and $D\cdot\alpha\geq0.$
	
	Choose a basis \(\gamma_1,\ldots,\gamma_\rho\) of
	\(N_1(X)_{\mathbb R}\) and \(t>0\) such that
	\(\alpha_0\pm t\gamma_j\in U\) for every \(j\). Then
	\(
	|D\cdot\gamma_j|\leq t^{-1}D\cdot\alpha_0.
	\)
	For any fixed norm on \(N^1(X)_{\mathbb R}\), there is consequently
	a constant \(C>0\) such that
	\begin{equation}
		\|D\|\leq C(D\cdot\alpha_0)
		\label{eq:exterior-base-class-bound}
	\end{equation}
	whenever \(a\theta+\pi^*D\) is pseudoeffective.
	
	Since \(\theta\) is not pseudoeffective, there is an
	\(\varepsilon>0\) such that
	\begin{equation}
		D\cdot\alpha_0\geq\varepsilon a
		\label{eq:exterior-positive-threshold}
	\end{equation}
	for all these classes. Otherwise, normalizing \(a=1\) would give
	pseudoeffective classes \(\theta+\pi^*D_n\) with
	\(D_n\cdot\alpha_0\to0\); by
	\eqref{eq:exterior-base-class-bound}, \(D_n\to0\), contradicting
	closedness of the pseudoeffective cone.
	
	Choose a rational number \(0<\delta<\varepsilon\), and define
	\(\beta\in N_1(Y)_{\mathbb Q}\) by
	\[
	\theta\cdot\beta=-\delta,
	\qquad
	\pi^*D\cdot\beta=D\cdot\alpha_0
	\quad(D\in N^1(X)_{\mathbb R}).
	\]
	The class \(\beta\) is rational because both \(\theta\) and
	\(\alpha_0\) are rational. For every nonzero pseudoeffective class,
	\[
	(a\theta+\pi^*D)\cdot\beta=D\cdot\alpha_0-\delta a>0.
	\]
	For \(a>0\), this follows from
	\eqref{eq:exterior-positive-threshold}; for \(a=0\), it follows from
	\eqref{eq:exterior-base-class-bound}. Thus \(\beta\) is an interior
	point of the dual pseudoeffective cone.
	
	By  \cite[Theorem~2.2]{BDPP}, this dual cone is the closed
	cone generated by strongly movable curves. A convex cone with nonempty
	interior has the same interior as its closure. Hence \(\beta\) lies in
	the cone generated by finitely many strongly movable classes. Since
	these generators and \(\beta\) are rational, the coefficients can be
	chosen rational. After clearing denominators in ample classes, we may
	therefore write
	\begin{equation}
		\beta=\sum_{\nu=1}^{N}\lambda_\nu(\mu_\nu)_*
		(A_{\nu,1}\cdots A_{\nu,\dim Y-1}),
		\qquad \lambda_\nu\in\mathbb Q_{>0},
		\label{eq:exterior-strongly-movable-expression}
	\end{equation}
	where \(\mu_\nu:\widetilde Y_\nu\to Y\) is projective and birational,
	\(\widetilde Y_\nu\) is smooth and projective, and every \(A_{\nu,j}\)
	is very ample. The defining equalities for \(\beta\) also give
	\(\pi_*\beta\equiv\alpha_0\) and \(\theta\cdot\beta=-\delta<0\).
\end{proof}

The following observation will ensure that the projection relation for
these curve classes survives reduction, even when the Picard number
of a fiber increases.

\begin{lemma}[{\cite[Proposition 1.1]{CTSk2013}}]
	\label{lem:exterior-numerical-one-cycles}
	Let \(Z\) be a smooth complex projective variety of dimension \(d\geq2\).
	A numerically trivial rational one-cycle on \(Z\) has zero rational
	Betti cycle class.
\end{lemma}
\begin{lemma}\label{lem:specialization-movable-witnesses}
	Let $\pi:Y\to X$ be a morphism of smooth complex projective
	varieties, with $d:=\dim X\geq2$ and $n:=\dim Y$.
	Let $H$ be an ample line bundle on $X$, and let
	$\theta\in\operatorname{Pic}(Y)\otimes\mathbb Q$.
	Suppose that
	\[
	\beta=\sum_{\nu=1}^N\lambda_\nu(\mu_\nu)_*
	(A_{\nu,1}\cdots A_{\nu,n-1})
	\in N_1(Y)_{\mathbb Q},
	\qquad \lambda_\nu\in\mathbb Q_{>0},
	\]
	where $\mu_\nu:\widetilde Y_\nu\to Y$ is projective and
	birational, $\widetilde Y_\nu$ is smooth and projective,
	and the $A_{\nu,j}$ are very ample divisors. Assume that
	\(
	\pi_*\beta\equiv H^{d-1}.
	\)
	Fix a smooth projective model of $(X,H)$ over
	$S=\operatorname{Spec}R$, where $R\subset\mathbb C$ is a
	finitely generated integral $\mathbb Z$-algebra.
	
	After enlarging $R$ inside $\mathbb C$ and shrinking the base,
	all these data admit models such that, on every geometric fiber,
	the class
	\[
	\beta_{\bar s}:=
	\sum_{\nu=1}^N\lambda_\nu(\mu_{\nu,\bar s})_*
	(A_{\nu,1,\bar s}\cdots A_{\nu,n-1,\bar s})
	\]
	satisfies
	\[
	(\pi_{\bar s})_*\beta_{\bar s}\equiv H_{\bar s}^{d-1},
	\qquad
	\theta_{\bar s}\cdot\beta_{\bar s}=\theta\cdot\beta,
	\qquad
	T\cdot\beta_{\bar s}\geq0
	\]
	for every effective divisor $T$ on $Y_{\bar s}$.
	The change of base can be chosen dominant and flat of finite
	presentation over a nonempty open subset of the original $S$.
	These choices can be made simultaneously for finitely many
	sets of data.
\end{lemma}

\begin{proof}
	Spread out the finitely many varieties, morphisms and line
	bundles involved. By generic flatness and shrinking, the new
	base can be taken dominant and flat over a nonempty open subset
	of the original base. We continue to denote it by $S$.
	After further shrinking, all the varieties are smooth and
	projective over $S$, the morphisms $\mu_{\nu,\bar s}$ are
	birational on every geometric fiber, and the divisors
	$A_{\nu,j}$ are relatively very ample.
	
	For every effective divisor $T$ on $Y_{\bar s}$, the projection
	formula gives
	\[
	T\cdot\beta_{\bar s}
	=\sum_{\nu=1}^N\lambda_\nu\,
	\mu_{\nu,\bar s}^*T\cdot
	A_{\nu,1,\bar s}\cdots A_{\nu,n-1,\bar s}\geq0.
	\]
	Indeed, each pullback $\mu_{\nu,\bar s}^*T$ is effective.
	Constancy of intersection numbers gives
	$\theta_{\bar s}\cdot\beta_{\bar s}=\theta\cdot\beta$.
	
	To prove the projection relation, consider the rational
	one-cycle class
	\[
	z:=
	\sum_{\nu=1}^N\lambda_\nu(\pi\circ\mu_\nu)_*
	(A_{\nu,1}\cdots A_{\nu,n-1})-H^{d-1}
	\]
	on $X$. It is numerically trivial by assumption, so
	Lemma~\ref{lem:exterior-numerical-one-cycles} gives
	$\operatorname{cl}_B(z)=0$.
	
	Fix a prime $\ell$ and shrink $S$ so that $\ell$ is invertible.
	Write $f:\mathscr X\to S$ for the structural morphism.
	Chern classes, cup products and proper cohomological
	pushforwards of the spread intersection expressions define
	a section of the lisse sheaf
	\(
	R^{2d-2}f_*\mathbb Q_\ell(d-1),
	\)
	whose value at each geometric point $\bar s$ is
	$\operatorname{cl}_\ell(z_{\bar s})$.
	These operations commute with geometric base change.
	Betti--\'etale comparison shows that this section vanishes
	at the complex geometric point. By smooth proper base change
	and connectedness of $S$, it vanishes on every geometric fiber
	(see \cite[Lemma~59.94.4, Tag~0GKD]{Stacks}).
	Pairing with the first Chern class of any line bundle on
	$X_{\bar s}$ proves
	\(
	(\pi_{\bar s})_*\beta_{\bar s}\equiv H_{\bar s}^{d-1}.
	\)
	This includes divisor classes on the special fiber that do not
	lift to characteristic zero.
	The same argument applies simultaneously to finitely many
	sets of data.
\end{proof}

The changes of model in
Lemma~\ref{lem:specialization-movable-witnesses}
preserve density of both bad loci considered here. Indeed, a flat
morphism of finite presentation is open, and failure of strong
semistability or strong geometric stability is invariant under
extension of algebraically closed fields.

\begin{proof}[Proof of Theorem~\ref{thm:pseudoeffective-exterior-power}]
	Suppose that none of the classes \(\theta_i\), \(1\leq i<r\), is
	pseudoeffective. Since \(\alpha_0=H^{d-1}\) is an interior movable
	class, openness of stability
	\cite[Theorem~3.3]{Greb-Kebekus-Peternell2016} gives an open neighborhood
	\(U\subset\operatorname{Mov}(X)\) of \(\alpha_0\) on which \(E\) is
	stable. The exterior powers \(G_i\) and all their symmetric powers
	are semistable throughout \(U\), since in characteristic zero they
	are direct summands of tensor powers of \(E\), which are semistable
	by \cite[Theorem~4.2 and Corollary~4.3]{Greb-Kebekus-Peternell2016}.
	Moreover,
	\[
	\frac{c_1(\operatorname{Sym}^aG_i)}
	{\operatorname{rk}(\operatorname{Sym}^aG_i)}
	=\frac{ai}{r}c_1(E).
	\]
	Thus Lemma~\ref{lem:movable-class-over-polarization} applies to
	\(\pi_i:\mathbb P(G_i)\to X\), with \(L=\mathcal O_{Y_i}(1)\)
	and \(\eta=(i/r)c_1(E)\), without requiring the exterior powers
	themselves to be stable. It gives rational curve classes \(\beta_i\)
	such that
	\begin{equation*}
		(\pi_i)_*\beta_i\equiv H^{d-1},
		\qquad \theta_i\cdot\beta_i=-\delta_i<0,
	\end{equation*}
	each with a finite expression as in
	\eqref{eq:exterior-strongly-movable-expression}.
	
	Apply Lemma~\ref{lem:specialization-movable-witnesses}
	simultaneously to the finitely many triples
	$(\pi_i,\theta_i,\beta_i)$.
	After changing the model and shrinking the base, preserving
	density of the bad locus, we obtain on every geometric fiber
	\begin{equation}
		T\cdot\beta_{i,\bar s}\geq0
		\label{eq:exterior-specialized-nonnegativity}
	\end{equation}
	for every effective divisor $T$ on $Y_{i,\bar s}$,
	\begin{equation}
		\theta_{i,\bar s}\cdot\beta_{i,\bar s}=-\delta_i,
		\qquad
		\theta_{i,\bar s}:=\xi_{i,\bar s}
		-\frac{i}{r}\pi_{i,\bar s}^*c_1(E_{\bar s}),
		\label{eq:exterior-specialized-negative-degree}
	\end{equation}
	and
	\begin{equation}
		(\pi_{i,\bar s})_*\beta_{i,\bar s}
		\equiv H_{\bar s}^{d-1}.
		\label{eq:exterior-specialized-projection}
	\end{equation}
	
	Choose a bad closed point \(s\in\Sigma^{\mathrm{ns}}(\mathscr E/S)\),
	of characteristic \(p>0\), and write
	\(F_{\bar s}:X_{\bar s}\to X_{\bar s}\) for absolute Frobenius.
	There are an integer \(m\geq0\) and a saturated subsheaf
	\[
	\mathscr F\subset F_{\bar s}^{m*}E_{\bar s},
	\qquad 1\leq i:=\operatorname{rk}\mathscr F<r,
	\]
	with slope at least that of
	\(F_{\bar s}^{m*}E_{\bar s}\). Set \(q:=p^m\) and
	\(L:=\det\mathscr F\). Smoothness of
	\(X_{\bar s}\) implies that \(L\) is a line bundle. The destabilizing
	inequality reads
	\begin{equation}
		\left(c_1(L)-\frac{qi}{r}c_1(E_{\bar s})\right)
		\cdot H_{\bar s}^{d-1}\geq0.
		\label{eq:exterior-destabilizing-degree}
	\end{equation}
	Taking determinants gives a nonzero morphism
	\(
	L\to F_{\bar s}^{m*}G_{i,\bar s}.
	\)
	It is initially defined on an open subset whose complement has
	codimension at least two and extends to all of \(X_{\bar s}\), because
	the source and target are locally free and \(X_{\bar s}\) is normal.
	
	The canonical Frobenius-power inclusion
	\(
	F_{\bar s}^{m*}G_{i,\bar s}
	\hookrightarrow\operatorname{Sym}^{q}G_{i,\bar s}
	\)
	is the identity when \(m=0\). Composing with the determinant morphism
	gives a nonzero section of
	\(\operatorname{Sym}^{q}G_{i,\bar s}\otimes L^{-1}\). Thus the divisor
	class
	\[
	\begin{aligned}
		q\xi_{i,\bar s}-\pi_{i,\bar s}^*c_1(L)
		&=q\theta_{i,\bar s}
		-\pi_{i,\bar s}^*\left(c_1(L)
		-\frac{qi}{r}c_1(E_{\bar s})\right)
	\end{aligned}
	\]
	is effective. However,
	\eqref{eq:exterior-specialized-negative-degree},
	\eqref{eq:exterior-specialized-projection} and
	\eqref{eq:exterior-destabilizing-degree} give
	\[
	\bigl(q\xi_{i,\bar s}-\pi_{i,\bar s}^*c_1(L)\bigr)
	\cdot\beta_{i,\bar s}
	=-q\delta_i
	-\left(c_1(L)-\frac{qi}{r}c_1(E_{\bar s})\right)
	\cdot H_{\bar s}^{d-1}<0,
	\]
	contradicting \eqref{eq:exterior-specialized-nonnegativity}.
	
	Therefore \(\theta_i\) is pseudoeffective for some \(1\leq i<r\).
	If \(\det E\) is pseudoeffective, so is
	\[
	\xi_i=\theta_i+\frac{i}{r}\pi_i^*c_1(E),
	\]
	which is equivalent to pseudoeffectivity of \(\bigwedge^iE\).
\end{proof}

\section{Positivity on homogeneous bundles of the holonomy group}
\label{ssb:sec-general-holonomy-positivity}

\subsection{Pseudoeffectivity and nefness}

We first isolate the base-locus argument that extends the
restriction and restricted base locus method of
\cite[Lemmas~2.13 and~3.4]{HoringPeternell2019}.

\begin{lemma}\label{ssb:lem-pseff-nef}
	Let $(X,H)$ be a smooth polarized complex projective surface, and let
	$\pi:Y\to X$ be a smooth projective morphism with geometrically integral
	fibers of dimension $d\geq1$.  Let $L$ be a relatively very ample line
	bundle.  Suppose that for all $m\geq0$ the sheaves
	$\mathcal V_m=\pi_*L^{\otimes m}$ are locally free, commute with base
	change, and satisfy $R^j\pi_*L^{\otimes m}=0$ for $j>0$.
	Suppose also that $\mathcal V_m$ is $H$-stable with
	$c_1(\mathcal V_m)\cdot H=0$ for every $m\geq1$.
	If $L$ is pseudoeffective, then $L$ is nef.
\end{lemma}

\begin{proof}
	Choose a Cartier divisor $\xi$ with $\mathcal O_Y(\xi)\simeq L$, and
	use the same symbol for its numerical class.  For a smooth curve
	$C\in|aH|$, set
	$Y_C:=\pi^{-1}(C)$, $L_C:=L|_{Y_C}$ and $\xi_C:=\xi|_{Y_C}$.
	
	\smallskip
	\noindent\emph{Restriction and strict positivity.}
	Choose $a$ sufficiently large so that
	$\mathcal V_1|_C$ is stable (this is possible by the strong restriction theorem \cite[Theorem 5.2]{Langer2004}).  
	Since 	$\mathcal V_1|_C$	 has degree zero, it is nef by \cite[Proposition 6.4.11]{ssb-Lazarsfeld2004II}.
	Consequently, the evaluation quotient shows that $\xi_C$ is nef.
	Moreover, if $R_m=\operatorname{rk}\mathcal V_m$, then
	\[
	\chi(Y_C,L_C^{\otimes m})
	=\chi(C,\mathcal V_m|_C)=R_m(1-g(C)).
	\]
	The polynomial $R_m$ has degree $d$.  Comparing leading coefficients in
	Riemann--Roch on the $(d+1)$-dimensional variety $Y_C$ gives
	\begin{equation}\label{ssb:eq-zero-top-on-curve}
		\xi_C^{d+1}=0.
	\end{equation}
	
	We claim that, for every proper integral subvariety $Z\subsetneq Y$ with
	$\dim\pi(Z)\geq1$, one can choose $a$ sufficiently large and $C$
	sufficiently general such that
	\begin{equation}\label{ssb:eq-strict-subvariety}
		\xi_C^{\dim Z-1}\cdot[Z\cap Y_C]>0.
	\end{equation}
	If $\pi(Z)$ is a curve, all components of a general $Z\cap Y_C$ are
	vertical over $C$, and relative ampleness proves the assertion.
	Suppose that $Z$ dominates $X$.  Fix $m\gg0$ such that
	\begin{equation}\label{ssb:eq-relative-serre}
		R^j\pi_*(\mathcal I_Z\otimes L^{\otimes m})=0\quad(j>0),
		\qquad
		\operatorname{rk}\pi_*(\mathcal I_Z\otimes L^{\otimes m})>0.
	\end{equation}
	The first assertion is relative Serre vanishing; the second follows
	because $Z$ is a proper subvariety on the generic fiber.
	Only after fixing this $m$, we choose $a$ so that both
	$\mathcal V_1|_C$ and $\mathcal V_m|_C$ are stable and
	$Z_C=Z\cap Y_C$ is reduced and pure-dimensional.
	
	The equation of $C$ is a nonzerodivisor on $\mathcal O_Z$.
	Applying the restriction sequence to
	$\mathcal I_Z\otimes L^{\otimes m}$, and using
	\eqref{ssb:eq-relative-serre}, gives the required base change for its
	zeroth direct image and vanishing of its higher direct images on $C$.
	Thus the ideal-sheaf sequence on $Y_C$ gives a surjection
	\(
	\mathcal V_m|_C\twoheadrightarrow
	\mathcal Q:=(\pi_C)_*(L_C^{\otimes m}|_{Z_C}).
	\)
	The torsion-free quotient $\mathcal Q_{\mathrm{tf}}$ has positive rank
	strictly smaller than $R_m$.  Every nonzero vector-bundle quotient of
	$\mathcal Q_{\mathrm{tf}}$ is a proper quotient of the stable degree-zero
	bundle $\mathcal V_m|_C$, and hence has positive degree.
	Therefore by \cite[Theorem 6.4.15]{ssb-Lazarsfeld2004II},
	$\mathcal Q_{\mathrm{tf}}$ is	ample.
	
	Let $W$ be an irreducible component of $Z_C$ dominating $C$.
	The torsion of $\mathcal Q$ maps to zero in the invertible sheaf
	$L_C^{\otimes m}|_W$, since $W$ is integral and dominates $C$.
	The evaluation therefore factors through a surjection
	\(
	(\pi_C|_W)^*\mathcal Q_{\mathrm{tf}}
	\twoheadrightarrow L_C^{\otimes m}|_W.
	\)
	The relative closed immersion defined by $L_C^{\otimes m}$ then factors
	as a closed immersion into $\mathbb P_C(\mathcal Q_{\mathrm{tf}})$.
	It follows that $L_C|_W$ is ample.  Vertical components are handled by
	relative ampleness.  This proves \eqref{ssb:eq-strict-subvariety}.

	\smallskip
	\noindent\emph{Eliminating positive-dimensional images of base loci.}
	Fix a very ample divisor $A$ on $Y$.  Recall that the restricted base locus
	\[
	\mathbf B_-(\xi)=\bigcup_{q\geq1}
	\mathbf B\left(\xi+\frac1q A\right)
	\]
	is	a countable union of closed algebraic subsets.  We use the standard
	asymptotic-multiplicity property of its irreducible components: if $Z$
	is such a component and $\xi$ is pseudoeffective, then
	\begin{equation}\label{ssb:eq-sigma-positive}
		\sigma_Z(\xi)>0.
	\end{equation}
	Here $\sigma_Z$ is the asymptotic order at the generic point of $Z$,
	computed by the exceptional divisor of the blowup of $Z$ if its
	codimension is greater than one (see
	\cite{ssb-EinLazarsfeldMustataNakamayePopa2006} and
	\cite[Lemma~6.12]{Druel2018}).  In particular, \eqref{ssb:eq-sigma-positive} implies
	that there is $b>0$ such that every effective $\mathbb Q$-divisor
	$D\sim_{\mathbb Q}\xi+\varepsilon A$, for sufficiently small positive
	rational $\varepsilon$, satisfies $\operatorname{mult}_Z D\geq b$.
	
	Suppose that some component of $\mathbf B_-(\xi)$ has
	positive-dimensional image in $X$, and choose one, denoted by $Z$, with
	smallest codimension $k$.  All components of smaller codimension map
	to points.  Choose a sufficiently ample very general $C\subset X$
	avoiding this countable collection of points and satisfying
	\eqref{ssb:eq-zero-top-on-curve} and
	\eqref{ssb:eq-strict-subvariety} for $Z$.  We may also require generic
	transversality for a countable sequence of linear systems used below.
	
	For positive rational $\varepsilon\to0$, choose a sufficiently divisible
	multiple of $\xi+\varepsilon A$ realizing its stable base locus, and
	choose $k$ general members, divided by that multiple.  Denote the
	resulting effective $\mathbb Q$-divisors by $D_1,\ldots,D_k$.
	Their restrictions to $Y_C$ meet properly in codimension $k$: a
	fixed component of smaller codimension would be contained in a
	component of $\mathbf B_-(\xi)$ whose point image was avoided by $C$.
	At the generic points of the reduced cycle $Z_C=Z\cap Y_C$, their
	orders are at least $b$.  The local intersection-multiplicity
	inequality in a regular local ring shows that
	\[
	(D_1|_{Y_C})\cdots(D_k|_{Y_C})-b^k[Z_C]
	\]
	is an effective real cycle.  Indeed, at the generic point of each
	component of $Z_C$, the $k$ local equations form a system of parameters
	in a regular local ring of dimension $k$, and their intersection
	multiplicity is at least the product of their orders.  Intersecting
	with the nef class
	$\xi_C^{d+1-k}$ gives
	\begin{equation}\label{ssb:eq-base-locus-inequality}
		(\xi_C+\varepsilon A_C)^k\cdot\xi_C^{d+1-k}
		\geq b^k\,\xi_C^{d+1-k}\cdot[Z_C]>0.
	\end{equation}
	For precision, one can first fix a countable sequence of these divisor
	choices and then choose the very general $C$ satisfying all the
	transversality conditions.  The integer governing stability of
	$\mathcal V_m|_C$ was already fixed above.
	Letting $\varepsilon$ tend to zero in
	\eqref{ssb:eq-base-locus-inequality} contradicts
	\eqref{ssb:eq-zero-top-on-curve}.
	Thus every irreducible component of $\mathbf B_-(\xi)$ maps to a point.
	
	\smallskip
	\noindent\emph{Nefness.}
	If an integral curve $\Gamma\subset Y$ had $\xi\cdot\Gamma<0$, then,
	for sufficiently small positive rational $\varepsilon$, every effective
	divisor in a sufficiently divisible multiple of $\xi+\varepsilon A$
	would contain $\Gamma$.  Hence $\Gamma$ would be contained in a component
	of $\mathbf B_-(\xi)$ and would map to a point of $X$.  Relative
	ampleness of $L$ excludes this.  Therefore $L$ is nef.
\end{proof}

\subsection{Closed orbits and arbitrary holonomy}

We use the convention that $\mathbb P(V)$ parametrizes one-dimensional
quotients of $V$. If a connected reductive group acts irreducibly on
$V$, its center acts by scalars and the unique closed orbit in
$\mathbb P(V)$ is also an orbit of its derived subgroup.

\begin{lemma}\label{ssb:lem-connected-orbit-nef}
Let $F$ be an $H$-stable vector bundle of rank $s\geq1$ on a smooth
complex projective surface $X$, with connected algebraic holonomy
$G_F$. Let $\pi:Z\to X$ be the bundle of closed $G_F$-orbits in
$\mathbb P_X(F)$, and set
\( L=\mathcal O_{\mathbb P_X(F)}(1)|_Z\) and
\(
\widehat L_F=L^{\otimes s}\otimes\pi^*(\det F)^{-1}.
\)
Then $\widehat L_F$ is nef if and only if
$\int_X\Delta(F)=0$.
\end{lemma}

\begin{proof}
For $s=1$ the assertion is immediate. Assume that $s\geq2$ and
write $Z=\mathcal P_{G_F}/P$, where $\mathcal P_{G_F}$ is the
holonomy reduction and $P\subset G_F$ is the stabilizer of a point
of the closed orbit. The character defining $\widehat L_F$ is
trivial on the center of $G_F$. Its restriction to the parabolic
subgroup in each simple factor of $G_F/Z(G_F)$ is nontrivial and
antidominant. Indeed, the defining representation of $G_F$ is
faithful and irreducible, so its highest weight is nonzero on every
simple factor of the derived group. The same holds for its dual,
which defines the closed orbit in our quotient convention.

If $\widehat L_F$ is nef, the equivalence of conditions (2) and (4)
in \cite[Theorem~4.3]{ssb-BiswasBruzzo2008} shows that every
pullback of $\mathcal P_{G_F}$ to a smooth projective curve is
semistable. Since the center acts by scalars on $F$, the associated
pullback of $F$ is semistable as well. It follows, by testing on
normalizations of curves in $\mathbb P_X(F)$, that
\(
\mathcal O_{\mathbb P_X(F)}(s)\otimes\pi_F^*(\det F)^{-1}
\)
is nef, where $\pi_F:\mathbb P_X(F)\to X$ is the projection.
Theorem~\ref{thm:FUY} gives $\int_X\Delta(F)=0$.
Conversely, that theorem gives nefness of the above line bundle
when $\int_X\Delta(F)=0$, and restriction to $Z$ proves the claim.
\end{proof}

\begin{theorem}\label{ssb:thm-holonomy-orbit-positivity} 
Let $(X,H)$ be a smooth polarized complex projective surface, and let
$E$ be an $H$-stable vector bundle of rank $r\geq2$. Fix $x\in X$ and
set
$V=E_x,$ $ G=\Hol_x(E)$ and $ D=[G^\circ,G^\circ].$
Choose an irreducible $G^\circ$-quotient $V\twoheadrightarrow W$, let
$O_0\subseteq\mathbb P(W)\subseteq\mathbb P(V)$ be its unique
closed $D$-orbit, and set $O=G\cdot O_0$. For the holonomy reduction
$\mathcal P_G$, form
\[
\pi:Y=\mathcal P_G\times^G O\longrightarrow X,
\qquad L=\mathcal O_{\mathbb P_X(E)}(1)|_Y,
\qquad
\widehat L=L^{\otimes r}\otimes\pi^*(\det E)^{-1}.
\]
Then $O$ is a finite disjoint union of closed $D$-orbits, and the
following conditions are equivalent:
\begin{enumerate}
\item $\widehat L$ is pseudoeffective;
\item $\widehat L$ is nef;
\item $\int_X\Delta(E)=0$;
\item $E$ is projectively unitary flat.
\end{enumerate}
\end{theorem}

\begin{proof}
The center of $G^\circ$ acts by scalars on $W$, so $O_0$ is also
the unique closed $G^\circ$-orbit in $\mathbb P(W)$. Since $G^\circ$
is normal and has finite index in $G$, its $G$-translates are
finitely many closed $G^\circ$-orbits. Distinct such orbits are
disjoint. Thus $Y\to X$ is smooth and projective, although its
fibers need not be connected.

Let  \(f:X'\to X\) be the finite \'etale Galois holonomy cover with group $G/G^\circ$
(see \cite[Lemma~40]{ssb-BalajiKollar2008}). Let us set $ H':=f^*H$.
Then the bundle $E':=f^*E$ is $H'$-polystable and its holonomy is $G^\circ$. The quotient $W$
therefore gives a stable direct summand $F$ of $E'$, of rank $s$,
with connected holonomy and
\begin{equation}\label{ssb:eq-orbit-equal-slopes}
\mu_{H'}(F)=\mu_{H'}(E').
\end{equation}
The component of $Y\times_X X'$ associated with $O_0$ is the
closed-orbit bundle
\(
\pi_0:Z\subseteq\mathbb P_{X'}(F)\to X'.
\)
Write $L_Z:=\mathcal O_{\mathbb P_{X'}(F)}(1)|_Z$ and
\(
\widehat L_Z:=L_Z^{\otimes r}\otimes\pi_0^*(\det E')^{-1}.
\)

Assume (1). Pullback preserves pseudoeffectivity, and $Z$ is an
open-and-closed component of the finite \'etale pullback
$Y\times_X X'$. Thus $\widehat L_Z$ is pseudoeffective.
For every $m\geq1$, the Borel--Weil theorem identifies
\(
\mathcal V_m=(\pi_0)_*\widehat L_Z^{\otimes m}
\)
with the $(rm)$th irreducible Cartan-power bundle, twisted by
$(\det E')^{-m}$. It is $H'$-stable, and
\[
\mu_{H'}(\mathcal V_m)
=rm\mu_{H'}(F)-m c_1(E')\cdot H'=0.
\]
The fiber calculations give local freeness, base change, and
$R^j(\pi_0)_*\widehat L_Z^{\otimes m}=0$ for $j>0$, by
the Borel--Weil theorem and vanishing for dominant weights
\cite[Part~II, Sections~2 and~4]{ssb-Jantzen2003}.
The base-change and vanishing assertions also hold for $m=0$, with
$\mathcal V_0=\mathcal O_{X'}$.
If $d:=\dim(Z/X')\geq1$, Lemma~\ref{ssb:lem-pseff-nef} proves
that $\widehat L_Z$ is nef. If $d=0$, then $s=1$ and $Z=X'$;
pseudoeffectivity and degree zero against $H'$ imply that
$\widehat L_Z$ is numerically trivial, hence nef.

Put
\[
\xi=c_1(L_Z),\qquad
\ell=c_1(\widehat L_Z)=r\xi-\pi_0^*c_1(E'),\qquad
B=\frac rs c_1(F)-c_1(E').
\]
By \eqref{ssb:eq-orbit-equal-slopes}, $B\cdot H'=0$. We claim that
\begin{equation}\label{ssb:eq-orbit-central-pushforward}
(\pi_0)_*\ell^{d+1}=(d+1)v r^d B,
\qquad
v:=\int_{Z_{x'}}\xi^d>0.
\end{equation}
For $k\geq0$, let $\mathcal A_k=(\pi_0)_*L_Z^{\otimes k}$ and
$R_k=\rk\mathcal A_k$. Since the holonomy of $F$ is connected
and its center acts by scalars, comparison of determinant characters
gives
\[
c_1(\mathcal A_k)=\frac{kR_k}{s}c_1(F),
\qquad
R_k=\frac{v}{d!}k^d+O(k^{d-1}).
\]
For $d\geq1$, comparison of the coefficients of $k^{d+1}$ in the
codimension-one part of relative Grothendieck--Riemann--Roch yields
\[
(\pi_0)_*\xi^{d+1}=(d+1)v\frac{c_1(F)}s.
\]
Together with $(\pi_0)_*\xi^d=v$, this proves
\eqref{ssb:eq-orbit-central-pushforward}. For $d=0$ the identity
is immediate.

The pushforward in \eqref{ssb:eq-orbit-central-pushforward} is nef:
its intersection with a curve $C\subset X'$ is the top intersection
of the nef class $\ell$ on the effective cycle $\pi_0^{-1}(C)$.
Consequently $B$ is nef. Since $B\cdot H'=0$ and $H'$ is ample,
the Hodge index theorem implies $B\equiv0$. Hence
\[
\ell\equiv\frac rs
\bigl(s\xi-\pi_0^*c_1(F)\bigr).
\]
Lemma~\ref{ssb:lem-connected-orbit-nef} now gives
$\int_{X'}\Delta(F)=0$.

The group $G/G^\circ$ permutes transitively the isomorphism classes
of irreducible $G^\circ$-constituents of $V$: otherwise the sum of
the isotypic components in one orbit would be a proper $G$-invariant
subspace of $V$, contrary to stability of $E$. Thus the deck group
permutes transitively the isomorphism classes of stable direct
summands of $E'$. Every such summand $F_a$ consequently satisfies
\[
\int_{X'}\Delta(F_a)=0,
\qquad
\frac{c_1(F_a)}{\rk F_a}\equiv\frac{c_1(E')}{r}.
\]
Indeed, these assertions for $F$ are preserved by deck
transformations, and repeated isomorphic summands give the same
classes. Additivity of the Chern character now yields
\(
\int_{X'}\Delta(E')=0.
\)
Since $\int_{X'}\Delta(E')=(\deg f)\int_X\Delta(E)$, this proves
(1)$\Rightarrow$(3).

The equivalence of (3) and (4) is Corollary~\ref{cor:DUY}.
If (3) holds, Theorem~\ref{thm:FUY} makes the normalized
tautological line bundle on $\mathbb P_X(E)$ nef; its restriction
to $Y$ is $\widehat L$, proving (2). Finally, nef line bundles on
projective varieties are pseudoeffective, so (2)$\Rightarrow$(1).
\end{proof}

\begin{remark}
If $D$ acts irreducibly on $V$, its unique closed orbit in
$\mathbb P(V)$ is $G$-invariant, so $Y\to X$ has geometrically
integral fibers and no finite union is needed.

The theorem does not assert the same discriminant criterion for
every homogeneous space $D/P$. A defining weight which vanishes on
a simple factor does not detect the curvature of that factor.
The corresponding criterion for a
general parabolic must detect all simple factors, as in
\cite[Theorem~4.3]{ssb-BiswasBruzzo2008}. 
\end{remark}

\subsection{The criterion of H\"oring and Peternell}

The preceding base-locus lemma recovers
\cite[Theorem~1.1]{HoringPeternell2019}. By
Lemma~\ref{ssb:lem-holonomy-square-dichotomy}, it is enough to assume
stability of the second reflexive symmetric power. For a reflexive sheaf, we
use pseudoeffectivity in the sense of
\cite[Definition~2.1]{HoringPeternell2019}: for every real $c>0$
there are positive integers $m,j$ with $m>cj$ such that
\(
H^0\bigl(X,\Sym^{[m]}E\otimes\mathcal O_X(jH)\bigr)\neq0.
\)

\begin{corollary}\label{ssb:cor-horing-peternell}
Let $X$ be a normal complex projective variety of dimension $n\geq2$
which is smooth in codimension two, and let $H$ be an ample Cartier
divisor. Let $E$ be a reflexive sheaf of positive rank such that
$\Sym^{[2]}E$ is $H$-stable and
\(
c_1(E)\cdot H^{n-1}=0.
\)
If $E$ is pseudoeffective, then
\[
c_1(E)^2\cdot H^{n-2}=c_2(E)\cdot H^{n-2}=0.
\]
If $X$ is smooth, then $E$ is a numerically flat vector bundle.
If $X$ is klt, there is a finite surjective Galois morphism
$\gamma:\widetilde X\to X$, \'etale in codimension one, such that
$\gamma^{[*]}E$ is a numerically flat vector bundle.
\end{corollary}

\begin{proof}
First suppose that $X$ is a smooth surface, so $E$ is locally free.
If $E$ has rank one, pseudoeffectivity and $c_1(E)\cdot H=0$
imply that $E$ is numerically trivial. Suppose that its rank $r$ is at
least two. By Lemma~\ref{ssb:lem-holonomy-square-dichotomy}, either
$E$ has finite projective holonomy or all $\Sym^mE$ are $H$-stable.
In the first case, the proof of
Lemma~\ref{ssb:lem-finite-projective-holonomy} gives a finite
\'etale cover $f:X'\to X$ such that $f^*E\simeq L^{\oplus r}$
for a line bundle $L$. Pulling back the sections witnessing
pseudoeffectivity shows that $L$ is pseudoeffective. Since
$c_1(L)\cdot f^*H=0$, it is numerically trivial. Thus $f^*E$, and
hence $E$, is numerically flat, proving the assertion in this case.
In the second case, apply Lemma~\ref{ssb:lem-pseff-nef} to
\(
\pi:\mathbb P_X(E)\to X  \) and $L=\mathcal O_{\mathbb P_X(E)}(1).$
Indeed, $\pi_*L^{\otimes m}=\Sym^mE$ is stable of slope zero for
$m\geq1$, and the required base change and higher-direct-image
vanishing follow from the projective bundle formula. The hypothesis
on $E$ says precisely that $L$ is pseudoeffective. Thus $L$ is nef,
so $E$ and $\det E$ are nef. Since $c_1(E)\cdot H=0$, the Hodge
index theorem gives $c_1(E)\equiv0$. Consequently $E$ is numerically
flat, and its Chern classes are numerically trivial.

In general, choose a sufficiently ample very general complete
intersection surface
\(
S\in |m_1H|\cap\cdots\cap |m_{n-2}H|
\)
contained in the smooth locally free locus of $E$, such that
$\Sym^2(E|_S)$ is $H|_S$-stable. This is possible by the restriction
theorem and the codimension-three bounds for the singular and
non-locally-free loci. We may also require $E|_S$ to be
pseudoeffective: choose one nonzero section witnessing the
definition above for each positive integer $c$, and take $S$ outside the
countably many proper closed conditions on which one of these
sections restricts to zero. Also
$c_1(E|_S)\cdot H|_S=0$. The surface case therefore gives
\(
c_1(E|_S)^2=c_2(E|_S)=0.
\)
Dividing these equalities by $m_1\cdots m_{n-2}$ proves the asserted
Chern number equalities on $X$.

If $X$ is smooth, these equalities give
$\Delta(E) H^{n-2}=0$. Theorem~\ref{thm:FUY} implies that
$E$ is locally free and its normalized tautological class is nef.
The Hodge index theorem, applied to
$c_1(E)\cdot H^{n-1}=c_1(E)^2\cdot H^{n-2}=0$, gives
$c_1(E)\equiv0$. Hence $E$ is nef with numerically trivial
determinant, and therefore numerically flat.

Finally suppose that $X$ is klt. By
\cite[Theorem~1.20]{Greb-Kebekus-Peternell-Duke2016}, the Chern
number equalities give a finite surjective Galois morphism
$\gamma:\widetilde X\to X$, \'etale in codimension one, such that
$F:=\gamma^{[*]}E$ is locally free and flat. In rank one this is
already numerically flat. Otherwise, the stability of $E$ and the
Chern number equalities allow us to apply
Theorem~\ref{klt-local-freeness}: $E|_{X_{\mathrm{reg}}}$ is
projectively unitary flat. Pulling back gives such a structure on
$F|_{\gamma^{-1}(X_{\mathrm{reg}})}$. This extends across the
codimension-two complement in $\widetilde X_{\mathrm{reg}}$.
Since $F$ is locally free, Lemmas~\ref{ssb:lem-unitary-extension}
and~\ref{ssb:lem-projective-flat-algebra}, applied to $\cEnd(F)$
with its multiplication and unit, extend this structure over all
of $\widetilde X$. Its determinant is flat, hence numerically trivial. Every pullback of $F$ to a smooth
projective curve is consequently semistable of degree zero, so
both $F$ and $F^*$ are nef. Thus $F$ is numerically flat.
\end{proof}

\begin{remark}
\cite[Example 6.9]{Druel2018} shows that Corollary \ref{ssb:cor-horing-peternell} fails if instead of stability of $\Sym^{[2]}E$ we assume only stability of $E$.
\end{remark}

\subsection{A pseudoeffectivity criterion for flatness of orthogonal bundles}

Theorem \ref{ssb:thm-holonomy-orbit-positivity} implies the following corollary:

\begin{corollary}\label{cor:orthogonal-flatness-criterion}
	Let $(X,H)$ be a smooth polarized complex projective surface,
	and let $E$ be an $H$-stable vector bundle of rank $r\geq3$.
	Assume that $\det E\simeq\cO_X$ and that $E$ is equipped with
	a nondegenerate symmetric form
	\(
	b:E\otimes E\longrightarrow\cO_X
	\)
	such that $\Hol_x(E)=\SO(E_x,b_x)$ for some $x\in X$.
	Let
	\(
	\pi:Q(E)\longrightarrow X
	\)
	be the smooth quadric bundle of isotropic one-dimensional
	quotients of $E$, where isotropy is taken with respect to the
	dual form on $E^*$. Then the following conditions are equivalent:
	\begin{enumerate}
		\item $-K_{Q(E)/X}$ is pseudoeffective;
		\item $-K_{Q(E)/X}$ is nef;
		\item $\int_Xc_2(E)=0$;
		\item $E$ is unitary flat.
	\end{enumerate}
\end{corollary}

\begin{proof}
	The inverse form defines a nowhere vanishing section
	\(
	b^{-1}\in H^0(X,\Sym^2E).
	\)
	Under the projective bundle formula, this gives a section of
	$\cO_{\mathbb P_X(E)}(2)$ whose zero scheme is $Q(E)$.
	Thus $Q(E)\subset\mathbb P_X(E)$ is a smooth relative quadric
	divisor. Write $\cO_{Q(E)}(1)$ for the restricted
	tautological line bundle. Since 
	$\omega_{\mathbb P_X(E)/X}
\simeq\cO_{\mathbb P_X(E)}(-r)$,
we have	by adjunction
	\begin{equation}\label{eq:orthogonal-relative-anticanonical}
		\cO_{Q(E)}(-K_{Q(E)/X})\simeq\cO_{Q(E)}(r-2).
	\end{equation}
	
	For $r\geq3$, the standard representation of $\SO_r$ is
	irreducible, and its unique closed orbit in the projective space
	of one-dimensional quotients is the quadric of isotropic
	quotients. Consequently $Q(E)$ is precisely the closed-orbit
	bundle occurring in Theorem~\ref{ssb:thm-holonomy-orbit-positivity}.
	Since $\det E\simeq\cO_X$, its normalized tautological line
	bundle is
$	\widehat L=\cO_{Q(E)}(r)$ 
and we have
$	\int_X\Delta(E)=2r\int_Xc_2(E).$
Since pseudoeffectivity and nefness of a line bundle are unchanged by
	taking a positive tensor power, Theorem~\ref{ssb:thm-holonomy-orbit-positivity} and
	\eqref{eq:orthogonal-relative-anticanonical} show that (1), (2)
	and (3) are equivalent, and that they are equivalent to
	projective unitary flatness of $E$.
	
	Finally,  since $\det E\simeq\cO_X$, projective unitary flatness
	of $E$ implies its unitary flatness. Conversely, a unitary flat bundle has
	vanishing rational Chern classes, so (4) implies (3).
\end{proof}

\begin{remark}
The orthogonal form already makes $E$ pseudoeffective, since
$b^{-1}$ is a nowhere-vanishing section of $\Sym^2E$. Moreover, for every $m\geq2$, multiplication by $b^{-1}$
	gives a proper nonzero subbundle
	\(
	\Sym^{m-2}E\lhook\joinrel\longrightarrow\Sym^mE
	\)
	of the same $H$-slope, namely zero. Thus none of the bundles
	$\Sym^mE$, $m\geq2$, is $H$-stable.
	The criterion above therefore detects flatness in a setting
	where pseudoeffectivity of $E$ is automatic and stability of
	its higher symmetric powers is unavailable. If
	$\int_Xc_2(E)>0$, it shows that
	$-K_{Q(E)/X}$ is not pseudoeffective.
\end{remark}

\section{The low rank theorems}

Theorem~\ref{ssb:thm-holonomy-orbit-positivity} implies the following theorem:

\begin{theorem}\label{ssb:prop-low-rank-holonomy}
Let $(X,H)$ be a smooth polarized complex projective surface, and let
$E$ be an $H$-stable vector bundle of rank $r\in\{2,3\}$.
If $\Sigma^{\mathrm{ns}}(E)$ is dense, then
$\int_X\Delta(E)=0$.
\end{theorem}

\begin{proof}
Put $G=\Hol_x(E)$ for a fixed $x\in X$ and $V=E_x$.
We may shrink the arithmetic base so that all geometric fibers of
$\mathscr E$ are stable. We will spread out finitely many additional
objects below. The dominant changes of model and shrinkings needed
for this preserve density, as explained after
Lemma~\ref{lem:specialization-movable-witnesses}.

We first record the consequence of the rank assumption. If
$F_{X_{\bar s}}^{e*}E_{\bar s}$ is not stable, then, for one of
$B=E$ or $B=E^*$, there is a saturated rank $1$ subsheaf $M\subset F_{X_{\bar s}}^{e*}B_{\bar s}$ such that
\[
\left(c_1(M)-\frac q r c_1(B_{\bar s})\right)\cdot H_{\bar s}
\geq0,
\]
where $q=p^e$.
For rank two this follows directly from nonstability. For rank three, a saturated
rank two subsheaf of slope at least that of the bundle gives the
required subsheaf in the dual by dualizing its rank one torsion-free
quotient. 

Suppose first that $G^\circ$ acts reducibly on $V$. Since $G$ acts
irreducibly, it permutes transitively the isomorphism classes of
irreducible $G^\circ$-constituents, and their dimensions and
multiplicities are equal. As $r$ is prime, all these constituents
are one-dimensional. On the finite \'etale Galois holonomy cover
\( f:X'\to X\)
of \cite[Lemma~40]{ssb-BalajiKollar2008}, the bundle $f^*E$ is
polystable and therefore a direct sum of line bundles of the same
$f^*H$-slope.
Either all summands are isomorphic, in which case
$\int_X\Delta(E)=0$ immediately, or
\[
f^*E=L_1\oplus\cdots\oplus L_r,
\]
where the deck group permutes the summands transitively.
Spread out the cover, the decomposition and this action, and shrink
so that the fibers of $X'$ are geometrically integral and all
displayed data remain defined.

Choose a bad geometric fiber and a line bundle $M$ as above. For
$B=E^*$ use the dual line bundles in the decomposition. Pullback
gives a nonzero equivariant morphism
\[
f_{\bar s}^*M\longrightarrow
\bigoplus_{j=1}^r L_{j,\bar s}^{\otimes q}.
\]
Its projection to every summand is nonzero: the source and the
morphism descend to $X_{\bar s}$, and the deck group permutes the
summands transitively. All target summands have the same
$f_{\bar s}^*H_{\bar s}$-degree, whereas the degree of the source
is at least this common degree. A nonzero morphism between these
line bundles consequently has a zero divisor of nonpositive ample
degree. It must be an isomorphism. Hence
\[
F_{X'_{\bar s}}^{e*}f_{\bar s}^*B_{\bar s}
\simeq(f_{\bar s}^*M)^{\oplus r},
\qquad
q^2(\deg f)\int_{X_{\bar s}}\Delta(E_{\bar s})=0.
\]
Constancy of the Chern numbers proves the assertion in this case.

Suppose now that $G^\circ$ acts irreducibly, and assume for a
contradiction that $\int_X\Delta(E)\ne0$. The derived group
$[G^\circ,G^\circ]$ in its given representation is
$\mathrm{SL}_2$ in rank two, and either $\mathrm{SL}_3$ or
$\mathrm{SO}_3$ in rank three; see
\cite[paragraph~45]{ssb-BalajiKollar2008}. Indeed, these are the semisimple
groups with a faithful irreducible representation of dimension two
or three, up to the isogeny defining the representation. Consequently,
for each $B\in\{E,E^*\}$, its closed-orbit bundle
\[
\pi_B:Y_B\subseteq\mathbb P_X(B)\longrightarrow X,
\qquad
L_B=\mathcal O_{\mathbb P_X(B)}(1)|_{Y_B},
\]
has fibers $\mathbb P^1$, $\mathbb P^2$, or smooth plane conics.
These fibers are geometrically integral and span the ambient
projective space. The unique closed $G^\circ$-orbit is invariant
under $G$, so no finite union of orbits is needed.
By Theorem~\ref{ssb:thm-holonomy-orbit-positivity}, the class 
\[
\theta_B:=c_1(L_B)-\frac1r\pi_B^*c_1(B)
\]
is not pseudoeffective for either choice of $B$.

We check the hypotheses of
Lemma~\ref{lem:movable-class-over-polarization}. By openness of slope
stability in the ample cone (see \cite[Section~4.C]{Huybrechts-Lehn2010}
or \cite[Theorem~3.3]{Greb-Kebekus-Peternell2016}), there is an open
neighborhood $U$ of $H$ on which $E$, and hence $E^*$, is stable.
On a smooth surface we view $U$ also as an open
set of movable curve classes. For every $a\geq0$, the 
bundle
\(
V_{B,a}:=(\pi_B)_*L_B^{\otimes a}
\)
is a direct summand of $\operatorname{Sym}^aB$: restriction of
homogeneous polynomials is surjective on these fibers, and the
holonomy representation is completely reducible. The center of
$G^\circ$ acts by scalars on $B_x$, and comparison of determinant
characters gives
\[
\frac{c_1(V_{B,a})}{\operatorname{rk}V_{B,a}}
=\frac a r c_1(B)
\quad\text{in }N^1(X)_{\mathbb Q}.
\]
Characters trivial on $G^\circ$ have finite order, which does not
affect this equality. Since tensor powers of $B$ are semistable
throughout $U$, the direct summands $V_{B,a}$ are semistable there
as well. Finally,
\[
\operatorname{Pic}(Y_B)\otimes\mathbb Q
=\mathbb Q[L_B]\oplus
\pi_B^*\bigl(\operatorname{Pic}(X)\otimes\mathbb Q\bigr).
\]
For projective-space fibers this is the usual projective-bundle
formula. For conic fibers, $L_B$ has degree two; the square of any
line bundle on $Y_B$, after twisting by an integral power of $L_B$,
is trivial on the fibers and therefore descends to $X$. This proves
the required rational Picard-group decomposition also in that case.

The lemma supplies rational classes $\beta_B$, each a finite positive
rational combination of strongly movable classes, such that
\[
(\pi_B)_*\beta_B\equiv H,
\qquad \theta_B\cdot\beta_B=-\delta_B<0.
\]
Apply Lemma~\ref{lem:specialization-movable-witnesses}
simultaneously to $(\pi_B,\theta_B,\beta_B)$ for
$B\in\{E,E^*\}$. After changing the model and shrinking the
base, preserving density of the bad locus, we obtain
\[
(\pi_{B,\bar s})_*\beta_{B,\bar s}\equiv H_{\bar s},
\qquad
\theta_{B,\bar s}\cdot\beta_{B,\bar s}=-\delta_B,
\qquad
T\cdot\beta_{B,\bar s}\geq0
\]
for every effective divisor $T$ on $Y_{B,\bar s}$.
After further shrinking, the orbit fibers remain smooth and
linearly nondegenerate in the corresponding projective spaces.

Choose a bad geometric fiber and a line
$M\to F_{X_{\bar s}}^{e*}B_{\bar s}$ as above. On
\[
Y_{B,\bar s}^{(q)}
:=Y_{B,\bar s}\times_{X_{\bar s},F_{X_{\bar s}}^e}X_{\bar s}
\subseteq\mathbb P_{X_{\bar s}}(F_{X_{\bar s}}^{e*}B_{\bar s}),
\]
the universal quotient gives a nonzero section of
$L_B^{(q)}\otimes\pi^*M^{-1}$. It is nonzero because the orbit
in the generic fiber spans the ambient projective space, so no
nonzero vector of $F_{X_{\bar s}}^{e*}B_{\bar s}$ vanishes on the
whole orbit. Pulling back by the surjective relative Frobenius
\[
\Phi_e:Y_{B,\bar s}\longrightarrow Y_{B,\bar s}^{(q)},
\qquad
\Phi_e^*L_B^{(q)}\simeq L_{B,\bar s}^{\otimes q},
\]
therefore gives an effective divisor $D_{\bar s}$, possibly zero,
with
\[
[D_{\bar s}]
=q\theta_{B,\bar s}
-\pi_{B,\bar s}^*\left(c_1(M)-\frac q r c_1(B_{\bar s})\right).
\]
It follows that
\[
0\leq D_{\bar s}\cdot\beta_{B,\bar s}
=-q\delta_B
-\left(c_1(M)-\frac q r c_1(B_{\bar s})\right)\cdot H_{\bar s}
<0,
\]
a contradiction. 
\end{proof}

The above theorem recovers Shepherd--Barron's theorem, at least in the surface case (for higher dimensional versions see Theorems \ref{cor:rank-at-most-three} and \ref{cor:rank-at-most-three-normal}).

\begin{corollary}[Shepherd--Barron, {\cite[Theorem 5]{ShepherdBarron1998}}]
	Let $(X,H)$ be a smooth polarized complex projective surface, and let
	$E$ be an $H$-semistable vector bundle of rank two.
If  $\Sigma^{\mathrm{nss}}(E)$ is dense, then $E$ is $H$-stable and $\int_X\Delta(E)=0$.
If  $\Sigma^{\mathrm{ns}}(E)$ is dense, then either $E$ is strictly $H$-semistable or it is 
$H$-stable with $\int_X\Delta(E)=0$.
\end{corollary}

\begin{proof}
If $E$ is $H$-stable, Theorem~\ref{ssb:prop-low-rank-holonomy}
gives $\int_X\Delta(E)=0$ under the second hypothesis, and hence
also under the first, since
$\Sigma^{\mathrm{nss}}(E)\subseteq\Sigma^{\mathrm{ns}}(E)$.

Suppose that $E$ is strictly $H$-semistable. The second assertion is
then immediate. To exclude this case under the first hypothesis,
spread an equal-slope Jordan--H\"older sequence after a dominant
base change and shrinking the model:
\[
0\longrightarrow\mathscr M\longrightarrow\mathscr E
\longrightarrow\mathscr N\longrightarrow0.
\]
We may arrange that on every geometric fiber this sequence is exact
and its end terms are rank one torsion-free sheaves of equal slope.
The base change can be chosen inside $\mathbb C$ and made flat over
a nonempty open subset of the original base, so the chosen complex
fiber and density of the geometric bad locus are preserved.
Frobenius is flat on each smooth fiber. Hence every Frobenius
pullback of the displayed sequence is an extension of rank one
torsion-free sheaves of equal slope, and is semistable. Thus all
these fibers of $\mathscr E$ are strongly semistable, contradicting
density of $\Sigma^{\mathrm{nss}}(E)$.
\end{proof}

\medskip

Suppose that the base $S=\Spec R$ of a model \(
(\mathscr X,\mathscr H,\mathscr E)\)
 of $(X,H, E)$ has dimension greater than one.  A proper horizontal
closed subscheme of $S$ may meet the fiber over every rational prime.
Consequently, the following conditions are not equivalent:
\begin{enumerate}[(i)]
	\item bad closed points exist over infinitely many rational primes;
	\item the bad closed points are Zariski dense in $S$.
\end{enumerate}
This distinction is noted in \cite[Section~6, footnote~3]{LangerPositivitySemistability}.
If $S$ is an open subscheme of the spectrum of the integers of a number field, an infinite set of closed points is dense, so the  ``infinitely many places'' formulation in {\cite[Theorem 5]{ShepherdBarron1998}} is sufficient. 
The following example shows that in general using an infinite set of primes is not sufficient.

\begin{example}
Let \( R=\bZ[t]\hookrightarrow\bC\), \(t\to\alpha,\)
where $\alpha$ is transcendental, and set
$
 \mathscr X:=\bP^1_R\times_R\bP^1_R$ and $
 \mathscr H:=\cO_{\mathscr X}(1,1).$
Let us also set $\cA:=\cO _{\cX}(-1,1)$ and $\cB:=\cO _{\cX}(1,0)$.  Then
\(
 \Ext^1_{\mathscr X}(\cB,\cA)
 =H^1(\mathscr X,\cO(-2,1))\simeq R^{\oplus2}.
\)
Choose a basis vector $e$ and define $\mathscr E$ by the extension of class
$te$:
\begin{equation}
 0\longrightarrow \cA\longrightarrow\mathscr E
 \longrightarrow \cB\longrightarrow0.
 \label{eq:SB-model-extension}
\end{equation}
Then $E= \mathscr E _{\bC}$ is $H$-stable.  Indeed, suppose that  $\cO(a,b)\hookrightarrow E$ destabilizes $E$ so
that \(a+b\geq\mu_H(E)=1/2\).  If the composition of  $\cO(a,b)\hookrightarrow E$ with
$E\to B:=\cB_{\bC}$ vanishes, a map to $A:=\cA_{\bC}$ forces $a\leq-1$ and $b\leq1$, whence
$a+b\leq0$.  Otherwise, a map to $B$ forces $a\leq 1$ and $b\leq0$ so we have $(a,b)=(1,0)$.  Such a map $B\to E$ would
split \eqref{eq:SB-model-extension}, contrary to
\(\alpha e\neq0\).

One can also easily see that $ \int_X \Delta(E)=4\neq0.$
For every rational prime $p$, however, reduction at
\(\mathfrak m_p=(p,t)\) kills the extension class.  Thus
\(
 \mathscr E_{\mathfrak m_p}\simeq
 \cO(-1,1)\oplus\cO(1,0),
\)
and $\cB_{\mathfrak m_p}$ destabilizes it. So the set  $\Sigma ^{\mathrm{nss}}(\cE/\Spec R)$ is infinite and every rational prime occurs as the characteristic of a residue field of a point $s\in \Sigma ^{\mathrm{nss}}(\cE/\Spec R) $.

Note that  all bad points $\Sigma ^{\mathrm{nss}}(\cE/\Spec R)$ lie on the
proper divisor $V(t)\subset\Spec\bZ[t]$. Indeed, outside of this divisor every extension class is nonzero
and the same argument as above proves geometric stability of $\cE _s$. Since closed fibers are 
isomorphic to $\bP^1\times\bP^1$, such bundles are automatically strongly geometrically stable.
Indeed, $\Omega^1_{\mathbb P^1\times\mathbb P^1}
\simeq\mathcal O(-2,0)\oplus\mathcal O(0,-2)$, so
Proposition~\ref{prop:Cartier-descent} shows that Frobenius
pullback preserves stability.
So if we choose a different model replacing $R$ by $R[t^{-1}]$ then the set of bad points $\Sigma ^{\mathrm{nss}}(\cE/\Spec R[t^{-1}])$ is empty.
\end{example}

\section{The arithmetic surface theorem}

\begin{theorem}\label{ssb:thm-main}
	Let $(X,H)$ be a normal polarized complex projective surface,
	and let $E$ be a vector bundle of rank $r\geq2$ such that $\Sym^2E$ is
	$H$-stable. If $\Sigma^{\mathrm{ns}}(E)$ is dense, then	\(	\int_X\Delta(E)=0.	\)
\end{theorem}

Fix an integral finitely generated $\mathbb Z$-subalgebra
$R\subset\mathbb C$, set $S=\operatorname{Spec}R$, and choose a model
\(
(\mathscr X,\mathscr H,\mathscr E)\to S
\)
of $(X,H,E)$. After shrinking, assume that $\mathscr X/S$ is flat
projective with geometrically normal and geometrically integral surface fibers,
$\mathscr H$ is relatively ample, and $\mathscr E$ is locally free.
For a geometric point $\bar s$ over a closed point $s\in S$ we write
$X_{\bar s}=\mathscr X_{\bar s}$, $H_{\bar s}=\mathscr H_{\bar s}$, and 
$E_{\bar s}=\mathscr E_{\bar s}$.  Under the hypotheses above, $E$ is geometrically
stable, so after shrinking $S$ every $E_{\bar s}$ is geometrically stable. 

We first prove the theorem when $X$ is smooth. The reduction of the
normal case to this case is given in
Subsection~\ref{ssb:subsec-normal-surfaces}. Until that subsection, we
assume that $X$ is smooth and choose the model with $\mathscr X/S$
smooth.

After separating the finite-projective-holonomy case, we use the
ordinary or isotropic Grassmann bundles selected by the
algebraic holonomy of $E$. On each of them, with Pl\"ucker line bundle
$L_i$ and projection $\pi_i$, we consider 
$\widehat L_i=L_i^{\otimes r}\otimes\pi_i^*(\det E)^{-i}.$
This normalization removes the scalar character and makes the argument
invariant under tensoring $E$ by a line bundle. Applying the general
criterion of Section~\ref{ssb:sec-general-holonomy-positivity} to the
homogeneous coordinate bundles gives
Corollary~\ref{ssb:thm-plucker-positivity}: pseudoeffectivity of
$\widehat L_i$ forces its nefness and vanishing of the discriminant.
We then adapt the proof of
Theorem~\ref{thm:pseudoeffective-exterior-power} to these Grassmann
bundles. Density of the bad reductions forces one of the normalized
Pl\"ucker line bundles to be pseudoeffective, proving theorem in the smooth case.

\subsection{Algebraic holonomy and normalized homogeneous bundles}
\label{ssb:sec-holonomy}

Throughout this subsection $X$ is a smooth polarized complex projective
surface, $r=\rk E\geq2$, and $\Sym^6E$ is $H$-stable.  We also fix $x\in X(\bC)$ and set
$V:=E_x$ and $G:=\Hol_x(E)$.
By Lemma \ref{ssb:prop-holonomy} $E$ is $H$-stable
and we have two possibilities for $D=[G^\circ, G^\circ]$. Namely, we have $D=\SL(V)$ or $D=\Sp(V,b)$.

In the special-linear case, for $1\leq i<r$, let
\(
\pi_i:Y_i=\Gr(E, i)\to X
\)
parametrize locally free rank $i$ quotients of $E$, let $Q_i$ be
the universal quotient, and put $L_i=\det Q_i$.
Choose a subspace $K_i\subset V$ of codimension $i$, and define
\[
P_i:=\operatorname{Stab}_{D}(K_i)
=\{g\in\SL(V)\mid g(K_i)=K_i\}.
\]
Thus $P_i$ is the maximal parabolic subgroup stabilizing the point
of $\Gr(V, i)$ represented by the quotient $V\twoheadrightarrow V/K_i$,
and
\(
D/P_i\simeq\Gr(V, i).
\)

In the symplectic case, write $r=2n$.
For $1\leq i\leq n$, let
\(
\pi_i:Y_i\to X
\)
be the closed relative isotropic Grassmannian in $\Gr(E, i)$
parametrizing rank $i$ quotients whose dual subbundles are isotropic
for the dual alternating form
\(
\beta^*:\bigwedge\nolimits^2E^*\to N^{-1}.
\)
Writing $Q_i$ also for the restriction of the universal quotient
to $Y_i$, this condition is
\[
\left.\pi_i^*\beta^*\right|_{\bigwedge^2Q_i^*}=0,
\qquad
Q_i^*\subseteq\pi_i^*E^*.
\]
Put $L_i=\det Q_i$.
Choose an $i$-dimensional isotropic subspace
$U_i\subset V^*$ for the dual form $b^*$, and define
\(
P_i:=\operatorname{Stab}_{D}(U_i),
\)
where $D=\Sp(V,b)$ acts on $V^*$ by the contragredient
representation.
Then $P_i$ is a maximal parabolic subgroup and
\(
D/P_i\simeq\operatorname{IGr}(V^*,b^*, i).
\)
In both cases, the conjugacy class of $P_i$ is independent of
the chosen reference subspace.

The morphisms $\pi_i$ are Zariski-locally trivial smooth projective
fibrations with geometrically integral homogeneous fibers $D/P_i$.
The relative Pl\"ucker embedding restricts to a closed immersion
\(
\iota_i:Y_i\hookrightarrow
\mathbb P_X\!\left(\bigwedge\nolimits^iE\right)
\)
such that
\(L_i\simeq\iota_i^* \mathcal O_{\mathbb P_X(\bigwedge^iE)}(1).\)
In particular, $L_i$ is $\pi_i$-very ample.
More generally, $L_i^{\otimes a}\otimes\pi_i^*M$ is
$\pi_i$-very ample for every integer $a\geq1$ and every line
bundle $M$ on $X$. In particular, the line bundles
$\widehat L_i:=L_i^{\otimes r}\otimes\pi_i^*(\det E)^{-i}$ are also
relatively very ample.

\begin{lemma}\label{ssb:lem-coordinate-bundles}
	Let us fix $i$. Then for  every $k\geq0$, the sheaf
	\(
	\mathcal V_{i,k}=(\pi_i)_*L_i^{\otimes k}
	\)
	is locally free, commutes with base change on $X$, and has
	$R^j(\pi_i)_*L_i^{\otimes k}=0$ for $j>0$.
	For $k\geq1$ it is $H$-stable and
	\begin{equation}\label{ssb:eq-c1-coordinate}
		c_1(\mathcal V_{i,k})=
		\frac{ik\rk\mathcal V_{i,k}}r c_1(E)
	\end{equation}
	in $N^1(X)_{\mathbb Q}$.	
	The normalized bundles
	\begin{equation}\label{ssb:eq-normalized-coordinate}
		\mathcal W_{i,m}:=(\pi_i)_*\widehat L_i^{\otimes m}
		=\mathcal V_{i,rm}\otimes(\det E)^{-im}
	\end{equation}
	are $H$-stable with trivial determinant for $m\geq1$. They satisfy
	the same local freeness, base-change, and higher-direct-image
	vanishing statements. Their evaluation maps give closed immersions
	\(
	Y_i\hookrightarrow\mathbb P_X(\mathcal W_{i,m})\) with \(
	\widehat L_i^{\otimes m}
	=\mathcal O_{\mathbb P_X(\mathcal W_{i,m})}(1)|_{Y_i}.
	\)
\end{lemma}

\begin{proof}
	The Borel--Weil theorem and vanishing for dominant weights
	identify the $D$-module
	$H^0((Y_i)_x,(L_i)_x^{\otimes k})$ with the irreducible
	representation of highest weight $k\omega_i$ (see
	\cite[Part~II, Sections~2 and~4]{ssb-Jantzen2003}).
	In the special-linear case this is the Schur module $\mathbb S_{(k^i)}V$.
	In the symplectic case it is the irreducible representation of highest weight \(k\omega_i\) (see Subsection \ref{subsection:Cartan}). The fiber calculations give 
	local freeness, base change, and higher-direct-image vanishing.
	Relative very ampleness gives the evaluation assertions.
	
	The fiber representation of $\mathcal V_{i,k}$ is irreducible under $D$, hence under $G$,
	and it occurs as a summand of $V^{\otimes ik}$. Tensor polystability and
	the holonomy correspondence therefore give stability of
	$\mathcal V_{i,k}$ (see Theorem \ref{Balaji-Kollar}). The two characters
	$(\det\mathcal V_{i,k})^{\otimes r}$ and
	$(\det E)^{\otimes ik\rk\mathcal V_{i,k}}$ agree: both are trivial on $D$ and they  have
	the same weight on scalar matrices. This proves
	\eqref{ssb:eq-c1-coordinate}.
	The scalar $t\operatorname{id}_V$ acts on $\mathcal V_{i,rm}$ with
	weight $irm$, canceled by the twist in
	\eqref{ssb:eq-normalized-coordinate}. Thus the normalized
	representation is trivial on $Z$. Its determinant is also trivial
	on the semisimple group $D$, hence on $G=Z\cdot D$.
	Consequently $\det\mathcal W_{i,m}\simeq\mathcal O_X$.
	Twisting preserves stability, and the remaining assertions follow by
	the projection formula.
\end{proof}

\subsection{Positivity of normalized Plücker line bundles}

\begin{corollary}\label{ssb:thm-plucker-positivity}
	Assume that $\Sym^6E$ is $H$-stable. For each of the ordinary or
	isotropic Grassmann bundles selected above, the following conditions
	are equivalent:
	\begin{enumerate}
		\item $	\widehat L_i$ is pseudoeffective,
		\item $	\widehat L_i$ is nef,
		\item $	\int_X\Delta(E)=0.$
	\end{enumerate}
	
	In particular, if $\int_X\Delta(E)>0$, all the relevant normalized Pl\"ucker line bundles are non-pseudoeffective.
\end{corollary}
\begin{proof}
	Fix $i$ and put $\mathcal W=\mathcal W_{i,1}$ and
	$R=\rk\mathcal W$. By Lemma~\ref{ssb:lem-coordinate-bundles},
	$\mathcal W$ is $H$-stable with trivial determinant, and its
	evaluation map embeds $Y_i$ into $\mathbb P_X(\mathcal W)$ with
	tautological line bundle $\widehat L_i$. The holonomy of
	$\mathcal W$ is the image of $G$, which equals the image of $D$
	because the normalized representation is trivial on $Z$.
	The fiber of $Y_i$ is its closed orbit in
	$\mathbb P(\mathcal W_x)$.
	Theorem~\ref{ssb:thm-holonomy-orbit-positivity}, applied to
	$\mathcal W$, therefore gives the equivalence of pseudoeffectivity
	and nefness of $\widehat L_i^{\otimes R}$.
	Since both properties are unchanged by taking a positive tensor
	power, conditions~(1) and~(2) are equivalent.
	
	If $\widehat L_i$ is nef, Proposition~\ref{ssb:prop-top-intersection}
	gives
	\[
	0\leq
	\int_{Y_i}c_1(\widehat L_i)^{d_i+2}
	=-r^{d_i+2}\kappa_i\int_X\Delta(E).
	\]
	Since $\kappa_i>0$, Bogomolov's inequality implies
	$\int_X\Delta(E)=0$. Thus~(2) implies~(3).
	
	Conversely, suppose that $\int_X\Delta(E)=0$.
	By Corollary~\ref{cor:DUY}, the bundle $E$ is projectively
	unitary flat. Since the normalized representation defining
	$\mathcal W$ is trivial on scalars, the induced connection
	on $\mathcal W$ is unitary flat. Hence $\mathcal W$ is nef,
	and the evaluation quotient
	\(
	\pi_i^*\mathcal W\twoheadrightarrow\widehat L_i
	\)
	shows that $\widehat L_i$ is nef. This proves~(3)$\Rightarrow$(2)
	and completes the proof.
\end{proof}

\begin{remark}
	For $i=1$ we have
	\[
	\int_{\mathbb P_X(E)}
	\bigl(r\,c_1(\mathcal O(1))-\pi^*c_1(E)\bigr)^{r+1}
	=-\frac{r^r}{2}\int_X\Delta(E).
	\]
	When $c_1(E)\equiv0$, one has
	$c_1(\widehat L_i)\equiv r\,c_1(L_i)$ and
	$\Delta(E)=2r\,c_2(E)$ numerically. Thus
	Corollary~\ref{ssb:thm-plucker-positivity} includes the equivalence of
	pseudoeffectivity of $L_i$, nefness of $L_i$, and
	$\int_Xc_2(E)=0$ in the earlier degree-zero setting.
\end{remark}

\subsection{Pseudoeffectivity of a normalized Pl\"ucker line bundle}

We first record the linear algebra needed to treat Frobenius pullbacks
which are semistable but not stable.

\begin{lemma}\label{ssb:lem-transverse-isotropic-quotient}
	Let $V$ be a symplectic vector space of dimension $2n$ over an
	algebraically closed field of characteristic different from two.
	For every subspace $A\subset V$ of dimension $1\leq i\leq n$, there is a
	rank $i$ quotient $V\twoheadrightarrow Q$ whose restriction to $A$
	is an isomorphism and whose dual $Q^*\subset V^*$ is isotropic.
\end{lemma}

\begin{proof}
	The Lagrangian Grassmannian of $V$ has dimension $n(n+1)/2$.
	The incidence variety of pairs $(\ell,K)$ with
	$\ell\subset A$ a line and $K$ Lagrangian containing $\ell$ has
	dimension
	\[
	i-1+\frac{n(n-1)}2<\frac{n(n+1)}2.
	\]
	Thus there is a Lagrangian $K$ with $K\cap A=0$.
	Choose a quotient $V/K\twoheadrightarrow Q$ of dimension $i$
	which is an isomorphism on the image of $A$. Since
	$(V/K)^*\subset V^*$ is isotropic, so is $Q^*$.
\end{proof}

The following is the Grassmannian version of
Theorem~\ref{thm:pseudoeffective-exterior-power} needed here.

\begin{proposition}\label{ssb:prop-grassmann-pseudoeffectivity}
	Assume that $\Sym^6E$ is $H$-stable and that
	$\Sigma^{\mathrm{ns}}(E)$ is dense. Then $\widehat L_i$ is
	pseudoeffective for at least one of the ordinary or isotropic
	Grassmann bundles defined in Subsection~\ref{ssb:sec-holonomy}.
\end{proposition}

\begin{proof}
	Suppose that none of the $\widehat L_i$ is pseudoeffective, and put
	\[
	\theta_i:=c_1(L_i)-\frac{i}{r}\pi_i^*c_1(E)
	=\frac1r c_1(\widehat L_i).
	\]
	Since $\pi_i$ is locally trivial and the Picard group of each fiber
	is generated by the restriction of $L_i$, we have
	\(
	\operatorname{Pic}(Y_i)=\mathbb Z[L_i]
	\oplus\pi_i^*\operatorname{Pic}(X).
	\)
	By openness of slope stability in the ample cone (see \cite[4.C]{Huybrechts-Lehn2010} or \cite[Theorem~3.3]{Greb-Kebekus-Peternell2016}), there is an
	open neighborhood $U\subset\operatorname{Amp}(X)$ of $H$ such
	that $E$ is stable with respect to every class in $U$. 
	As in the proof of Lemma~\ref{ssb:lem-coordinate-bundles}, each $\mathcal V_{i,a}=(\pi_i)_*L_i^{\otimes a}$ is a direct
	summand of $E^{\otimes ia}$. For every $\alpha\in U$ and $a\geq0$, the tensor power
	$E^{\otimes ia}$ is $\alpha$-polystable. Hence
	\eqref{ssb:eq-c1-coordinate} shows that $\mathcal V_{i,a}$ is
	$\alpha$-semistable with slope $(ia/r)c_1(E)\cdot\alpha$.
	
	Since $X$ is a smooth surface, we can identify $N^1(X)_{\mathbb R}$ with $N_1(X)_{\mathbb R}$. Then, viewed as curve classes, $U$ is an open subset of $\operatorname{Mov}(X)$ and
	Lemma~\ref{lem:movable-class-over-polarization} supplies rational
	curve classes $\beta_i$, each a finite positive rational combination
	of strongly movable classes, such that
	\[
	(\pi_i)_*\beta_i\equiv H,
	\qquad \theta_i\cdot\beta_i=-\delta_i<0.
	\]
	
	Apply Lemma~\ref{lem:specialization-movable-witnesses}
	simultaneously to the triples $(\pi_i,\theta_i,\beta_i)$.
	In the symplectic case, spread out also the form
	\eqref{ssb:eq-form}, shrink so that it remains nondegenerate,
	and exclude characteristic two.
	After these changes of model and shrinkings, which preserve
	density of the bad locus, we have on every geometric closed fiber
	\begin{equation}\label{ssb:eq-grassmann-specialized-witnesses}
		(\pi_{i,\bar s})_*\beta_{i,\bar s}\equiv H_{\bar s},
		\qquad \theta_{i,\bar s}\cdot\beta_{i,\bar s}=-\delta_i,
		\qquad T\cdot\beta_{i,\bar s}\geq0
	\end{equation}
	for every effective divisor $T$ on $Y_{i,\bar s}$.

	Choose a bad closed point $s$, of characteristic $p$, and $e\geq1$
	such that $G_{\bar s}:=F_{X_{\bar s}}^{e*}E_{\bar s}$ is not $H_{\bar s}$-stable. Put
	$q=p^e$, and choose a nonzero proper saturated subsheaf
	$F\subset G_{\bar s}$ with $\mu_{H_{\bar s}}(F)\geq\mu_{H_{\bar s}}(G_{\bar s})$.
	In the symplectic case, if $\rk F>r/2$, replace $F$ by
	\[
	(G_{\bar s}/F)^*\otimes N_{\bar s}^{\otimes q}\subset G_{\bar s},
	\]
	using the isomorphism $G_{\bar s}\simeq G_{\bar s}^*\otimes N_{\bar s}^{\otimes q}$.
	This is saturated, has rank at most $r/2$, and still has slope
	at least $\mu_{H_{\bar s}}(G_{\bar s})$, since
	$\mu_{H_{\bar s}}(G_{\bar s}/F)\leq\mu_{H_{\bar s}}(G_{\bar s})$ and
	$c_1(N_{\bar s}^{\otimes q})\cdot H_{\bar s}=2\mu_{H_{\bar s}}(G_{\bar s})$.
	Set $i=\rk F$ and $M=\det F$. Then $i$ belongs to the selected
	set of indices in both cases, and
	\begin{equation}\label{ssb:eq-grassmann-destabilizing-degree}
		\left(c_1(M)-\frac{qi}{r}c_1(E_{\bar s})\right)\cdot H_{\bar s}\geq0.
	\end{equation}
	The determinant map extends across codimension two to a
	nonzero morphism $M\to\bigwedge^iG_{\bar s}$.
	
	Compose the determinant map with the universal determinant quotient
	on $Y_i(G_{\bar s})$. This gives a nonzero section of
	$L_i(G_{\bar s})\otimes\pi^*M^{-1}$: at the geometric generic point choose a quotient
	isomorphic on $F$, with isotropic dual in the symplectic case.
	Such a quotient exists by
	Lemma~\ref{ssb:lem-transverse-isotropic-quotient} in the symplectic
	case, without any isotropy assumption on $F$. The relative $e$-fold Frobenius
	\[
	\Phi_e:Y_i(E_{\bar s})\longrightarrow
	Y_i(E_{\bar s})\times_{X_{\bar s},F_{X_{\bar s}}^e}X_{\bar s}=Y_i(G_{\bar s})
	\]
	is surjective and satisfies
	$\Phi_e^*L_i(G_{\bar s})\simeq L_{i,\bar s}^{\otimes q}$; in the symplectic
	case it preserves the isotropy equations. Pulling back the section
	gives an effective divisor $D_{\bar s}$, possibly zero, with
	\[
	[D_{\bar s}]=q\theta_{i,\bar s}
	-\pi_{i,\bar s}^*\left(c_1(M)-\frac{qi}{r}c_1(E_{\bar s})\right).
	\]
	By \eqref{ssb:eq-grassmann-specialized-witnesses} and
	\eqref{ssb:eq-grassmann-destabilizing-degree},
	\[
	0\leq D_{\bar s}\cdot\beta_{i,\bar s}
	=-q\delta_i
	-\left(c_1(M)-\frac{qi}{r}c_1(E_{\bar s})\right)\cdot H_{\bar s}<0,
	\]
	a contradiction. Hence at least one $\widehat L_i$ is pseudoeffective.
\end{proof}

\subsection{Proof of the arithmetic surface theorem}
\label{ssb:subsec-normal-surfaces}

\begin{proof}[Proof of Theorem~\ref{ssb:thm-main}] 
	
	\emph{Step 1.} First assume that $X$ is smooth.
	By Lemma~\ref{ssb:lem-holonomy-square-dichotomy}, either the
	projective holonomy of $E$ is finite or $\Sym^6E$ is $H$-stable.
	In the first case, Lemma~\ref{ssb:lem-finite-projective-holonomy}
	shows that $E$ is projectively unitary flat, so
	$\int_X\Delta(E)=0$. We may therefore assume the second case.
	The changes of model used above preserve density: after shrinking,
	a dominant extension of the model ring is flat and hence open, and
	failure of strong stability is invariant under extension of
	geometric residue fields, as in the proof of
	Theorem~\ref{thm:pseudoeffective-exterior-power}.
	By Proposition~\ref{ssb:prop-grassmann-pseudoeffectivity}, one of the
	line bundles $\widehat L_i$ is pseudoeffective.
	Corollary~\ref{ssb:thm-plucker-positivity} therefore gives
	$\int_X\Delta(E)=0$.
	
	\medskip
	
	\emph{Step 2.} Now let us consider the general case.
	Let $f:\widetilde X\to X$ be a projective resolution which is an
	isomorphism over $X_{\mathrm{reg}}$, and set \(\widetilde E=f^*E\).
	Choose an effective exceptional divisor $D$ such that $-D$ is
	$f$-ample. Then $\tilde H:=nf^*H-D$ is ample for all sufficiently large integers
	$n$.
	
	The bundle $V:=\Sym^2\widetilde E=f^*(\Sym^2E)$ is $f^*H$-stable.
	So by Lemma \ref{ssb:lem-nearby-polarizations} changing $n$ if necessary, we can assume the bundle 
	$V$ is \(\widetilde H\)-stable.
	
	Spread out $f$, $D$, and $\widetilde H$ together with the original
	model. After enlarging the model ring and shrinking the base, we may
	assume that $\widetilde{\mathscr X}/S$ is smooth projective,
	each geometric fiber $f_{\bar s}:\widetilde X_{\bar s}\to X_{\bar s}$ is birational,
	$D_{\bar s}$ is effective and every component of its support is contracted
	by $f_{\bar s}$, and $\widetilde H_{\bar s}=nh_{\bar s}-D_{\bar s}$ is ample, where
	$h_{\bar s}=f_{\bar s}^*H_{\bar s}$. We may also assume that $\Sym^2\widetilde E_{\bar s}$ is
	$\widetilde H_{\bar s}$-stable. As explained in the smooth-case proof,
	these changes of the model preserve density of the bad locus.
	
	For every geometric closed fiber and every $e\geq0$,
	Lemma~\ref{ssb:lem-birational-transfer}, applied to
	$W=F_{X_{\bar s}}^{e*}E_{\bar s}$, shows that nonstability downstairs
	implies nonstability upstairs, since
	\(
	f_{\bar s}^*F_{X_{\bar s}}^{e*}E_{\bar s}
	\simeq F_{\widetilde X_{\bar s}}^{e*}\widetilde E_{\bar s}.
	\)
	Thus every bad fiber downstairs is also bad upstairs for the
	fixed polarization $\widetilde H_{\bar s}$, and
	$\Sigma^{\mathrm{ns}}(\widetilde E)$ is dense.
	
	The smooth case applied to $(\widetilde X,\widetilde H,\widetilde E)$
	gives $\int_{\widetilde X}\Delta(\widetilde E)=0$. Since $E$ is
	locally free, functoriality of Chern classes and the projection
	formula yield
	\[
	\int_X\Delta(E)
	=\int_{\widetilde X}\Delta(f^*E)=0.
	\]
\end{proof}

\section{Vector bundles on normal projective varieties}

\subsection{Restriction of the bad locus}

Let us note the following lemma.

\begin{lemma}\label{ssb:lem-density-on-surfaces}
	Let $(\mathscr X,\mathscr H,\mathscr E)\to S$ be a flat projective
	model over an integral scheme of finite type over $\mathbb Z$, with
	characteristic-zero generic point. Assume that its geometric fibers
	are normal and integral of dimension $n\geq3$, that $\mathscr H$ is
	relatively ample, and that the fibers of $\mathscr E$ are reflexive.
	Fix positive integers $m_1,\ldots,m_{n-2}$ for which the corresponding
	linear systems define embeddings and let $T$ be a nonempty open subset of the relative parameter space of
	complete intersections of these degrees. Suppose that the universal
	family $\mathscr Y\to T$ has geometrically normal integral surface
	fibers and carries a flat family $\mathscr F$ with
	\[
	\mathscr F_{\bar t}\simeq
	\bigl(\mathscr E_{\bar s}|_{Y_{\bar t}}\bigr)^{**}
	\qquad(\bar t\mapsto\bar s).
	\]
	If $\Sigma^{\mathrm{nss}}(\mathscr E/S)$ is dense, then
	$\Sigma^{\mathrm{nss}}(\mathscr F/T)$ is dense. Similarly, if $\Sigma^{\mathrm{ns}}(\mathscr E/S)$ is dense, then
	$\Sigma^{\mathrm{ns}}(\mathscr F/T)$ is dense.
\end{lemma}

\begin{proof}
	After shrinking $S$, the parameter spaces are products of projective
	bundles. Thus $T\to S$ is open, and its nonempty geometric fibers are
	integral. We may replace $S$ by the image of $T$. We show only the first implication, with the second one being completely analogous (one only needs to replace the strict slope inequality below by a non-strict one).
	
	Fix a closed point $s\in\Sigma^{\mathrm{nss}}(\mathscr E/S)$ and a
	saturated destabilizing subsheaf
	\(B\subsetneq G:=F_{X_{\bar s}}^{[e]}E_{\bar s}
	\)
	for some $e\geq1$.  Remove from $X_{\bar s}$ its singular
	locus and the non-locally-free loci of $E_{\bar s}$, $G$, $B$, and
	$G/B$. Their union has codimension at least two. On the remaining
	open subset $W$ the inclusion is an inclusion of vector bundles
	with locally free quotient, and reflexive Frobenius pullback agrees
	with ordinary pullback.
	
	For a general geometric point $\bar t$ of $T_{\bar s}$, the complement of $Y_{\bar t}\cap W$ in
	$Y_{\bar t}$ is finite. Restriction to $Y_{\bar t}\cap W$ and compatibility of
	absolute Frobenius with every morphism give an injection there.
	Extending across the finite complement yields
	\(
	\bigl(B|_{Y_{\bar t}}\bigr)^{**}
	\hookrightarrow F_{Y_{\bar t}}^{[e]}\mathscr F_{\bar t}.
	\)
	The extension is injective because it is generically injective and
	its source is torsion-free. Degrees can be computed on a further
	general complete-intersection curve in $Y_{\bar t}$ avoiding the finite
	complement. Hence
	\begin{equation*}
		\mu_{H_{\bar s}|_{Y_{\bar t}}}\bigl((B|_{Y_{\bar t}})^{**}\bigr)
		-\mu_{H_{\bar s}|_{Y_{\bar t}}}\bigl(F_{Y_{\bar t}}^{[e]}\mathscr F_{\bar t}\bigr)=
		\bigl(\mu_{H_{\bar s}}(B)-\mu_{H_{\bar s}}(G)\bigr)\prod_{j=1}^{n-2}m_j>0.
	\end{equation*}
	Saturation of the image, if needed, can only increase its slope.
	Thus a dense open subset of $T_{\bar s}$ consists of bad points.
	The exponent $e$ and this open subset may depend on $s$.
	
	Let $V\subset T$ be nonempty and open. Its image in $S$ is nonempty
	and open, so it contains a closed point $s$ of the original bad
	locus. The nonempty open subset $V_{\bar s}$ meets the preceding
	dense open subset of the integral fiber $T_{\bar s}$. This
	intersection contains a point defined over a finite extension of
	$k(s)$, and hence gives a closed point of $V$ at which strong
	semistability fails. This proves density.
\end{proof}

The families used in this lemma are obtained by restricting the
universal complete intersection to a suitable open parameter space.
For reflexive sheaves, one first spreads out the reflexive hull on
the generic surface. Generic flatness, Bertini for reflexive sheaves
\cite[Theorem~1.5]{Langer-Simpson-singular},
and openness of the fiberwise $S_2$ condition allow one to arrange
the displayed fiberwise identification: the two sheaves agree on a
fiberwise big open subset and are reflexive. For vector bundles the
family is simply the ordinary restriction. In each application below
the open parameter space remains an open subset of a product of
projective bundles.

We will also use that enlarging the field of definition does not
destroy density. Spread out the enlargement and restrict its model
to a nonempty open subset on which the map to the original model is
flat and hence open. Every nonempty open subset then maps to a
nonempty open subset of the original base. A destabilizing sheaf
remains destabilizing after extending the residue field. The same
open-image argument as in the lemma proves the assertion.

\begin{theorem}\label{ssb:cor-higher-dimensional-normal}
	Let $X$ be a normal complex projective variety of dimension $n\ge 2$ and let $H$ be an ample
	line bundle. Let $E$ be a vector bundle of rank $r\geq2$ such
	that $\operatorname{Sym}^2E$ is $H$-stable.
	If $\Sigma^{\mathrm{ns}}(E)$ is dense, then $E$ is irreducible
	projectively unitary flat.
\end{theorem}

\begin{proof}
	Remark \ref{stability-of-Sym-gives-usual} shows that $E$ is $H$-stable.  
	So in dimension two the assertion follows from Theorem~\ref{ssb:thm-main} and
	Proposition~\ref{ssb:prop-stable-zero-discriminant-normal}.
	
	Suppose $n\geq3$. Choose a model of $(X,H,E)$ as in
	Lemma~\ref{ssb:lem-density-on-surfaces}, with $\mathscr E$ locally
	free. By the stable restriction theorem (see, e.g., Theorem \ref{thm:restriction-prescribed-points}), we can choose sufficiently large	degrees $m_1,\ldots,m_{n-2}$ such that on a general normal
	complete intersection surface $Y\in |m_1H|\cap...\cap |m_{n-2}H|$, the restriction
	\begin{equation}\label{ssb:eq-stable-sym-square-on-surface-vector}
		(\operatorname{Sym}^2 E)|_Y
		=\operatorname{Sym}^2(E|_Y)	\end{equation}
	is $H|_Y$-stable. Keep the complete intersection variable and take the nonempty open
	parameter space $T$ of normal integral surfaces described in 	Lemma~\ref{ssb:lem-density-on-surfaces}. The universal bundle is $\mathscr F=\mathscr E|_{\mathscr Y}$.
	Its bad locus is dense by Lemma~\ref{ssb:lem-density-on-surfaces}.
	
	Embed an algebraic closure of the function field of $T$ into
	$\mathbb C$, extending the original finitely generated field of
	definition. The resulting geometric generic surface $Y_{\bar\eta}$ is
	normal and projective, and its bundle satisfies
	\eqref{ssb:eq-stable-sym-square-on-surface-vector}. Its arithmetic model
	over $T$ has dense bad locus. Theorem~\ref{ssb:thm-main} therefore
	gives
	\begin{equation*}
		0=\int_{Y_{\bar\eta}}\Delta(E|_{Y_{\bar\eta}})=
		\Delta(E)H^{n-2}\cdot \prod_{j=1}^{n-2}m_j.
	\end{equation*}
	So Proposition~\ref{ssb:prop-stable-zero-discriminant-normal}
	implies the theorem.
\end{proof}

\medskip

The rank two and rank three results give the following extension
to normal projective varieties.

\begin{theorem}\label{cor:rank-at-most-three}
	Let $X$ be a normal complex projective variety of dimension $n\ge 2$ and let $H$ be an ample
	line bundle. Let $E$ be  an $H$-stable vector bundle of rank $r\le 3$. 
	If  $\Sigma^{\mathrm{ns}}(E)$ is dense, then $E$ is
	irreducible projectively unitary flat.
\end{theorem}

\begin{proof}
	The rank one case is immediate, so assume that $r=2$ or $r=3$. 
	
	Suppose first that $\dim X=2$. Choose a projective resolution
	$f:\widetilde X\to X$ which is an isomorphism over
	$X_{\mathrm{reg}}$, and set \(\widetilde E:=f^*E\) and $\widetilde H:=f^*H.$
	The bundle $\widetilde E$ is $\widetilde H$-stable. Choose an effective
	exceptional divisor $D$ such that $-D$ is $f$-ample.
	By Lemma~\ref{ssb:lem-nearby-polarizations}, for a sufficiently
	large integer $N$, the divisor
	\(
	P=N\widetilde H-D
	\)
	is ample and $\widetilde E$ is $P$-stable.
	
	Spread out these data together with the original model.
	After enlarging the model and shrinking its base, we may assume
	that $\widetilde{\mathscr X}/S$ is smooth projective, every
	geometric fiber $f_{\bar s}:\widetilde X_{\bar s}\to X_{\bar s}$ is birational,
	$D_{\bar s}$ is effective and supported on curves contracted by $f_{\bar s}$,
	and $P_{\bar s}=N\widetilde H_{\bar s}-D_{\bar s}$ is ample, where $\widetilde H_{\bar s}=f_{\bar s}^*H_{\bar s}$.
	We may also assume that $\widetilde E_{\bar s}$ is geometrically
	$P_{\bar s}$-stable. These changes preserve density of the bad locus.
	
	Applying Lemma~\ref{ssb:lem-birational-transfer} over the
	algebraic closure of each residue field, with
	$W=F_{X_{\bar s}}^{e*}E_{\bar s}$, shows that
	if $F_{X_{\bar s}}^{e*}E_{\bar s}$ is not geometrically $H_{\bar s}$-stable then
	$F_{\widetilde X_{\bar s}}^{e*}\widetilde E_{\bar s}$
	is not geometrically $P_{\bar s}$-stable.
	Thus every bad fiber downstairs is also bad upstairs, and
	$\Sigma^{\mathrm{ns}}(\widetilde E)$ is dense.
	Applying Theorem~\ref{ssb:prop-low-rank-holonomy} gives
	\(
	\int_{\widetilde X}\Delta(f^* E)=0.
	\)
	Since $E$ is locally free, the projection formula implies that
	\(
	\int_X\Delta(E)=0.
	\)
	Proposition~\ref{ssb:prop-stable-zero-discriminant-normal}
	now proves the assertion in dimension two.
	
	For $\dim X=n\geq3$, repeat the variable-surface construction
	in the proof of Theorem~\ref{ssb:cor-higher-dimensional-normal},
	choosing the degrees so that $E|_Y$ is stable on a general
	normal complete-intersection surface
	\(
	Y\in|m_1H|\cap\cdots\cap|m_{n-2}H|.
	\)
	By Lemma~\ref{ssb:lem-density-on-surfaces}, the universal
	restricted bundle has dense $\Sigma^{\mathrm{ns}}$ over its
	parameter space. Applying the surface case to the geometric
	generic surface gives
	\[
	0=\int_Y\Delta(E|_Y)
	=\Delta(E)H^{n-2}\prod_{i=1}^{n-2}m_i.
	\]
	Hence $\Delta(E)H^{n-2}=0$, and
	Proposition~\ref{ssb:prop-stable-zero-discriminant-normal}
	shows that $E$ is irreducible projectively unitary flat.
\end{proof}

\section{Reflexive sheaves on normal projective varieties}

Let $X$ be a normal complex projective variety of dimension $n\geq2$
with quotient singularities in codimension two. In this section we use
Chern classes of reflexive sheaves, equivalently Mumford's
$\mathbb Q$-Chern classes, as in \cite{ssb-LangerIntersection2025} and \cite[Section~7]{ssb-Langer2024}, and set
\[
\Delta(E)=2r c_2(E)-(r-1)c_1(E)^2,
\]
where $r=\operatorname{rk}E.$
Thus $\int_X\Delta(E)$ denotes a symmetric $(n-2)$-linear form on
$N^1(X)_{\mathbb Q}$. The assertion $\int_X\Delta(E)=0$ means that
\[
\Delta(E) L_1\cdots L_{n-2}=0
\]
for all Cartier divisors $L_1,\ldots,L_{n-2}$; for $n=2$ it is the
vanishing of a number. For vector bundles these pairings agree with
the usual Chern-class intersection numbers.

Let $H$ be an ample line bundle and let $E$ be reflexive. Choose a flat
projective model $(\mathscr X,\mathscr H,\mathscr E)\to S$ over an
integral scheme of finite type over $\mathbb Z$, with characteristic-zero
generic point. After shrinking $S$, we may assume that the fibers are
geometrically normal and geometrically integral, and that $E_{\bar s}$ is
reflexive for every geometric point $\bar s$ over a closed point $s\in S$.
We interpret the bad locus as
\begin{equation*}
	\Sigma^{\mathrm{nss}}(\mathscr E/S)
	=\left\{s\in S \text{ closed}:
	F_{X_{\bar s}}^{[e]}E_{\bar s}\text{ is not $H_{\bar s}$-semistable for some }e\geq1
	\right\}.
\end{equation*}
As before, semistability can equivalently be tested over
the finite residue field.
Similarly, we set
\[
\Sigma^{\mathrm{ns}}(\mathscr E/S)
=\left\{s\in S \text{ closed}:
F_{X_{\bar s}}^{[e]}E_{\bar s}\text{ is not $H_{\bar s}$-stable for some }e\geq1
\right\},
\]
where $\bar s$ is a geometric point over $s$. In the nonstability
arguments below, we work on these geometric fibers.

\begin{theorem}\label{ssb:cor-higher-dimensional-reflexive}
	Let $X$ be a normal complex projective variety of dimension $n\geq2$
	with quotient singularities in codimension two. Let $H$ be an ample
	line bundle and let $E$ be a reflexive sheaf of rank $r\geq2$ such that
	$\operatorname{Sym}^{[2]}E$ is $H$-stable. If $\Sigma^{\mathrm{ns}}(E)$ is dense,
	then $E|_{X_{\mathrm{reg}}}$ is locally free,
	$\cEnd(E)|_{X_{\mathrm{reg}}}$ is nef, and
	$\int_X\Delta(E)=0$.
	
	If, moreover, $X$ is klt, then $E|_{X_{\mathrm{reg}}}$ is irreducible
	projectively unitary flat. There exist a normal projective variety
	$\bar X$ and a finite surjective Galois morphism
	$\gamma:\bar X\to X$, \'etale in codimension one, such that the
	projectivization of
	\(
	(\gamma|_{\gamma^{-1}(X_{\mathrm{reg}})})^*
	(E|_{X_{\mathrm{reg}}})
	\)
	extends to an algebraic Brauer--Severi bundle on $\bar X$ induced by a representation
	\(
	\pi_1^{\mathrm{top}}(\bar X^{\mathrm{an}})
	\longrightarrow\mathrm{PU}(r).
	\)
\end{theorem}

\begin{proof}
	Let us recall that by Remark \ref{stability-of-Sym-gives-usual} the sheaf $E$ is $H$-stable.
	By Lemma~\ref{ssb:lem-holonomy-square-dichotomy}, either the projective
	holonomy of $E$ is finite or $\operatorname{Sym}^{[6]}E$ is $H$-stable.
	In the first case, Lemma~\ref{ssb:lem-finite-projective-holonomy} gives
	$\Delta(E)H^{n-2}=0$, and Step~3 below proves all the assertions.
	Thus in Steps~1 and~2 we may assume that
	$\operatorname{Sym}^{[6]}E$ is $H$-stable.
	
	\emph{Step 1. The surface case.}
	Suppose that $n=2$. By \cite[Lemma~7.1]{ssb-Langer2024}, there exist
	a normal projective surface $Y$ and a finite surjective morphism
	$f:Y\to X$ such that $V:=f^{[*]}E$ is locally free. The same lemma gives
	\begin{equation}\label{ssb:eq-reflexive-discriminant-cover}
		\int_Y\Delta(V)=\deg(f)\int_X\Delta(E).
	\end{equation}
	By Corollary~\ref{stability-pullback-of-symmetric-powers}, the vector
	bundle $\operatorname{Sym}^6V$ is $f^*H$-stable.
	
	We verify that the dense bad locus also pulls back. Spread out $f$ and
	$V$, enlarging the model and shrinking the base when necessary. The
	open-base-change argument used in
	Lemma~\ref{ssb:lem-density-on-surfaces} preserves density. We may assume
	that every geometric fiber $f_{\bar s}:Y_{\bar s}\to X_{\bar s}$ is finite and surjective of degree $\deg(f)$,
	that $Y_{\bar s}$ is geometrically normal and geometrically integral, and that
	$V_{\bar s}$ is locally free with
	\(
	V_{\bar s}\simeq f_{\bar s}^{[*]}E_{\bar s}.
	\)
	For the last assertion, spread out the map $f^*E\to V$ and its
	isomorphism over the locally free locus of $E$. After shrinking, the
	complement of this locus has codimension at least two on every fiber.
	Its inverse image under $f_{\bar s}$ has the same codimension bound, so this isomorphism follows from reflexivity.
	
	Let $s$ belong to the original bad locus and, on the geometric fiber,
	choose a saturated nonzero subsheaf
	\( F\subsetneq G:=F_{X_{\bar s}}^{[e]}E_{\bar s}
	\)
	with $\mu_{H_{\bar s}}(F)\geq\mu_{H_{\bar s}}(G)$ for some $e\geq1$.
	Since a finite morphism preserves the condition of
	being a big open subset under inverse image, the usual commutation of
	Frobenius with pullback gives, after taking reflexive hulls,
	\[
	f_{\bar s}^{[*]}G
	\simeq F_{Y_{\bar s}}^{[e]}(f_{\bar s}^{[*]}E_{\bar s})
	=F_{Y_{\bar s}}^{e*}V_{\bar s}.
	\]
	Pullback followed by reflexive hull gives an injection
	\(
	f_{\bar s}^{[*]}F\hookrightarrow F_{Y_{\bar s}}^{e*}V_{\bar s}.
	\)
	Indeed, it is generically injective and its source is torsion free.
	The finite morphism $f_{\bar s}$ is flat over the codimension-one points of
	$X_{\bar s}$, so the determinant pullback and projection formulas give
	\[
	\mu_{f_{\bar s}^*H_{\bar s}}(f_{\bar s}^{[*]}F)
	-\mu_{f_{\bar s}^*H_{\bar s}}(F_{Y_{\bar s}}^{e*}V_{\bar s})=\deg(f)\bigl(\mu_{H_{\bar s}}(F)-\mu_{H_{\bar s}}(G)\bigr)\geq0.
	\]
	Thus every such $s$ is also bad for $V$, and
	$\Sigma^{\mathrm{ns}}(V)$ is dense. Applying
	Theorem~\ref{ssb:thm-main} to $(Y,f^*H,V)$ gives
	$\int_Y\Delta(V)=0$. Equation~\eqref{ssb:eq-reflexive-discriminant-cover}
	therefore implies $\int_X\Delta(E)=0$.
	
	\emph{Step 2. Restriction to a variable surface.}
	Suppose that $n\geq3$. Apply
	\cite[Theorem~7.3]{ssb-Langer2024} (or Theorem \ref{thm:restriction-prescribed-points}) successively to choose sufficiently
	large positive integers $m_1,\ldots,m_{n-2}$ such that, for a general
	complete intersection
	\(
	Y\in |m_1H|\cap...\cap |m_{n-2}H|,
	\)
	the surface $Y$ is normal with quotient singularities and
	\begin{equation}\label{ssb:eq-stable-sym-six-on-surface}
		\bigl((\operatorname{Sym}^{[6]}E)|_Y\bigr)^{**}
		\simeq\operatorname{Sym}^{[6]}E_Y,
		\qquad E_Y:=(E|_Y)^{**},
	\end{equation}
	is $H|_Y$-stable. 
	By \cite[Theorems~2.4 and~5.8(3)]{ssb-LangerIntersection2025}, for a very general such surface,
	\begin{equation}\label{ssb:eq-discriminant-restriction}
		\int_Y\Delta(E_Y)=\Delta(E)H^{n-2}\cdot \prod_{j=1}^{n-2}m_j.
	\end{equation}
	After spreading out the linear systems,
	take a nonempty open subscheme $T$ of the product of their relative
	parameter spaces, with universal surface $\mathscr Y\to T$. By
	generic flatness and openness of geometric normality and reflexivity,
	we may shrink $T$ so that all its geometric fibers are normal integral
	surfaces and there is a flat coherent sheaf $\mathscr F$ with
	$\mathscr F_{\bar t}\simeq(E_{\bar s}|_{Y_{\bar t}})^{**}$
	for $\bar t\in T_{\bar s}$.
	The geometric generic fiber has quotient singularities and satisfies
	\eqref{ssb:eq-stable-sym-six-on-surface}. The morphism $T\to S$ is
	dominant and open, with geometrically integral nonempty fibers.
	By	Lemma~\ref{ssb:lem-density-on-surfaces} the set \(	\Sigma^{\mathrm{ns}}(\mathscr F/T)\)
	{ is Zariski dense in }$T$.
	
	Embed an algebraic closure of the function field of $T$ into
	$\mathbb C$, extending the chosen field of definition of $(X,H,E)$.
	Apply Step~1 to the resulting complex geometric generic surface and
	its reflexive sheaf. It gives
	$\int\Delta(\mathscr F_{\bar\eta})=0$. The generic surface also
	satisfies the restriction rule
	\eqref{ssb:eq-discriminant-restriction}: the generic restriction
	conditions in the proof of the cited restriction rule are all
	satisfied at the geometric generic point. Therefore
	$\Delta(E)H^{n-2}=0$.
	
	\emph{Step 3. Local freeness and projective unitary flatness.}
	Since $E$ is $H$-stable, Theorem~\ref{general-local-freeness} implies that $E|_{X_{\mathrm{reg}}}$
	is locally free, $\cEnd(E)|_{X_{\mathrm{reg}}}$ is nef and  the entire
	numerical form $\int_X\Delta(E)$ vanishes.
	
	If $X$ is klt, Theorem~\ref{klt-local-freeness} gives projective
	unitary flatness on $X_{\mathrm{reg}}$ and the asserted finite Galois
	cover. To see irreducibility, choose a sufficiently general smooth
	complete-intersection curve $C\subset X_{\mathrm{reg}}$ such that
	$E|_C$ is stable, using the restriction theorem. If the projective
	unitary monodromy on $X_{\mathrm{reg}}$ preserved a nonzero proper
	linear subspace, its restriction to $C$ would give a nonzero proper
	holomorphic subbundle of $E|_C$ with the same slope: the induced
	connections on the subbundle and on $E|_C$ have the same normalized
	scalar curvature. This contradicts stability of $E|_C$.
\end{proof}

For ranks at most three one can use 
Theorem \ref{ssb:prop-low-rank-holonomy} to
dispense with the symmetric-power
hypothesis. The finite-cover step requires the additional argument
given below.

\begin{theorem}\label{cor:rank-at-most-three-normal}
	Let $X$ be a normal complex projective variety of dimension $n\geq2$
	with quotient singularities in codimension two. Let $H$ be an ample
	line bundle and let $E$ be an $H$-stable reflexive sheaf of rank $r\leq 3$. If $\Sigma^{\mathrm{ns}}(E)$ is dense,
	then $E|_{X_{\mathrm{reg}}}$ is locally free,
	$\cEnd (E)|_{X_{\mathrm{reg}}}$ is nef, and
	$\int_X\Delta(E)=0$.
	
	If, moreover, $X$ is klt, then $E|_{X_{\mathrm{reg}}}$ is irreducible
	projectively unitary flat and there exist a normal projective variety
	$\bar X$ and a finite surjective Galois morphism
	$\gamma:\bar X\to X$, \'etale in codimension one, such that the
	projectivization of
	\(
	(\gamma|_{\gamma^{-1}(X_{\mathrm{reg}})})^*
	(E|_{X_{\mathrm{reg}}})
	\)
	extends to an algebraic Brauer--Severi bundle on $\bar X$ induced by a representation
	\(
	\pi_1^{\mathrm{top}}(\bar X^{\mathrm{an}})
	\longrightarrow\mathrm{PU}(r).
	\)
\end{theorem}

\begin{proof}
	The rank one case is immediate, so assume that $r\in\{2,3\}$.
	
	\emph{The surface case.}
	Suppose first that $\dim X=2$. By
	\cite[Lemma~7.1]{ssb-Langer2024}, there is a finite surjective
	morphism $f_0:Y_0\to X$ from a normal projective surface such that
	$f_0^{[*]}E$ is locally free. Replace $Y_0$ by its normalization in
	a Galois closure of $\mathbb C(Y_0)/\mathbb C(X)$, and write
	\[
	f:Y\longrightarrow X,\qquad
	\Gamma=\operatorname{Gal}(\mathbb C(Y)/\mathbb C(X)),\qquad
	A=f^*H,\qquad V=f^{[*]}E.
	\]
	The bundle $V$ is the ordinary pullback of $f_0^{[*]}E$, since
	these sheaves agree on a big open subset and are reflexive.
	In particular, $V$ is locally free. By \cite[Theorem 0.2]{ssb-LangerIntersection2025} we also have
	\[
	\int_Y\Delta(V)=\deg(f)\int_X\Delta(E).
	\]
	
	The Galois pullback $V$ is $A$-polystable (see
	\cite[Lemmas~3.2.2 and~3.2.3]{Huybrechts-Lehn2010}).
	Write its isotypic decomposition as
	\[
	V=\bigoplus_{i=1}^{t}W_i^{\oplus a_i},
	\]
	where the $W_i$ are pairwise nonisomorphic $A$-stable bundles of
	the same slope. The canonical $\Gamma$-linearization permutes
	the isotypic summands transitively. Indeed, the sum of any orbit
	is a saturated $\Gamma$-invariant subsheaf. Over the big open
	subset of $X$ where $f$ is flat and $E$ is locally free it descends
	to a subsheaf of $E$: descend its generic subspace and take the
	saturation, which commutes with flat pullback. Extending by
	saturation to $X$ gives a subsheaf of slope $\mu_H(E)$.
	Stability of $E$ excludes a proper nonzero orbit sum.
	Consequently, all $W_i$ have the same rank and all $a_i$ have
	the same value $a$. Since $r=t\,a\,\operatorname{rk}W_i$ is prime,
	either $V$ is stable, or $V\simeq L^{\oplus r}$ for a line bundle
	$L$, or
	\[
	V=L_1\oplus\cdots\oplus L_r
	\]
	for pairwise nonisomorphic line bundles permuted transitively
	by $\Gamma$.
	
	If $V$ is stable, its bad locus $\Sigma^{\mathrm{ns}}(V)$ is
	dense. Indeed, spread out $f$ and $V$ as in the proof of
	Theorem~\ref{ssb:cor-higher-dimensional-reflexive}.
	The pullback of a proper subsheaf of $F_{X_{\bar s}}^{[e]}E_{\bar s}$ with
	slope at least that of $F_{X_{\bar s}}^{[e]}E_{\bar s}$ gives, after saturation,
	a proper subsheaf of $F_{Y_{\bar s}}^{e*}V_{\bar s}$ with slope at least that
	of $F_{Y_{\bar s}}^{e*}V_{\bar s}$. Here slopes scale by $\deg(f)$ before
	saturation. Thus every bad fiber downstairs is bad upstairs.
	Theorem~\ref{cor:rank-at-most-three} gives
	$\int_Y\Delta(V)=0$ and hence $\int_X\Delta(E)=0$.
	
	The case $V\simeq L^{\oplus r}$ also gives
	$\int_Y\Delta(V)=0$. It remains to consider the decomposition
	into $r$ pairwise nonisomorphic line bundles.
	Spread out $f$, the $\Gamma$-action, the line bundles, the
	decomposition of $V$, and the isomorphisms describing the
	permutation of the summands. After enlarging the model and
	shrinking its base, these operations preserve the density
	hypothesis, every geometric fiber $f_{\bar s}:Y_{\bar s}\to X_{\bar s}$ is a
	finite Galois cover with group $\Gamma$, and
	\[
	V_{\bar s}\simeq f_{\bar s}^{[*]}E_{\bar s}
	\simeq\bigoplus_{i=1}^{r}L_{i,\bar s}.
	\]
	We may also assume that $E_{\bar s}$ is stable, that the residue
	characteristic does not divide $|\Gamma|$, and that the Chern
	numbers of $V_{\bar s}$ and the degrees of the line bundles are constant.
	Choose a bad geometric fiber and an integer $e\geq1$ such that
	$F_{X_{\bar s}}^{[e]}E_{\bar s}$ is not stable, and put
	$p=\operatorname{char}\kappa(s)$ and $q=p^e$.
	As in the proof of Theorem~\ref{ssb:prop-low-rank-holonomy},
	the rank assumption gives a saturated rank one reflexive
	subsheaf $M$ of either $F_{X_{\bar s}}^{[e]}E_{\bar s}$ or its dual
	whose slope is at least that of the ambient sheaf.
	In the first case, reflexive pullback gives an equivariant map
	\[
	f_{\bar s}^{[*]}M\longrightarrow
	\bigoplus_{i=1}^{r}L_{i,\bar s}^{\otimes q}.
	\]
	Transitivity makes every component nonzero. The source degree
	is at least the common target degree, so each component has a
	zero divisor of nonpositive ample degree. It is therefore an
	isomorphism of rank one reflexive sheaves. The dual case is
	identical, using $L_{i,\bar s}^{\otimes(-q)}$.
	Thus all the $q$th powers of the line bundles are isomorphic, and
	\(
	q^2\int_{Y_{\bar s}}\Delta(V_{\bar s})=0.
	\)
	Constancy of Chern numbers and the finite-pullback formula prove
	$\int_X\Delta(E)=0$ in all surface cases.
	
	\emph{Higher dimensions and the remaining conclusions.}
	For $n\geq3$, use the variable complete-intersection surfaces
	from Step~2 of the proof of
	Theorem~\ref{ssb:cor-higher-dimensional-reflexive}, choosing
	their degrees so that the reflexive restriction of $E$ is
	stable.  By Lemma~\ref{ssb:lem-density-on-surfaces}, the universal
	family of reflexive restrictions has dense
	$\Sigma^{\mathrm{ns}}$ over its parameter space.
	The surface case applied to the geometric generic member,
	together with the discriminant restriction formula, yields
	\(
	\Delta(E)H^{n-2}=0.
	\)
	Theorem~\ref{general-local-freeness} now gives local freeness on
	$X_{\mathrm{reg}}$, nefness of $\cEnd (E)|_{X_{\mathrm{reg}}}$,
	and vanishing of the entire discriminant form.
	If $X$ is klt, Theorem~\ref{klt-local-freeness} gives the
	projectively unitary flat structure and the asserted finite
	Galois cover. Its irreducibility follows by restricting to a
	general smooth complete-intersection curve on which $E$ is
	stable, exactly as in Step~3 of the preceding proof.
\end{proof}

\section{A higher-rank differential-form analogue of
	Shepherd--Barron's theorem}

The following theorem strengthens \cite[Proposition 6.4]{Langer2004}, which generalized 
\cite[Theorem 5, (2)]{ShepherdBarron1998} to higher rank. 

\begin{theorem}\label{thm:higher-rank-SB}
	Let $(X,H)$ be a smooth polarized complex projective variety of
	dimension $n\geq 2$, and let $E$ be a reflexive sheaf of rank $r\geq 2$.
	Set
	\[
	b(n,r):=\min\left\{n,\left\lfloor\frac{r^2}{4}\right\rfloor\right\}.
	\]
	\begin{enumerate}
		\item If $E$ is $H$-stable and $\Sigma^{\mathrm{ns}}(E)$ is dense,
		then there exist an integer $1\leq j\leq b(n,r)$ and a line
		bundle $L\hookrightarrow\Omega_X^j$
		such that $c_1(L)\cdot H^{n-1}\geq 0$.
		\item If $E$ is $H$-semistable and
		$\Sigma^{\mathrm{nss}}(E)$ is dense, then $j$ and $L$ can be
		chosen as in~{\rm(1)} with	$c_1(L)\cdot H^{n-1}>0.$
	\end{enumerate}
	In either case, one has $\kappa(L)\leq j$.
	More precisely, $L$ can be chosen in the form
	$(\det J)^{-1}$, where $J$ is a torsion-free quotient of $T_X$
	of rank $j$ and of nonpositive, respectively negative, $H$-degree.
\end{theorem}

\begin{proof}
	Choose a model
	\(
	(\mathscr X,\mathscr H,\mathscr E)\to S,
	\)
	where $S=\operatorname{Spec}R$ for an integral finitely generated
	$\mathbb Z$-subalgebra $R\subset\mathbb C$.
	After shrinking $S$ and replacing $\mathscr H$ by a positive
	multiple, we may assume that $\mathscr X/S$ is smooth projective
	with geometrically integral fibers of dimension $n$, that
	$\mathscr H$ is relatively very ample, and that all geometric
	fibers of $\mathscr E$ are stable in case~(1), respectively
	semistable in case~(2).
	We can also assume that the intersection numbers
	\(
	h:=\mathscr H_{\bar s}^{\,n},\) and
	\(
	k:=K_{\mathscr X_{\bar s}}\cdot
	\mathscr H_{\bar s}^{\,n-1}
	\)
	are independent of $s$, and the residue characteristics  of $\kappa (s)$ at closed points $s\in S$ are larger than $r$. 	For a coherent sheaf $A$ on a smooth polarized variety $(Y,H)$
	of dimension $n$, write
	\[
	\deg_H A:=c_1(A)\cdot H^{n-1}. 
	\]
	
	Fix a closed point $s$ in the relevant bad set, and write
	\(
	X_{\bar s}:=\mathscr X_{\bar s},\)
	$H_{\bar s}:=\mathscr H_{\bar s}$ and
	$E_{\bar s}:=\mathscr E_{\bar s}.$
	Let $F$ denote the absolute Frobenius of $X_{\bar s}$ and put
	$p=\operatorname{char}\kappa(s)>r$.
	In case~(1), choose the least integer $m\geq 1$ such that
	$F^{m*}E_{\bar s}$ is not stable.
	In case~(2), choose instead the least integer $m\geq 1$ such that
	$F^{m*}E_{\bar s}$ is not semistable.
	Set
	\(
	V:=F^{m*}E_{\bar s},\) and \(W:=F^{(m-1)*}E_{\bar s}.
	\)
	Thus $V=F^*W$, and $W$ is stable in case~(1), respectively
	semistable in case~(2).
	
	If $V$ is not semistable, let $G\subset V$ be the first nonzero
	term of its Harder--Narasimhan filtration.
	Otherwise, which can occur only in case~(1), let $G$ be the first
	nonzero term of a Jordan--H\"older filtration of $V$.
	In either case, $G$ is a proper saturated semistable subsheaf of
	$V$, and $Q:=V/G$ is torsion free.
	If $V$ is not semistable, then
	\begin{equation*}
		\mu_{H_{\bar s}}(G)>\mu_{H_{\bar s}}(V),
		\qquad
		\mu_{H_{\bar s},\max}(Q)<\mu_{H_{\bar s}}(G).
	\end{equation*}
	If $V$ is strictly semistable, then both $G$ and $Q$ are
	semistable of slope $\mu_{H_{\bar s}}(V)$.
	
	The canonical connection on $V=F^*W$ induces an $\cO_{X_{\bar s}}$-linear map $G\longrightarrow Q\otimes\Omega_{X_{\bar s}}^1$,
	or, equivalently, a morphism
	$\theta_{\bar s}\colon T_{X_{\bar s}}\to\mathcal H om(G,Q).$
	This morphism is nonzero by
	Proposition~\ref{prop:Cartier-descent}.
	
	We next estimate the maximal slope of $\mathcal H om(G,Q)$.
	The low-height tensor-product theorem
	\cite{Ilangovan-Mehta-Parameswaran}, in the form
	\cite[Corollary~1.3]{LangerRestriction2010}, implies that the
	tensor product, modulo torsion, of semistable torsion-free
	sheaves of ranks $a$ and $b$ is semistable whenever
	$p>a+b-2$.
	Apply this to $G^*$ and to the semistable factors of the
	Harder--Narasimhan filtration of $Q$.
	The sum of the ranks of each such pair is at most $r<p$.
	Therefore
	\begin{equation}\label{eq:higher-rank-hom-slope}
		\mu_{H_{\bar s},\max}\bigl(\mathcal H om(G,Q)\bigr)=
		\mu_{H_{\bar s},\max}(Q)-\mu_{H_{\bar s}}(G)
		\leq 0.
	\end{equation}
	Here the last inequality is strict whenever $V$ is not
	semistable, and hence always in case~(2).
	For completeness, if $B$ is a Harder--Narasimhan factor of $Q$,
	then
	\(
	\mathcal H om(G,B)
	\hookrightarrow (G^*\otimes B)^{**}
	\)
	is an isomorphism in codimension $1$, so both sheaves are semistable of slope
	$\mu_{H_{\bar s}}(B)-\mu_{H_{\bar s}}(G)$. This shows the above equality of slopes.
	
	Let
	\(
	J_{\bar s}:=\operatorname{im}\theta_{\bar s}.
	\)
	By \eqref{eq:higher-rank-hom-slope}, we have $	\deg_{H_{\bar s}}J_{\bar s}\leq 0$ in case~(1)
	and $\deg_{H_{\bar s}}J_{\bar s}<0$ in case~(2).
	Writing $a_s=\operatorname{rk}G$ and
	$j_s=\operatorname{rk}J_{\bar s}$, we also have
	\[
	1\leq j_s
	\leq\min\{n,a_s(r-a_s)\}
	\leq b(n,r).
	\]
	
	The leading two terms of the Hilbert polynomial of $J_{\bar s}$ are
	\begin{equation*}
		\chi\bigl(J_{\bar s}(tH_{\bar s})\bigr)
		={}\frac{j_s h}{n!}t^n
		+\frac{\deg_{H_{\bar s}}J_{\bar s}-\frac{j_s}{2}k}
		{(n-1)!}t^{n-1}
		+O(t^{n-2}).
	\end{equation*}
	Since $1\leq j_s\leq b(n,r)$ and
	$\deg_{H_{\bar s}}J_{\bar s}\leq 0$, the coefficients of degrees $n$ and
	$n-1$ are bounded above, independently of $s$ and of the
	Frobenius exponent $m$.
	Moreover, each $J_{\bar s}$ is pure of dimension $n$ and is a quotient
	of $T_{X_{\bar s}}$.
	Theorem~\ref{thm:Grothendieck-boundedness}, applied to
	$T_{\mathscr X/S}$ with $d=n$, shows that these quotients form
	a bounded family.
	In particular, only finitely many Hilbert polynomials
	$P_1,\ldots,P_N$ occur.
	
	Consider the finite disjoint union
	\[
	\mathscr Q:=
	\bigsqcup_{\nu=1}^N
	\operatorname{Quot}_{T_{\mathscr X/S}/\mathscr X/S}^{P_\nu}.
	\]
	The morphism $\mathscr Q\to S$ is projective, and its image
	contains the relevant dense bad set.
	Hence it is surjective.
	On the chosen complex fiber, we therefore obtain a coherent
	quotient
	\(
	T_X\twoheadrightarrow\widetilde J
	\)
	of rank $1\leq j\leq b(n,r)$.
	Rank and first-Chern degree are determined by the leading two
	Hilbert coefficients, since $h$ and $k$ are constant.
	Consequently, $	\deg_H\widetilde J\leq 0$ in case~(1)
	and $	\deg_H\widetilde J<0$ in case~(2).

	Let $\mathcal T$ be the torsion subsheaf of $\widetilde J$ and
	set $J:=\widetilde J/\mathcal T$.
	Then $J$ is a torsion-free quotient of $T_X$ of rank $j$.
	The codimension-one cycle of $\mathcal T$ is effective, so
	\[
	\deg_H J
	=\deg_H\widetilde J-\deg_H\mathcal T
	\leq\deg_H\widetilde J.
	\]
	Thus removing torsion preserves the required degree inequality.
	
	Taking the $j$th exterior power of the quotient
	$T_X\twoheadrightarrow J$ gives a nonzero composition
	\[
	\bigwedge\nolimits^j T_X
	\longrightarrow
	\bigwedge\nolimits^j J
	\longrightarrow
	\det J.
	\]
	Dualizing, we obtain a nonzero morphism
	\(
	L:=(\det J)^{-1}\longrightarrow\Omega_X^j.
	\)
	It is injective because $L$ is invertible and $X$ is integral.
	Moreover,
	\(
	c_1(L)\cdot H^{n-1}=-\deg_H J,
	\)
	which proves the asserted degree inequalities.
	Finally,  Bogomolov's vanishing theorem gives
	$\kappa(L)\leq j$ (see
	\cite[Corollary~6.9]{EsnaultViehweg1992}).
\end{proof}

\section{Counterexamples and examples}

We construct three bundles on the same polarized surface: a strictly
semistable rank three bundle, a stable pseudoeffective rank four bundle,
and a stable rank six bundle which is not pseudoeffective. In each case
the discriminant is positive and the first Frobenius pullback is not
semistable at a set of primes of positive Dirichlet density.

Throughout this section, $\operatorname{Hol}_x(E)$ denotes algebraic
holonomy, whereas $\operatorname{Hol}^{\mathrm{an}}_x(E)$ denotes the
analytic holonomy group of the specified connection. An overline denotes
closure in the usual topology. We suppress the basepoint when considering
these groups up to conjugacy.

\subsection{The common construction}

Let us recall one of Brenner's examples of a semistable vector bundle on a smooth complex projective curve, for which 
a dense set of its reductions is not strongly semistable (see \cite[Corollary~2]{Brenner-Miyaoka}).

Let $R=\mathbb Z[1/5]$ and let
\[
\mathscr C=\operatorname{Proj} R[x,y,z]/(x^5+y^5+z^5)\to \Spec R
\]
be a spreading out of the Fermat quintic curve in $\bP^2$.
On $\mathscr C$ consider a rank two syzygy bundle
defined by
\begin{equation}\label{eq:brenner-quintic}
	0\longrightarrow \mathscr V\longrightarrow
	\mathcal O_{\mathscr C}(1)^{\oplus 3}
	\xrightarrow{(x^2,y^2,z^2)}
	\mathcal O_{\mathscr C}(3)\longrightarrow0.
\end{equation}
If $p\equiv2\pmod 5$, then $ F_{C_p}^{*}V_p$ is not semistable and
\begin{equation}\label{eq:brenner-HN}
	0\longrightarrow\mathcal O_{C_p}(1)
	\longrightarrow F_{C_p}^{*}V_p
	\longrightarrow\mathcal O_{C_p}(-1)\longrightarrow0
\end{equation}
is its Harder--Narasimhan filtration
(see \cite[Lemma~2.1]{BrennerKaid-descent}).
Note that  the primes $p\equiv2\pmod5$ have Dirichlet density $1/4$. Thus the
bad-prime set is dense in $\Spec R$ and it even has a positive arithmetic density.

On the other hand, by \cite[Corollary~2]{Brenner-Miyaoka} the restriction $V$ of $\mathscr V$ to the generic geometric fiber $\mathscr C _{\bC}$ is semistable. So for all sufficiently large primes $p$, the bundle
$V_p$ is semistable. A strictly semistable rank two bundle of degree zero
on a smooth projective curve is an extension of two degree-zero line
bundles and is therefore strongly semistable. Thus \eqref{eq:brenner-HN}
implies that for large $p\equiv2\pmod 5$, the bundles $V_p$ are geometrically stable. So by
openness of geometric stability (see \cite[Proposition 2.3.1]{Huybrechts-Lehn2010}),  $V$ is stable.

Since \(V\) is stable with trivial determinant, it is unitary flat
by the Narasimhan--Seshadri theorem. We will also use that
\begin{equation}\label{eq:rank-six-brenner-holonomy}
	\operatorname{Hol}(V_{\mathbb C})
	=\operatorname{SL}_2.
\end{equation}
If the
identity component of its algebraic holonomy were a proper subgroup
of \(\operatorname{SL}_2\), it would be a torus or trivial. After a
finite \'etale cover killing the component group
\cite[Lemma~40]{ssb-BalajiKollar2008}, the pullback of \(V\) would
therefore split into degree-zero line bundles. The cover and splitting 
can be defined over a finite extension of \(\mathbb Q\) and spread
out. Every sufficiently good reduction of that pullback would then be
strongly semistable. Strong semistability descends under finite
surjective morphisms of smooth curves, since a destabilizing subsheaf
remains destabilizing after pullback, including after every Frobenius
pullback. This contradicts \eqref{eq:brenner-HN} at primes above
sufficiently large \(p\equiv2\pmod5\), and proves
\eqref{eq:rank-six-brenner-holonomy}. 

Set
\[\mathscr X=\mathscr C\times_R\mathbb P^1_R,\qquad
A=p_1^*\mathcal O_{\mathscr C}(1),\qquad
B=p_2^*\mathcal O_{\mathbb P^1}(1), \]
and polarize the fibers by $\cH:=A\otimes B$. We write $a=c_1(A)$ and
$b=c_1(B)$. On every surface fiber, we have
	$a^2=b^2=0$ and $ a\cdot b=5$.
Throughout this section we write $X=\cX _{\bC}$ and $H=\cH_{\bC}$. 
Set $G:=p_1^*V$. Since $\pi_1(X)\simeq\pi_1(C)$, the bundle $G$
is unitary flat with
$\overline{\operatorname{Hol}^{\mathrm{an}}(G)}=\operatorname{SU}(2)$.

The morphism $g\colon C\to\mathbb P^1$, $ [x:y:z]\longmapsto[x:y]$,
has degree five. Composing $g\times\operatorname{id}_{\mathbb P^1}$
with the Segre embedding and a linear projection from a rational point
outside the Segre quadric gives a finite morphism
\(f\colon X\to\mathbb P^2\), defined over $\mathbb Q$, with
\begin{equation}\label{eq:rank-six-finite-map}
	f^*\mathcal O_{\mathbb P^2}(1)\simeq H,
	\qquad \deg f=H^2=10.
\end{equation}
We use the same symbols for these objects over $\mathbb Q$ and after
extension to $\mathbb C$ when no confusion can arise.

The following lemma will be used for both stable tensor-product examples.

\begin{lemma}\label{lem:counterexamples-tensor-holonomy}
	Let $F$ be an $H$-stable vector bundle of rank $r\geq2$ on $X$ such that $	\det F\simeq\mathcal O_X,$
	$	\operatorname{Hol}(F)=\operatorname{SL}_r$ and $	\int_Xc_2(F)>0.$
	Then $E:=F\otimes G$ is $H$-stable and
	\(
	\operatorname{Hol}(E)
	=\operatorname{im}\bigl(
	\operatorname{SL}_r\times\operatorname{SL}_2
	\xrightarrow{\,\otimes\,}\operatorname{SL}_{2r}\bigr).
	\)
\end{lemma}

\begin{proof}
Equip $F$ with its Hermite--Einstein connection, with trivial
determinant connection. By \cite[Theorem~3.1]{Biswas2011},
its closed analytic holonomy group is $\operatorname{SU}(r)$.
The Lie algebra of its restricted analytic holonomy is invariant
under the full analytic holonomy group and hence under its closure.
It is therefore an ideal in $\mathfrak{su}(r)$.
This ideal is nonzero, since otherwise the connection would be
flat, contradicting $\int_Xc_2(F)>0$.
Since $\mathfrak{su}(r)$ is simple, the restricted analytic
holonomy group is also $\operatorname{SU}(r)$.
	
	Let 
	\(
	J:=\overline{\operatorname{Hol}^{\mathrm{an}}(F\oplus G)}\subseteq\operatorname{SU}(r)\times\operatorname{SU}(2)
	\)
	be the closure of the joint analytic holonomy group of $F\oplus G$.
	Parallel transport around contractible loops is trivial on $G$
	and generates the restricted analytic holonomy on $F$. Thus $J$ contains
	$\operatorname{SU}(r)\times\{1\}$. Its projection onto the second
	factor is surjective, so
	\(
	J=\operatorname{SU}(r)\times\operatorname{SU}(2).
	\)
	The tensor product connection on $E$ is Hermite--Einstein, and its
	closed analytic holonomy group $\overline{\operatorname{Hol}^{\mathrm{an}}(E)}$
	is the image of $J$ under the tensor product representation.
	This representation is irreducible, so the
	Donaldson--Uhlenbeck--Yau theorem shows that $E$ is $H$-stable.
By \cite[Theorem~3.1]{Biswas2011}, the algebraic holonomy
of $E$ is the Zariski closure of the holonomy of its
Hermite--Einstein connection. Hence
$\operatorname{Hol}(E)$ is as asserted.
\end{proof}

\subsection{A strictly semistable rank three bundle with nonzero
	discriminant}
The following proposition shows that stability in characteristic zero cannot be weakened to
semistability in a rank three analogue of Shepherd--Barron's theorem.

\begin{proposition}\label{prop:semistable-rank-three-syzygy}
	Let $\mathscr V$ denote the syzygy model over $R$ of the bundle $V$
	in \eqref{eq:brenner-quintic}, and set
	$\mathscr E:=p_1^*\mathscr V\oplus(A\otimes B^{-1}).$
	The characteristic-zero fiber $E$ on $X$ of $\mathscr E$ is
	$H$-semistable but not $H$-stable. For every prime
	$p\equiv2\pmod5$, the bundle $F_{X_p}^*E_p$ is not semistable.
	Moreover,	$c_1(E)=a-b$, $c_2(E)=0$,  $\Delta(E)=4ab$,
	and hence $\int_X\Delta(E)=20$.
\end{proposition}

\begin{proof}
	By construction,	$E=p_1^*V\oplus(A\otimes B^{-1}).$
	The bundle $G=p_1^*V$ is unitary flat and hence
	$H$-polystable of slope zero. Furthermore,
	\[
	\mu_H(A\otimes B^{-1})=(a-b)(a+b)=0.
	\]
	Thus $E$ is a direct sum of semistable bundles of the same
	slope, so it is semistable but  not stable. 
	
	For $p\equiv2\pmod5$, pulling back the first arrow in
	\eqref{eq:brenner-HN} gives an inclusion
	\[
	A_p\longrightarrow F_{X_p}^*(p_1^*V_p)
	\longrightarrow F_{X_p}^*E_p.
	\]
	Since
	\(
	\mu_{H_p}(A_p)=5>0=\mu_{H_p}(F_{X_p}^*E_p),
	\)
	the bundle $F_{X_p}^*E_p$ is not semistable.
	Finally, $c_1(p_1^*V)=c_2(p_1^*V)=0$, so the direct-sum
	decomposition gives $c_1(E)=a-b$ and $c_2(E)=0$.
	Consequently,
	\(
	\Delta(E)=6c_2(E)-2c_1(E)^2
	=-2(a-b)^2=4ab.
	\)
	As $\int_Xab=5$, this gives $\int_X\Delta(E)=20$.
\end{proof}

\subsection{A stable rank four bundle with nonzero discriminant}

Theorem~\ref{thm:higher-rank-SB} cannot, in arbitrary
rank, be strengthened by asserting that dense Frobenius instability forces
the discriminant to vanish. The following proposition shows that this fails already in rank four.

\begin{proposition}\label{prop:rank-four-counterexample}
	There exists a rank $4$, $H$-stable vector bundle $E$ on $X$, defined over $\mathbb Q$,
	such that
	\(
	\det E\simeq\mathcal O_X\) and
	$\int_X\Delta(E)>0$,
	whereas, for a set of primes of Dirichlet density $1/4$, the reduction
	$E_p$ is stable and its first Frobenius pullback is not semistable.
\end{proposition}

\begin{proof}
	Choose a rank two geometrically stable vector bundle $S$ on $\mathbb P^2_{\mathbb Q}$ with
	trivial determinant. Any such
	$S$ has algebraic holonomy $\mathrm{SL}(2)$: since $\mathbb P^2(\mathbb C)$
	is simply connected, its algebraic holonomy is connected by
	\cite[Proposition~4]{ssb-BalajiKollar2008}.
	It is a reductive subgroup of $\mathrm{SL}(2)$ acting irreducibly,
	and hence equals $\mathrm{SL}(2)$.
	Moreover, stability and the Riemann--Roch theorem give
	\[
	\int_{\mathbb P^2}c_2(S)=2+h^1(\mathbb P^2,S)>0.
	\]
	For example, one can take as $S$  the bundle $S_0$ defined by
	\begin{equation*}
		0\longrightarrow S_0\longrightarrow
		\mathcal O_{\mathbb P^2}(1)^{\oplus3}
		\xrightarrow{(x^2,y^2,z^2)}
		\mathcal O_{\mathbb P^2}(3)\longrightarrow0.
	\end{equation*}
	Its determinant is trivial, and its geometric stability follows from
	$H^0(\mathbb P^2,S_0)=0$.
	Set	
$	F:=f^*S $ and 
$ E:=F\otimes G.$
	By Corollary~\ref{ssb:cor-derived-holonomy-pullback}, $F$ and all its
	symmetric powers are $H$-stable, and $F$ has algebraic holonomy
	group $\mathrm{SL}(2)$. Moreover, we have
	\[
	\int_Xc_2(F)=10\int_{\mathbb P^2}c_2(S)>0.
	\]
	
	Lemma~\ref{lem:counterexamples-tensor-holonomy} now shows that $E$
	is $H$-stable and
	\(
	\operatorname{Hol}(E)
	=\operatorname{im}\bigl(
	\operatorname{SL}_2\times\operatorname{SL}_2
	\xrightarrow{\,\otimes\,}\operatorname{SL}_4\bigr).
	\)
	
	Both $F$ and $G$ have trivial determinant, and $c_2(G)=0$ because
	$G$ is pulled back from a curve. The tensor-product discriminant formula
	therefore gives
	\[
	\int_X\Delta(E)
	=4\int_X\Delta(F)+4\int_X\Delta(G)=40\int_{\mathbb P^2}\Delta(S)
	=160\int_{\mathbb P^2}c_2(S)>0.
	\]
	
	All the algebraic data are defined over $\mathbb Q$ and spread out over
	$\mathbb Z[1/N]$ for some $N$. After increasing $N$, the reductions
	$E_p$ are stable for every $p\nmid N$. For $p\equiv2\pmod5$ with
	$p\nmid N$, pulling back \eqref{eq:brenner-HN} to $X_p$ and tensoring
	with $F_{X_p}^*F_p$ gives a rank two subbundle
$		F_{X_p}^*F_p\otimes A_p		\to F_{X_p}^*E_p.$
	Since $c_1(F_p)=c_1(E_p)=0$, its slope is
	\[
	\mu_{H_p}(F_{X_p}^*F_p\otimes A_p)
	=a(a+b)=5>0=\mu_{H_p}(F_{X_p}^*E_p).
	\]
	Thus the first Frobenius pullback is not semistable. The primes
	$p\equiv2\pmod5$ with $p\nmid N$ have Dirichlet density $1/4$.
\end{proof}

\begin{lemma}\label{lem:symmetric-powers-rank-four-example}
	Let \(E=F\otimes G\) be the bundle constructed in
	Proposition~\ref{prop:rank-four-counterexample}. Then
	\(\operatorname{Sym}^mE\) is \(H\)-polystable for every \(m\geq1\), but it
	is not \(H\)-stable for any \(m\geq2\). More precisely, we have
	\[
	\operatorname{Sym}^mE
	\simeq
	\bigoplus_{i=0}^{\lfloor m/2\rfloor}
	\operatorname{Sym}^{m-2i}F
	\otimes
	\operatorname{Sym}^{m-2i}G .
	\]
	In particular, $E$ is pseudoeffective.	
\end{lemma}

\begin{proof}  Since \(E\) is \(H\)-stable,	\(\operatorname{Sym}^mE\) is \(H\)-polystable (see \cite[Theorem 3.2.11]{Huybrechts-Lehn2010}).
	The Cauchy decomposition (see \cite[Corollary~2.3.3]{Weyman-Syzygies})  
	gives
	\[
	\operatorname{Sym}^m(F\otimes G)
	\simeq
	\bigoplus_{i=0}^{\lfloor m/2\rfloor}
	\mathbb S_{(m-i,i)}F\otimes
	\mathbb S_{(m-i,i)}G,
	\]
	where $\mathbb S$ denotes the bundle associated to the corresponding Schur module.  	
	Since \(\det F\) and \(\det G\) are trivial this gives the asserted decomposition (see Subsection \ref{subsection:Cartan}).
	For \(m\geq2\), the Cauchy decomposition contains at least two nonzero
	direct summands, so \(\operatorname{Sym}^mE\) is not
	\(H\)-stable. 
	The last assertion comes from the fact that \(\mathcal O_{\mathbb P(E)}(2)\) has a nonzero section as \(\operatorname{Sym}^2E\) contains $\cO_X$ as a direct summand.
\end{proof}

\subsection{A rank six counterexample to pseudoeffectivity}

In this subsection we show that even if $E$ satisfies the assumptions of
Theorem~\ref{thm:pseudoeffective-exterior-power} and its determinant is
trivial, $E$ need not be pseudoeffective.

We replace the rank two factor pulled back from \(\mathbb P^2\) in
Proposition~\ref{prop:rank-four-counterexample} by a rank three factor
with full special-linear holonomy. The resulting rank six bundle is
not pseudoeffective, although its third exterior power is
pseudoeffective. Thus the exterior-power pseudoeffectivity conclusion
cannot in general be replaced by pseudoeffectivity of the original
bundle.

\begin{proposition}\label{prop:rank-six-counterexample}
	On the polarized surface \((X,H)\) fixed above, there exists a rank six vector
	bundle \(E\), defined over \(\mathbb Q\), with the following properties:
	\begin{enumerate}
		\item \(E_{\mathbb C}\) is \(H\)-stable, $\det E\simeq\mathcal O_X$ and
		\(\int_Xc_2(E)=120.	\)
		\item Its algebraic holonomy group is
		\(
		\operatorname{SL}_3\times\operatorname{SL}_2
		\hookrightarrow\operatorname{SL}_6,
		\)
		acting through the tensor product of the standard representations.
		\item After spreading out over \(\mathbb Z[1/N]\), every sufficiently
		large prime \(p\equiv2\pmod5\) gives a geometrically stable reduction
		\(E_{p}\) whose first Frobenius pullback is not \(H_p\)-semistable.
		In particular, the non-strongly-semistable reduction locus contains
		a set of primes of Dirichlet density \(1/4\) and is Zariski dense.
		\item \(E_{\mathbb C}\) is not pseudoeffective, whereas
		\(\bigwedge^3E_{\mathbb C}\) is pseudoeffective.
	\end{enumerate}
\end{proposition}

The non-pseudoeffectivity assertion requires a uniform vanishing
statement for Schur powers. We obtain it from the closed-orbit
criterion of Theorem~\ref{ssb:thm-holonomy-orbit-positivity}.

\subsubsection{Dominant classes on a rank three flag bundle}
\begin{lemma}\label{lem:rank-three-flag-non-psef}
	Let \((X,H)\) be a smooth polarized complex projective surface, and let
	\(F\) be an \(H\)-stable vector bundle of rank three such that
	$\det F\simeq\mathcal O_X$,
	$\operatorname{Hol}(F)=\operatorname{SL}_3$ and
	$\int_Xc_2(F)>0.$
	Let
	\(
	p:\mathcal F:=\operatorname{Fl}(F)\longrightarrow X
	\)
	be the full flag bundle that we identify with the incidence
	divisor in
	\(
	\mathbb P(F)\times_X\mathbb P(F^*).
	\)
	Let us denote by \(\ell_1,\ell_2\) the restrictions of the tautological
	classes from the first and second factors. Then for all $a,b\in\mathbb R_{\geq0},\ (a,b)\neq(0,0)$, the divisor
	$ a\ell_1+b\ell_2$ is not pseudoeffective.
\end{lemma}

\begin{proof}
	Theorem~\ref{ssb:thm-holonomy-orbit-positivity}, applied to $F$
	and $F^*$, shows that their tautological classes are not
	pseudoeffective: their closed-orbit bundles are the whole
	projective bundles, and $\Delta(F)=\Delta(F^*)=6c_2(F)$.
	Since pseudoeffectivity is preserved and detected by surjective
	pullback, this proves the assertion when $a=0$ or $b=0$.
	
	Suppose that $a,b>0$, put $D=a\ell_1+b\ell_2$, and assume that
	$D$ is pseudoeffective. We first show that for every prime divisor
	$Z\subset\mathcal F$ and every real $t>0$, the class $D-tZ$ is
	not pseudoeffective.
	
	The description of the Picard group of the flag bundle gives
	\[
	\mathcal O_{\mathcal F}(Z)
	\simeq\mathcal O_{\mathcal F}(u,v)\otimes p^*\mathcal O_X(B)
	\]
	for some integers $u,v$ and a Cartier divisor $B$ on $X$. If $Z$
	dominates $X$, restriction of its defining section to the generic
	fiber shows that $u,v\geq0$ and $(u,v)\neq(0,0)$.
The	Borel--Weil theorem identifies
	\(
	\mathcal A_{u,v}:=p_*\mathcal O_{\mathcal F}(u,v)
	\)
	with the bundle associated to the irreducible
	$\operatorname{SL}_3$-representation of highest weight
	$u\omega_1+v\omega_2$. Thus $\mathcal A_{u,v}$ is $H$-stable,
	has trivial determinant, and has rank greater than one.
	The section defining $Z$ induces a nonzero homomorphism
	$\mathcal O_X(-B)\to\mathcal A_{u,v}$, so stability gives
	$B\cdot H>0$. If $Z$ does not dominate $X$, it is the inverse
	image of a prime divisor on $X$, since $p$ is smooth with
	geometrically integral fibers. In this case $u=v=0$ and $B$ is
	that prime divisor, so again $B\cdot H>0$.
	
	Choose a sufficiently ample very general smooth curve
	$C\in|mH|$ such that $F|_C$ is semistable. Its degree is zero,
	so $F|_C$ and $F^*|_C$ are nef. Therefore on
	$\mathcal F_C=p^{-1}(C)$ the real class $D_C:=D|_{\mathcal F_C}$
	is nef. The incidence description and $c_1(F|_C)=0$ give
	\begin{equation*}
		D_C^4=0,\qquad
		D_C^3\cdot Z|_{\mathcal F_C}
		=3ab(a+b)\,B\cdot C>0.
	\end{equation*}
	Indeed, the terms involving four fundamental classes vanish,
	whereas the degree of $(a\ell_1+b\ell_2)^3$ on a flag fiber is
	$3ab(a+b)$.
	
	If $D-tZ$ were pseudoeffective, we could also choose $C$ so that
	its restriction to $\mathcal F_C$ were pseudoeffective. To see
	this for a real class, approximate $D-tZ$ by effective real
	divisors and choose $C$ so that $\mathcal F_C$ is contained in
	none of their supports. Intersecting the restricted class with
	the nef class $D_C^3$ would then give
	\[
	0\leq D_C^3\cdot(D-tZ)|_{\mathcal F_C}
	=-3tab(a+b)\,B\cdot C<0,
	\]
	a contradiction.
	
	Consider the divisorial Zariski decomposition
	$D=P_\sigma(D)+N_\sigma(D)$. If a prime divisor $Z$ occurred in
	$N_\sigma(D)$ with positive coefficient, then $D-tZ$ would be
	pseudoeffective for some $t>0$. Hence $N_\sigma(D)=0$.
By \cite[Chapter~III, Lemma~1.8 and Proposition~1.14(1)]{Nakayama2004},
the class $D=P_\sigma(D)$ belongs to the closed movable cone.
Choose rational divisor classes $M_j\to D$ such that some positive
integral multiple of each $M_j$ has a nonempty linear system
with no fixed divisorial components. Two general members of
such a system intersect properly, so $M_j^2$ is represented
by an effective rational cycle. Passing to the limit shows
that $D^2$ is a pseudoeffective codimension-two class.
	
	Let $q:\mathcal F\to\mathbb P(F)$ be the first projection and
	let $\xi=c_1(\mathcal O_{\mathbb P(F)}(1))$. The incidence
	description gives
	\(
	q_*(D^2)=b(2a+b)\,\xi.
	\)
	For completeness, if $x,y$ are the tautological classes on
	$\mathbb P(F)\times_X\mathbb P(F^*)$, the incidence divisor has
	class $x+y$. Pushing $(ax+by)^2(x+y)$ to the first factor,
	only the terms $(2ab+b^2)xy^2$ and $b^2y^3$ contribute; the
	second contribution is zero because $c_1(F^*)=0$.
	Proper pushforward preserves pseudoeffectivity of cycle classes.
	Thus $b(2a+b)>0$ implies that $\xi$ is pseudoeffective, contrary
	to the first paragraph. This proves the assertion for all real
	$a,b\geq0$ with $(a,b)\neq(0,0)$.
\end{proof}

\begin{corollary}\label{cor:rank-six-uniform-schur-vanishing}
	Under the assumptions of Lemma~\ref{lem:rank-three-flag-non-psef},
	there exists \(\varepsilon>0\) such that
	\[
	H^0\!\left(
	X,\mathbb S_{(\lambda_1,\lambda_2)}F
	\otimes\mathcal O_X(kH)
	\right)=0
	\]
	whenever
	\[
	\lambda_1\geq\lambda_2\geq0,\qquad
	\lambda_1+\lambda_2>0,\qquad
	0\leq k<\varepsilon(\lambda_1+\lambda_2).
	\]
\end{corollary}

\begin{proof}
	By the Borel--Weil theorem, a nonzero section would give an effective class
	\[
	(\lambda_1-\lambda_2)\ell_1+
	\lambda_2\ell_2+kp^*H
	\]
	on \(\mathcal F\). The compact segment
	\[
	\left\{
	a\ell_1+b\ell_2:
	a,b\geq0,\ a+2b=1
	\right\}
	\]
	is disjoint from the closed pseudoeffective cone by
	Lemma~\ref{lem:rank-three-flag-non-psef}.
	It therefore remains disjoint after adding \(tp^*H\) for all
	sufficiently small \(t\geq0\). Divide the displayed effective class
	by \(\lambda_1+\lambda_2\) to obtain the asserted uniform bound.
\end{proof}

\subsubsection{The rank three factor}

On the target plane of \(f\), with coordinates \([s:t:w]\), define a
rank three bundle \(T\) by
\begin{equation}\label{eq:rank-six-plane-factor}
	0\longrightarrow T\longrightarrow
	\mathcal O_{\mathbb P^2}(1)^{\oplus4}
	\xrightarrow{(s^3,t^3,w^3,stw)}
	\mathcal O_{\mathbb P^2}(4)\longrightarrow0.
\end{equation}
The map is surjective, and we have
\begin{equation}\label{eq:rank-six-plane-chern}
	\operatorname{rk}T=3,\qquad
	\det T\simeq\mathcal O_{\mathbb P^2},\qquad
	c_2(T)=6h^2,
	\qquad h=c_1(\mathcal O_{\mathbb P^2}(1)).
\end{equation}
The four cubics have no linear syzygies, so \(H^0(T)=0\). The dual
sequence gives \(H^0(T^*)=0\), and
\(\bigwedge^2T\simeq T^*\). These vanishings imply stability:
a saturated destabilizing subsheaf of rank \(i=1\) or \(2\) would have
determinant \(\mathcal O_{\mathbb P^2}(e)\), with \(e\geq0\), and
would yield a nonzero section of \(\bigwedge^iT(-e)\), hence of
\(\bigwedge^iT\), a contradiction.

Moreover,
\begin{equation}\label{eq:rank-six-plane-holonomy}
	\operatorname{Hol}(T_{\mathbb C})
	=\operatorname{SL}_3.
\end{equation}
Since \(\mathbb P^2_{\mathbb C}\) is simply connected, this holonomy is
connected. The rank three classification
\cite[Proposition~4 and paragraph~45]{ssb-BalajiKollar2008}
leaves only the special-linear and orthogonal possibilities. The latter
would imply \(T\simeq T^*\), whereas
\(
h^0(T(1))=3,\) and  $h^0(T^*(1))=4.$
For the first equality, every quintic monomial is divisible by one of
\(s^3,t^3,w^3,stw\). Thus the map on sections of
\eqref{eq:rank-six-plane-factor}, twisted by \(\mathcal O(1)\), is
surjective and has kernel of dimension \(24-21=3\). The second equality
follows from the dual sequence.

With \(f\) and \(G\) as above, we set
$F:=f^*T$ and $E:=F\otimes G$.
All the algebraic data are defined over \(\mathbb Q\).

\begin{proof}[Proof of Proposition~\ref{prop:rank-six-counterexample}]
	By Corollary~\ref{ssb:cor-derived-holonomy-pullback} and
	\eqref{eq:rank-six-plane-holonomy}, the bundle $F=f^*T$ is
	$H$-stable and has algebraic holonomy $\operatorname{SL}_3$.
	By \eqref{eq:rank-six-finite-map} and
	\eqref{eq:rank-six-plane-chern},
	\[
	\int_Xc_2(F)=6\deg f=60.
	\]
	Lemma~\ref{lem:counterexamples-tensor-holonomy} therefore shows
	that $E=F\otimes G$ is $H$-stable and that its algebraic holonomy
	is the image of $\operatorname{SL}_3\times\operatorname{SL}_2$.
	This tensor representation is faithful, since its kernel consists
	of scalar pairs \((\zeta I_3,\zeta^{-1}I_2)\) with
	\(\zeta^3=\zeta^2=1\).
	
	Both factors have trivial determinant, and \(c_2(G)=0\), since
	$G$ is pulled back from $C$. Therefore $	\det E\simeq\mathcal O_X$ and
	\[
	\int_Xc_2(E)=2\int_Xc_2(F)+3\int_Xc_2(G)=120.
	\]
	In particular, \(\int_X\Delta(E)=1440\).
	
	Spread out the construction over \(\mathbb Z[1/N]\), increasing
	\(N\) so that all remaining reductions of \(E\) are geometrically
	stable. For $p\equiv2\pmod5$ with $p\nmid N$, the proof of Proposition \ref{prop:rank-four-counterexample}
  gives the subbundle
	$F_{X_p}^*F_p\otimes A_p\subset F_{X_p}^*E_p$.
	Its slope is $5$, whereas the ambient bundle has slope zero.
	Hence the first Frobenius
	pullback is not semistable. The density assertion follows from the
	same set of primes $p\equiv2\pmod5$.
	
	It remains to prove non-pseudoeffectivity. 
	We first record the uniform injections
	\begin{equation}\label{eq:rank-six-flat-uniform-injection}
		\operatorname{Sym}^dG
		\hookrightarrow\mathcal O_X(3H)^{\oplus N_d}
		\qquad(d\geq0).
	\end{equation}
Indeed, $C$ has genus six,  and the bundles
	\(
	(\operatorname{Sym}^dV)^*\otimes\mathcal O_C(3)
	\)
	are globally generated. For any point \(c\in C\), the first
	cohomology after twisting by \(\mathcal O_C(-c)\) vanishes by Serre
	duality: the dual space is the space of sections of the semistable
	bundle
	\(
	\operatorname{Sym}^dV\otimes\omega_C(c)
	\otimes\mathcal O_C(-3),
	\)
	whose slope is \(-4\). Pull back to \(X\) and tensor with
	\(B^{\otimes3}\). Dualizing evaluation for the resulting globally
	generated bundle gives \eqref{eq:rank-six-flat-uniform-injection}.
	
	Suppose that \(E\) is pseudoeffective. The characterization by
	sections of symmetric powers gives positive integers \(n_j,k_j\)
	such that
	\[
	n_j\longrightarrow\infty,\qquad k_j/n_j\longrightarrow0,
	\qquad H^0(X,\operatorname{Sym}^{n_j}E(k_jH))\neq0.
	\]
	Apply the Cauchy decomposition (see \cite[Corollary~2.3.3]{Weyman-Syzygies})  
	to \(F\otimes G\). Since
	\(\operatorname{rk}G=2\) and \(\det G\simeq\mathcal O_X\), it gives
	\[
	\operatorname{Sym}^nE\simeq
	\bigoplus_{\substack{\lambda_1\geq\lambda_2\geq0\\
			\lambda_1+\lambda_2=n}}
	\mathbb S_{(\lambda_1,\lambda_2)}F
	\otimes\operatorname{Sym}^{\lambda_1-\lambda_2}G.
	\]
	For each \(j\), some partition in this decomposition therefore gives
	a nonzero section of
	\(
	\mathbb S_{(\lambda_1,\lambda_2)}F
	\otimes\operatorname{Sym}^{\lambda_1-\lambda_2}G(k_jH).
	\)
	Using \eqref{eq:rank-six-flat-uniform-injection}, we obtain
	\[
	H^0\!\left(X,
	\mathbb S_{(\lambda_1,\lambda_2)}F((k_j+3)H)
	\right)\neq0.
	\]
	This contradicts
	Corollary~\ref{cor:rank-six-uniform-schur-vanishing}, because \(\lambda_1+\lambda_2=n_j\) and
	\((k_j+3)/n_j\to0\).
	
	Finally, the exterior Cauchy decomposition (see \cite[Corollary~2.3.3]{Weyman-Syzygies})  gives a direct summand
	\[
	\det F\otimes\operatorname{Sym}^3G
	\simeq\operatorname{Sym}^3G
	\subset{\bigwedge}^3E.
	\]
	This nonzero summand is unitary flat, hence nef and pseudoeffective.
	The inclusion carries its nonzero symmetric-power sections, after
	arbitrarily small ample twists, to those of \(\bigwedge^3E\).
	Consequently \(\bigwedge^3E\) is pseudoeffective.
\end{proof}

\begin{remark}
	The commutator of the algebraic holonomy of \(E\) is
	\(\operatorname{SL}_3\times\operatorname{SL}_2\), acting on
	\(\mathbb C^3\otimes\mathbb C^2\).
	This representation has no nonconstant polynomial invariants:
	already \(\operatorname{SL}_3\) acts with an open orbit on pairs of
	linearly independent vectors in \(\mathbb C^3\). The flag-bundle
	argument supplies the stronger uniform vanishing needed to prove
	non-pseudoeffectivity.
\end{remark}

\section{Étale trivializable reductions and infinite monodromy}
\label{sec:double-covers}

In this section we give counterexamples to \cite[Conjecture~7.7]{LangerPositivitySemistability}
that was an analogue of the Grothendieck--Katz conjecture for Higgs bundles. 
The conjecture fails already for line bundles on elliptic curves. 
We also show that it fails for stable bundles with trivial determinant, so
the counterexample is not just a problem with the determinant being non-torsion. 

In this section $C$ denotes an elliptic curve  over $\mathbb Q$ with origin \(O\) and having a rational point $P$
of infinite order and a nonzero rational point of order two.  Existence of an elliptic curve satisfying these assumptions is well-known (see, e.g., \cite{Fermigier1996} or \cite[\S6.1, Example 1, p.~193]{PoonenEllipticCurves}).

\subsection{A non-torsion line bundle counterexample}

The following proposition shows that 	\cite[Conjecture~7.7]{LangerPositivitySemistability} is false already in the simplest possible cases.

\begin{proposition}\label{prop:prime-to-p-line-bundle}
Let \(L:=\mathcal O_C(P-O)\). Then  $L_{\mathbb C}$ is not \'etale trivializable but	for a suitable model $(\mathcal C,\mathcal L)\to S$ over a nonempty open
subscheme $S\subset\operatorname{Spec}\mathbb Z$, the reduction
$\mathcal L_s$ is \'etale trivializable for every closed point $s\in S$.
\end{proposition}

\begin{proof}
Since $C$ has a point of order $2$, its reductions to almost all primes also give a point of order $2$ showing that
$2\mid  N_p:=\#\mathcal C_s(\mathbb F_p).$
On the other hand, Hasse's bound gives
	\begin{equation*}
		0<N_p
		\leq p+1+2\sqrt p<2p.
	\end{equation*}
for  $p\geq7$.	So 	\( p\nmid\#\mathcal C_s(\mathbb F_p)\) for large $p$.

By construction, $L$ has infinite order in $\operatorname{Pic}^0(C)$, so $L_{\mathbb C}$ is not \'etale trivializable. 
On the other hand, for large primes $p$ the reduction $\mathcal L_s$
has order dividing $N_p$. Since $p\nmid N_p$, its Kummer torsor
is finite \'etale and trivializes $\mathcal L_s$.
\end{proof}

\begin{remark}
	Although an example of line bundles on abelian varieties was studied in \cite[Example 7.8]{LangerPositivitySemistability}, the author did not notice that  existence of a rational $2$-torsion on an elliptic curve over $\bQ$ forces the number of points of the reduction to be non-divisible by large $p$. This fact was known as mentioned in \cite[p.~420]{Katz2009}.
\end{remark}

\subsection{A stable rank two counterexample} 

The geometric construction used below is a special case of Mumford's
construction for Prym varieties of hyperelliptic double covers
\cite[p.~346]{MumfordPrymI}.  The associated rank two bundle is the
Prym--Narasimhan--Ramanan construction
\cite{NarasimhanRamananGeneralisedPrym} (see also
\cite[\S2.3, Proposition~2.2]{DucrohetFrobenius}).

Apply Mumford's construction to the quotient of $C$ by its
involution $[-1]$ and to a double cover of $\mathbb P^1$ ramified at two
rational points outside the branch locus of the quotient map.  This gives
a smooth projective geometrically connected curve $Y$ of genus three, a
fixed-point-free involution $\sigma$ on $Y$, and a smooth projective curve
$X$ of genus two such that the quotient map
$\pi\colon Y\to X$ is a connected \'etale double cover.  Moreover, there is
a morphism $q\colon Y\to C$ of degree two satisfying
$q\circ\sigma=[-1]\circ q$.  Pullback by $q$ identifies
$C\simeq\operatorname{Pic}^0(C)$ with the Prym variety
$\operatorname{Prym}(Y/X)$.

Let $\ell=\mathcal O_C(P-O)$ and put $L=q^*\ell$.  Then $L$ is a
non-torsion line bundle of degree zero.  Indeed, its norm with respect to
$q$ is $\ell^{\otimes 2}$.  The Prym property gives
$\operatorname{Nm}_{\pi}(L)\simeq\mathcal O_X$, and the relation between
$q$ and $\sigma$ gives $\sigma^*L\simeq L^{-1}$.

Let $\eta\in\operatorname{Pic}^0(X)[2]$ be the nonzero line bundle defining
the cover $\pi$.  Thus
$\pi_*\mathcal O_Y\simeq\mathcal O_X\oplus\eta$.  The determinant formula
for a finite flat morphism shows that
$\det(\pi_*L)\simeq\eta$.  

After replacing $\mathbb Q$ by a finite extension $K$, choose a line bundle
$B\in\operatorname{Pic}^0(X_K)$ such that
$B^{\otimes2}\simeq\eta_K$.  The bundle $B$ has order four.  Set
$\mathcal F=(\pi_*L)_K$ and
$\mathcal E=\mathcal F\otimes B$.  Then
$\det\mathcal E\simeq\mathcal O_{X_K}$.
We henceforth base change $\pi$ and $L$ to $K$ without changing notation.

\begin{theorem}\label{thm:stable-rank-two-counterexample}
	$\mathcal E$ is a geometrically stable rank two vector bundle with trivial determinant.  For every embedding
	$K\hookrightarrow\mathbb C$, the bundle $\mathcal E_{\mathbb C}$ is not
	\'etale trivializable.  On the other hand, there is a model
	$(\mathcal X,\mathscr E)\to S$ over a nonempty open subscheme
	$S\subset\operatorname{Spec}\mathcal O_K$ such that $\mathscr E_s$ is
	\'etale trivializable for every closed point $s\in S$.  After shrinking
	$S$, all these reductions are geometrically stable.
\end{theorem}

\begin{proof}
	We first prove the stability of $\mathcal F$.  After base change to an
	algebraic closure, suppose that a line bundle $N$ of nonnegative degree
	admits a nonzero map to $\mathcal F$.  By adjunction, there is a nonzero
	map $\pi^*N\to L$.  Since $L$ has degree zero, this forces
	$\deg N=0$ and $\pi^*N\simeq L$.  The line bundle $\pi^*N$ is
	$\sigma$-invariant, whereas $\sigma^*L\simeq L^{-1}$.  It would follow
	that $L^{\otimes2}$ is trivial, contradicting the fact that $L$ is
	non-torsion.  Thus $\mathcal F$ is geometrically stable. So $\mathcal E$ is geometrically stable as well.
	
	The standard splitting
	$\pi^*\mathcal F\simeq L\oplus L^{-1}$ shows that $\mathcal F$ is not
	\'etale trivializable.  Indeed, if a finite \'etale cover trivialized
	$\mathcal F$, then after base change to $Y$ the pullback of $L$ would be
	a degree-zero line summand of a trivial bundle.  It would consequently
	be trivial.  Taking the norm along this finite cover would make $L$
	torsion, which is impossible.
	
	Since $B$ has finite order, it is itself \'etale trivializable in
	characteristic zero.  If $\mathcal E$ were \'etale trivializable, a common
	finite \'etale refinement trivializing both $\mathcal E$ and $B$ would
	also trivialize $\mathcal F$.  Hence $\mathcal E_{\mathbb C}$ is not
	\'etale trivializable.
	
	Spread out all the data over a suitable open subscheme of
	$\operatorname{Spec}\mathcal O_K$.  Let $s$ be a closed point lying over a
	rational prime $p$, after excluding the finitely many primes of bad
	reduction, the prime $2$, and the primes smaller than $7$. 
As in the proof of Proposition \ref{prop:prime-to-p-line-bundle}, we see that 	
 $p$ cannot divide	$\#C_p(\mathbb F_p)$.
	It follows that the reduction of $P$ has order prime to $p$.  Consequently
	the reductions of $\ell$ and $L=q^*\ell$ also have prime-to-$p$ order,
	including after extension from $\mathbb F_p$ to the residue field of $s$.
	The Kummer torsor of $L_s$ is therefore finite \'etale and trivializes
	$L_s$.  Together with the splitting of $\pi_s^*\mathcal F_s$, this gives
	a finite \'etale cover of $X_s$ trivializing $\mathcal F_s$.
	
	The reduction $B_s$ has order dividing four and is consequently also
	\'etale trivializable.  A common finite \'etale refinement therefore
	trivializes $\mathscr E_s$.  Finally, after shrinking $S$ once more,
	geometric stability of all closed fibers follows from openness of
	stability in a projective family of curves.
\end{proof}

\appendix
\section{Restriction theorems for prescribed basepoints}

The following strong restriction theorem follows easily from \cite[Theorem 5.2]{Langer2004}.
A weaker version that deals only with semistability is the content of \cite{Graf2016}.

\begin{theorem}
	\label{thm:restriction-prescribed-points}
	Let $X$ be a normal projective variety of dimension $n\geq 2$ defined over an algebraically closed field of
	characteristic zero. Let $H$ be an ample line bundle on $X$, and let $E$ be
	a  slope $H$-stable reflexive sheaf of rank $r>0$.
	Fix an integer $s\geq 0$.
	There exists an integer $m_0=m_0(X,H,E,s)$ such that, for every
	$m\geq m_0$ and every set $S=\{x_1,\ldots,x_s\}\subset X_{\mathrm{reg}}$
	of $s$ distinct closed points, a general divisor $D\in |mH|$ passing through $S$
	is integral and normal, it is smooth at every point of $S$, and  the restriction $E|_D$
is	torsion-free and slope $H|_D$-stable.
\end{theorem}

\begin{proof} Clearly, we can assume that  $r\geq 2$.
	Choose a projective resolution $\pi:Y\to X$ which is an isomorphism over
	$X_{\mathrm{reg}}$, and set
$	A:=\pi^*H$ and
	$\widetilde{E}:=\pi^{[*]}E$. Then 
$ A^n=H^n>0.$
	In particular, $A$ is nef and big, and $\widetilde{E}$ is reflexive.
	
Note that $\widetilde{E}$ is slope-stable with respect to $A$.
	Indeed, given a saturated subsheaf	$\widetilde{F}\subset\widetilde{E}$, let $F\subset E$ be the saturated	subsheaf defined by the same subspace of the common generic fiber.
	The two subsheaves agree over a big open subset $X_{\mathrm{reg}}$ of $X$. So
	\(\pi_*c_1(\widetilde{F})=c_1(F)\),
	\(	\pi_*c_1(\widetilde{E})=c_1(E), \)
and hence
	\[
	\mu_A(\widetilde{F})
	=\mu_H(F)
	<\mu_H(E)
	=\mu_A(\widetilde{E}).
	\]

	We next choose a geometric degree bound uniformly in $S$.
	Put $N=s(n+1)+2$.
	There is an integer $m_{0}$ depending only on $(X,H,s)$ such that
	for all $m\geq m_{0}$ the evaluation map
	\[
	H^0(X,\mathcal O_X(mH))\longrightarrow
	H^0(Z,\mathcal O_Z(mH))
	\]
	is surjective for every zero-dimensional subscheme $Z\subset X$ of length
	at most $N$.
	This follows from uniform Serre vanishing for the universal ideal sheaves
	over the projective Hilbert schemes $\operatorname{Hilb}^{\ell}(X)$,
	$1\leq\ell\leq N$.
	Consequently, the system of divisors containing $S$ has base locus exactly
	$S$, separates points outside $S$, and permits arbitrary first-order
	terms at the points of $S$.
	For the last assertion, use the disjoint union of the subschemes defined
	by $\mathfrak m_{x_i}^{2}$, whose total length is $s(n+1)$.
	
	Bertini's theorems now show that a general member $D$ is integral and
	normal away from $S$, and the first-order assertion makes it smooth at $S$.
	The pulled-back system on $Y$ has base locus exactly
	$\widetilde S=\pi^{-1}(S)$ and likewise permits arbitrary first-order terms
	there. Thus a general total transform $\pi^*D$ is smooth.
	We may also require that $D$ contain neither an irreducible component of
	$X_{\mathrm{sing}}$ nor the image of any exceptional prime divisor of $\pi$.
	These are finitely many proper linear conditions: all these subvarieties
	are disjoint from $S$, and the system has no base points outside $S$.
	It follows that $\pi^*D$ has no exceptional divisorial component and equals
	the integral strict transform $\widetilde D$.
	Moreover,
	\(
	\pi_D:=\pi|_{\widetilde D}:\widetilde D\to D
	\)
	is an isomorphism over $D\cap X_{\mathrm{reg}}$, whose complement has
	codimension at least two in $D$.
The divisor $\widetilde D$ belongs to $|mA|$ and is smooth, so
$\widetilde{E}|_{\widetilde D}$ is torsion-free (see \cite[Corollary 1.1.14]{Huybrechts-Lehn2010}).	
Hence by \cite[Theorem~5.2]{Langer2004}, if
	\begin{equation}
		\label{eq:restriction-numerical-bound}
		m_0>
		\frac{r-1}{r}\, 
		\Delta(\widetilde{E}) A^{n-2}
		+\frac{1}{r(r-1)A^n},
	\end{equation}
then $\widetilde{E}|_{\widetilde D}$ is $A|_{\widetilde D}$-stable.
	
We also know that $E|_D$ is torsion-free by \cite[Corollary 1.1.14]{Huybrechts-Lehn2010}.
	For a saturated subsheaf $G\subset E|_D$ of rank strictly
	between $0$ and $r$, take the saturated subsheaf
	$\widetilde{G}\subset
	\widetilde{E}|_{\widetilde D}$ defined by its generic fiber.
	Over $D\cap X_{\mathrm{reg}}$ these sheaves agree so as before we have
	\[
		\mu_{H|_D}(G)
		=\mu_{A|_{\widetilde D}}(\widetilde{G})
		<\mu_{A|_{\widetilde D}}
		(\widetilde{E}|_{\widetilde D})
		=\mu_{H|_D}(E|_D),
	\]
proving stability of $E|_D$.
\end{proof}

\begin{remark}
	For $r\geq 2$, the argument proves the following strong restriction theorem.
	Once \eqref{eq:restriction-numerical-bound} holds and $m\ge m_0$, the restriction is stable
	for every integral normal divisor $D\in|mH|$ such that $\pi^*D$ is integral
	and normal and $D\cap X_{\mathrm{reg}}$ has complement of codimension at
	least two in $D$.
	Generality in Theorem~\ref{thm:restriction-prescribed-points} ensures these
	geometric conditions even after imposing the prescribed points.
\end{remark}

\medskip 
If $E$ is locally free at every point of $S$, successive
restriction and reflexivization give a strong restriction
theorem for general complete intersection curves
\[
C=D_1\cap\cdots\cap D_{n-1},
\qquad D_i\in|m_iH|,\quad S\subset D_i,
\]
and $m_{n-1}\gg \cdots\gg m_1\ge m_0$ (here each degree is chosen sufficiently large after the preceding
degree has been fixed). This can be improved to the following strengthening of
\cite[Theorem 1.1]{Graf2016}.

\begin{theorem}
	\label{cor:restriction-curves-prescribed-points}
	Let $X$, $H$, and $E$ be as in
	Theorem~\ref{thm:restriction-prescribed-points}, and set
	\[ U:=\{x\in X_{\mathrm{reg}}: E \text{ is locally free at }x\}.	\]
	Fix an integer $s\geq0$.
	There exists an integer $m_0=m_0(X,H,E,s)$ such that, for every
	$m_1,\ldots,m_{n-1}\geq m_0$ and every set $S\subset U$ of $s$
	distinct closed points, a general complete intersection
	\[
	C=D_1\cap\cdots\cap D_{n-1},
	\qquad D_i\in|m_iH|,\quad S\subset D_i,
	\]
	is a smooth integral curve contained in $U$, and $E|_C$ is a stable
	vector bundle.
\end{theorem}

\begin{proof} We can assume that $r\geq2$.
	Choose a projective resolution $\pi:Y\to X$, isomorphic over $U$,
	such that $\widetilde E=\pi^{[*]}E$ is locally free.
	As before, $\widetilde E$ is stable with respect to $A=\pi^*H$.
	Choose integers $a_1,\ldots,a_{n-1}\geq2$ successively, sufficiently
	large that repeated application of
	\cite[Theorem~5.2]{Langer2004} gives stable restrictions to general
	complete intersection curves of every multidegree
	\[
	(b_1,\ldots,b_{n-1}),
	\qquad b_i\in\{a_i,a_i+1\}.
	\]
	Indeed, on an intermediate complete intersection of degrees
	$b_1,\ldots,b_j$, the discriminant intersection number and the volume
	appearing in the restriction bound are the corresponding numbers
	on $Y$ multiplied by $b_1\cdots b_j$.
	Thus at each step only finitely many numerical bounds must be satisfied.
	The uniform jet-separation argument in the preceding theorem supplies
	the geometric conditions uniformly in $S$.
	The defining divisors may independently be required either to contain
	$S$ or to avoid $S$.
	The resulting curves lie in $U$, where $\pi$ is an isomorphism.
	
	Every integer $m_i\geq a_i(a_i-1)$ is a sum of copies of $a_i$
	and $a_i+1$.
	Accordingly, degenerate each $D_i$ into a union of general divisors
	of these degrees, requiring exactly one summand to contain $S$
	and the others to avoid $S$.
	Their complete intersection is a reduced nodal curve
	\[
	C_0=\bigcup_\alpha C_\alpha\subset U,
	\]
	smooth along $S$, such that every $E|_{C_\alpha}$ is stable.
	
	The usual specialization argument now gives stability on a general
	smooth deformation of $C_0$ keeping $S$ fixed.
	Indeed, after a finite base change and resolution of the total surface,
	a hypothetical destabilizing subsheaf on the generic fiber extends
	by saturation to a locally free subsheaf.
	Its restrictions inject with the same rank on every component.
	Summing the strict slope inequalities on the original components
	contradicts the destabilizing inequality on the generic fiber.
	Any rational components introduced by the resolution carry trivial
	pullbacks of $E$ and contribute non-strict inequalities.
	
	Taking $m_0\geq\max_i a_i(a_i-1)$ therefore gives a stable smooth
	member for every prescribed multidegree.
	Openness of stability and irreducibility of the parameter space
	give the assertion for a general complete intersection.
\end{proof}

\medskip

We also have a strong restriction theorem for normal divisors passing through singular points. We prove here just the following simple version: 

\begin{proposition}	
		Let $X$, $H$, and $E$ be as in
	Theorem~\ref{thm:restriction-prescribed-points}, and assume that
	$n=\dim X\geq 3$.
	Let $x\in X$ be a closed point, possibly singular, such that	$\operatorname{depth}\mathcal O_{X,x}\geq 3.$
	Then, for every sufficiently large integer $m$, a general divisor  $D\in |mH|$ passing through $x$
	is integral and normal, and the restriction $E|_D$
	is torsion-free and slope   $H|_D$-stable.
\end{proposition}

\begin{proof} Bertini's theorem gives integrality and normality of $D$ away from $x$.
Since
	\(	\operatorname{depth}\mathcal O_{D,x}	=\operatorname{depth}\mathcal O_{X,x}-1\geq 2\)
Serre's criterion implies that $D$ is normal.
	
	Choose a projective resolution $\pi:Y\to X$ which principalizes
	$\mathcal I_x$ and is an isomorphism over
	$X_{\mathrm{reg}}\setminus\{x\}$, and write
	\(
	\mathcal I_x\mathcal O_Y=\mathcal O_Y(-Z).\)
	The projection formula shows that $\widetilde{E}=(\pi^*E)^{**}$
	is slope-stable with respect to the nef divisor $A=\pi^*H$.
	For $m\gg0$, the pulled-back system, after removing its fixed part $Z$,
	is base-point-free. Thus the strict transform $T$ of a general $D$
	is a smooth integral divisor satisfying
	\(
	T\sim mA-Z.	\)
	
	Since $n\geq3$ and $Z$ is supported over $x$, for every Cartier divisor $L$ on $Y$ one has
	\[
	Z\cdot L\cdot A^{n-2}=0.
	\]
	Consequently, setting $d=A^n=H^n$, we obtain
	\(
	T\cdot A^{n-1}=md 
	\) and
$	T^2\cdot A^{n-2}=m^2d.$
	These are precisely the intersection identities involving the
	restricting divisor used in the elementary-transform proof of
	\cite[Theorem~5.2]{Langer2004}.
	That proof therefore applies with the fixed nef polarization
	$(A,\ldots,A)$ and gives stability of
	$\widetilde{E}|_T$ whenever $r\geq2$
	and
	\[
	m>
	\frac{r-1}{r}\,
	\Delta(\widetilde{E}) A^{n-2}
	+\frac{1}{d\,r(r-1)}.
	\]
As before both 	$\widetilde{E}|_T$ and $E|_D$ are torsion-free.
	Moreover, $T\to D$ is an isomorphism over a big open subset of $D$,
	where the two restricted sheaves agree.
	The projection-formula argument from the preceding proof therefore
	descends stability to $E|_D$.
\end{proof}

\section{Push-forward formulas for flag varieties}

The following proposition is a corollary of the results of Darondeau and Pragacz \cite{ssb-DarondeauPragacz2017}.

\begin{proposition}\label{ssb:prop-top-intersection}
	Let $E$ be a vector bundle of rank $r$ on a smooth complex
	algebraic variety $X$. Let $\pi_i:Y_i\to X$ be either:
	\begin{enumerate}
		\item the Grassmann bundle of rank $i$ quotients of $E$,
		where $1\leq i<r$; or
		\item if $r=2n$ and $E$ carries a nondegenerate alternating
		form with values in a line bundle $N$, the bundle of
		rank $i$ quotients whose duals are isotropic in $E^*$,
		where $1\leq i\leq n$.
	\end{enumerate}
	Let $Q_i$ be the universal quotient and put
	\(
	L_i=\det Q_i,\)
	\(\widehat\xi_i=c_1(L_i)-i\pi_i^*\frac{c_1(E)}{r},\)
	\(d_i=\dim(Y_i/X)\) and 
	\(	v_i=\int_{(Y_i)_x}c_1((L_i)_x)^{d_i}.\)
	Then
	\begin{equation}\label{ssb:eq-top-intersection}
		(\pi_i)_*\widehat\xi_i^{d_i+2}
		=-\kappa_i\Delta(E)
	\end{equation}
	in $A^2(X)_{\mathbb Q}$,	where
	\begin{equation*} 
		\kappa_i=\frac{i(d_i+1)(d_i+2)}{2r^2(r+1)}\,v_i 
		\begin{cases}
			\displaystyle\frac{r-i}{r-1},
			&\text{in the ordinary case},\\[6pt]
			\displaystyle 1,
			&\text{in the symplectic case}.
		\end{cases}
	\end{equation*}
	In particular, $\kappa_i>0$.
\end{proposition}

\begin{proof}
	Write $d=d_i$, $v=v_i$ and $\pi=\pi_i$.	The Gysin formulas of
	\cite[Theorems~1.1 and~2.1]{ssb-DarondeauPragacz2017}
	express $\pi_*\widehat\xi_i^{d+2}$ as a universal polynomial
	in Chern classes. Tensoring $E$ by a line bundle preserves
	$\widehat\xi_i$; in the symplectic case one simultaneously
	tensors $N$ by its square. We may therefore formally replace
	the Chern roots $x_j$ by $x_j-c_1(E)/r$.
	Since $c_1(E)=rc_1(N)/2$ in the symplectic case, this also
	sets the form class equal to zero. The centered Segre series
	of $E^*$ is
	\[
	\prod_{j=1}^r
	\left(1-\left(x_j-\frac{c_1(E)}r\right)t\right)^{-1}
	=1-\frac{\Delta(E)}{2r} t^2+O(t^3).
	\]
	
	Set $q=1$ in the ordinary case and $q=2$ in the symplectic
	case, and write
	\[
	S=\sum_{a=1}^i t_a,\qquad
	M=\prod_{a=1}^i t_a^{r-a},\qquad
	V_q=\prod_{a<b}(t_a^q-t_b^q)
	=\det \left[t_a^{q(i-b)}\right]_{a,b=1,...,i}.
	\]
	Applying the cited formulas to $E^*$, to match our quotient
	convention, gives
	\[
	v=[M](S^dV_q),\qquad
	\pi_*\widehat\xi_i^{d+2}=-\frac{\Delta(E)}{2r}\, C,\qquad
	C=[M]\left(S^{d+2}V_q\sum_{a=1}^i t_a^{-2}\right),
	\]
	where $[M](w)$ denotes the coefficient of the monomial $M$ in the Laurent polynomial $w$.
	
	To evaluate these coefficients, let $z_1,\ldots,z_i$ be
	nonnegative integers. Factoring $1/z_b!$ out of column $b$
	of the determinant below leaves the entry
	$z_b(z_b-1)\cdots(z_b-a+2)$, with the empty product equal to $1$.
	These are the evaluations at $z_b$ of monic polynomials of
	degrees $a-1$. The Vandermonde determinant therefore gives
	\[
	f(z_1,...,z_i):=\det\left[\frac1{(z_b-a+1)!}\right]_{a,b=1,...,i}
	=\frac{\prod_{a<b}(z_b-z_a)}{\prod_b z_b!},
	\]
	with $1/j!=0$ for $j<0$.
	
	Now set $z_b=r-1-q(i-b)$. We claim that
	\[
	v=d!\, f(z_1,\ldots, z_i),\qquad
	C=(d+2)!\, \sum_{b=1}^i f(z_1,\ldots,z_{b-1},z_b+2,z_{b+1},\ldots,z_i).
	\]
	We have
	\[
	d=i(r-i)-(q-1)\binom{i}{2}
	=\sum_{a=1}^i(r-a)-q\binom{i}{2}.
	\]
	Expanding the determinant gives
	\[
	V_q=\sum_{\sigma\in\mathfrak S_i}
	\operatorname{sgn}(\sigma)
	\prod_{a=1}^i t_a^{q(i-\sigma(a))}.
	\]
	For each permutation $\sigma$, put
	\[
	k_a(\sigma):=r-a-q(i-\sigma(a))
	=z_{\sigma(a)}-a+1.
	\]
	Then $\sum_a k_a(\sigma)=d$. Hence the multinomial theorem gives
	\[
	v=\sum_{\sigma\in\mathfrak S_i}
	\operatorname{sgn}(\sigma)
	[t_1^{k_1(\sigma)}\cdots t_i^{k_i(\sigma)}]S^d
	=d!\sum_{\sigma\in\mathfrak S_i}
	\operatorname{sgn}(\sigma)
	\prod_{a=1}^i\frac1{(z_{\sigma(a)}-a+1)!}
	=d!\,f(z_1,\ldots,z_i).
	\]
	
	Similarly, multiplication by $t_c^{-2}$ means that the
	exponents required from $S^{d+2}$ are
	$k_a(\sigma)+2\delta_{ac}$, where $\delta$ denotes the
	Kronecker delta. Their sum is $d+2$, so
	\[
	\frac{C}{(d+2)!}
	=\sum_{\sigma\in\mathfrak S_i}\operatorname{sgn}(\sigma)
	\sum_{c=1}^i
	\prod_{a=1}^i
	\frac1{(z_{\sigma(a)}-a+1+2\delta_{ac})!}.
	\]
	For fixed $\sigma$, reindex the inner sum by $b=\sigma(c)$.
	Since $\delta_{ac}=\delta_{\sigma(a),b}$, interchanging the
	two finite sums yields
	\[
	\begin{aligned}
		\frac{C}{(d+2)!}
		&=\sum_{b=1}^i
		\sum_{\sigma\in\mathfrak S_i}\operatorname{sgn}(\sigma)
		\prod_{a=1}^i
		\frac1{(z_{\sigma(a)}+2\delta_{\sigma(a),b}-a+1)!}\\
		&=\sum_{b=1}^i
		\det\left[
		\frac1{(z_j+2\delta_{jb}-a+1)!}
		\right]_{a,j=1,...,i}
		=\sum_{b=1}^i
		f(z_1,\ldots,z_{b-1},z_b+2,z_{b+1},\ldots,z_i).
	\end{aligned}
	\] 
	
	For $q=1$, only $b=i$ and, if $i>1$, $b=i-1$ contribute.
	The sum of their ratios to $f(z)$ is
	\[
	\frac{i(i+1)}{2r(r+1)}
	-\frac{i(i-1)}{2r(r-1)}
	=\frac{i(r-i)}{r(r^2-1)}.
	\]
	For $q=2$, only $b=i$ contributes, with ratio
	\[
	\frac1{r(r+1)}
	\prod_{h=1}^{i-1}\frac{h+1}{h}
	=\frac{i}{r(r+1)}.
	\]
	Consequently,
	\[
	C=v(d+1)(d+2)
	\begin{cases}
		\displaystyle\frac{i(r-i)}{r(r^2-1)},&q=1,\\[6pt]
		\displaystyle\frac{i}{r(r+1)},&q=2,
	\end{cases}
	\]
	which implies the required formula (\ref{ssb:eq-top-intersection}).
	Finally, $v>0$ because $L_i$ is
	ample on the fibers. This proves that $\kappa_i>0$.
\end{proof}

\section*{AI Declaration}

During the development and preparation of this article, the author used ChatGPT (OpenAI; GPT-6 Astra) extensively as an interactive tool to explore proof strategies, discuss mathematical arguments,  and assist with drafting and revision. In June 2026 Gemini Pro (Google) was also consulted in connection with the examples concerning étale-trivializable reductions
(see Proposition \ref{prop:prime-to-p-line-bundle}). Its suggestions included also an incorrect construction of  a counterexample satisfying Theorem \ref{thm:stable-rank-two-counterexample}.

The author independently checked the mathematical arguments and cited sources, revised the text, and takes full responsibility for the results and the final manuscript.

\section*{Acknowledgements}

The author would like to thank Debam Biswas, who in June 2026 informed him about the line bundle example from Proposition \ref{prop:prime-to-p-line-bundle}. 

The author was partially supported by the National Science Centre, Poland, contract number 2025/59/B/ST1/02168.

\bibliographystyle{amsalpha}
\bibliography{References}

\end{document}